\documentclass{amsart}

\calclayout
\usepackage[english]{babel}
 \usepackage{tikz}
\usepackage[utf8x]{inputenc}
\usepackage[T1]{fontenc}

\usepackage[a4paper,top=3cm,bottom=2cm,left=3cm,right=3cm,marginparwidth=1.75cm]{geometry}
\usepackage{amsmath,amssymb,amsthm}
\usepackage{graphicx}
\usepackage[colorlinks=true, allcolors=blue]{hyperref}

\theoremstyle{plain}                    
\newtheorem{teo}{Theorem}[section]      
\newtheorem{prop}[teo]{Proposition}    
\newtheorem{cor}[teo]{Corollary}       
\newtheorem{lem}[teo]{Lemma}    
\theoremstyle{definition}               
\newtheorem{defin}{Definition}[section]
\newtheorem{ese}{Example}[section]      
\theoremstyle{remark}                   
\newtheorem{oss}{Remark}[section]       

\newcommand{\longdownarrow}{\lower 1.4ex\hbox{\begin{picture}(18,18)(0,0)
\thicklines
\put(0,18){\vector(0,-1){18}}
\end{picture}}}
\newcommand{\longsearrow}{\lower 1.4ex\hbox{\begin{picture}(18,18)(0,0)
\thicklines
\put(0,18){\vector(1,-1){18}}
\end{picture}}}
\newcommand{\longeearrow}{\lower 1.4ex\hbox{\begin{picture}(0,0)(9,9)
\thicklines
\put(0,18){\vector(1,0){18}}
\end{picture}}}

\title{The Rumin complex on nilpotent Lie groups}
\author{Francesca Tripaldi}

\begin{document}
\begin{abstract}
   In this paper an alternative definition of the Rumin complex $(E_0^\bullet,d_c)$ is presented, one that relies on a different concept of weights of forms. In this way, the Rumin complex can be constructed on any nilpotent Lie group equipped with a Carnot-Carath\'eodory metric. Moreover, this construction allows for the direct application of previous non-vanishing results of $\ell^{q,p}$ cohomology to all nilpotent Lie groups that admit a positive grading.
\end{abstract}
\maketitle

\tableofcontents

\section{Introduction}

SubRiemannian structures have been largely studied in several aspects of pure as well as applied mathematics, such as geometric measure theory, subelliptic differential equations, differential geometry, complex variables, optimal control theory, mathematical models in neurosciences, and robotics. 
Roughly speaking, a subRiemannian structure on a manifold $M$ is defined by a subbundle $H$ of the tangent bundle $TM$, that defines the ``admissible'' directions at any point $p$ on $M$. Usually $H$ is referred to as the \textit{horizontal} bundle. If we endow each fibre $H_p$ of $H$ with a scalar product, there is a naturally associated Carnot-Carath\'eodory (CC) distance $d_{CC}$ on $M$, defined as the Riemannian length of the horizontal curves on $M$, that is those curves $\gamma(t)$ such that $\gamma'(t)\in H_{\gamma(t)}$. 

One of the most studied categories among subRiemannian spaces is that of Carnot groups, that is stratifiable nilpotent Lie groups with a CC-distance. Their characteristic group dilations are strictly linked to the stratification of their Lie algebra, and endow them with an intrinsic notion of homogeneity.
Carnot groups appear in several mathematical contexts, such as in harmonic analysis, in the study of hypoelliptic differential operators, and as boundaries of strictly pseudo-convex complex domains, see the books \cite{Stein:book, Capogna-et-al} as initial references.

In particular, the study of hypoelliptic Laplacians on differential forms in Carnot groups has flourished in recent years. The classical approach cannot be directly applied to the usual de Rham complex $(\Omega^\bullet,d)$ since in general it is not homogeneous with respect to the group dilations. This problem can be avoided by considering a homotopically equivalent subcomplex $(E_0^\bullet,d_c)$ known as the Rumin complex, which better reads the homogeneity of the underlying Carnot group. We refer to \cite{BFP1,pansu2019averages,BFP2,BFP3,BFP4,pansu-rumin} for the main results relating to hypoelliptic Laplacians on Carnot groups.

In this paper we extend the construction of such subcomplex $(E_0^\bullet,d_c)$ to arbitrary nilpotent Lie groups (not necessarily stratifiable) equipped with a CC-metric.

\section{Notation}

Throughout the paper we will denote as $G$ an arbitrary nilpotent Lie group of dimension $n$, and by $\mathfrak{g}$ its Lie algebra. We will introduce an inner product $\langle\cdot,\cdot\rangle$ on $\mathfrak{g}$ and we will denote by $\mathfrak{g}^\ast$ the dual space of $\mathfrak{g}$, the vector space of linear functionals on the elements of $\mathfrak{g}$. We will use both the following notations to denote the duality product, that is the action of a linear functional $\theta\in\mathfrak{g}^\ast$ on an element $X\in\mathfrak{g}$:
\begin{align*}
    \langle\theta\mid X\rangle:=\theta(X)\,.
\end{align*}

Given $\lbrace X_1,\ldots,X_n\rbrace$ an orthonormal basis for $\mathfrak{g}$, we can consider a particular basis $\lbrace\theta_1,\ldots,\theta_n\rbrace$ of $\mathfrak{g}^\ast$, the so-called dual basis for which
\begin{align*}
    \theta_i(X_j)=\langle \theta_i\mid X_j\rangle=\delta_{ij}\;,\;\forall\,i,j=1,\ldots,n\,.
\end{align*}
We will denote the dual correspondence as $\theta_i=X_i^\ast$. Moreover, we can introduce an inner product, also denoted as $\langle\cdot,\cdot\rangle$ on $\mathfrak{g}^\ast$ such that the dual
basis above will be an orthonormal basis.

Given a $k$-dimensional vector space $V$ with basis $\lbrace v_1,\ldots,v_k\rbrace$, we will denote by $V^\ast$ the dual space of $V$, and the exterior algebras of $V$ and $V^\ast$ will be given respectively as
\begin{align*}
    \Lambda^\bullet V=\bigoplus_{h=1}^k\Lambda^hV\;\text{ and } \Lambda^\bullet V^\ast=\bigoplus_{h=1}^k\Lambda^hV^\ast\,,
\end{align*}
where
\begin{align*}
    \Lambda^hV=&span\lbrace v_{i_1}\wedge\cdots\wedge v_{i_h}\mid 1\le i_1<\cdots<i_h\le k\rbrace\,,\\
    \Lambda^hV^\ast=&span\lbrace v^\ast_{i_1}\wedge\cdots\wedge v_{j_h}^\ast\mid 1\le j_1<\cdots<j_h\le k\rbrace\,.
\end{align*}
Moreover, the inner products on $\mathfrak{g}$ and $\mathfrak{g}^\ast$ extend canonically to $\Lambda^h\mathfrak{g}$ and $\Lambda^h\mathfrak{g}^\ast$ for $1\le h\le n$, making their bases orthonormal too, and they will both be denoted as $\langle\cdot,\cdot\rangle$.

Let us briefly recall the terminology for stratifications and gradings for a nilpotent Lie group $G$.
Given two subspaces $V, W$ of a Lie algebra $\mathfrak{g}$, we set
$
[V,W] :=  span\{[X,Y]\mid  X\in V,\ Y\in W\} .
$
 A \textit{stratification of step $s$} of the Lie algebra $\mathfrak{g}$ is a direct-sum decomposition
$$\mathfrak{g} = V_1 \oplus V_2 \oplus \cdots \oplus V_s
  \label{def:stratifiable-algebras}
  $$
for some integer $s\geq 1$, where $[V_1,V_j] = V_{j+1}$ for all integers $j\in \{1,\dots,s\}$, $V_s \not= \{0\}$ and $V_{s+1} = \{0\}$. 
We say that a Lie algebra is \textit{stratifiable} if there exists a stratification of it.

Nilpotent Lie groups $G$ whose Lie algebra $\mathfrak{g}$ is stratifiable are also called {\em Carnot groups}.
 We should stress that in the case of a stratifiable algebra, the choice of a stratification is essentially unique, that is every two stratifications of $\mathfrak{g}$ differ by a Lie algebra automorphism of $\mathfrak{g}$ (for a reference see \cite[Proposition~2.17]{ledonne_primer}).

A  stratification is a particular example of a positive grading. Indeed, it is a positive grading where the first layer $V_1$ is Lie generating.
  A \emph{positive grading} of a Lie algebra $\mathfrak{g}$ is a
family $(V_t)_{t\in (0,+\infty)}$ of linear subspaces of $\mathfrak{g}$, where all but finitely many of the $V_t$s are $\{0\}$, and such that
$$\mathfrak{g}=  \bigoplus_{t\in (0,+\infty)} V_t \;\;,\;\;\text{ with }[V_t, V_u]\subset V_{t+u} \;  \text{ for all } t,u >0.$$
We say that a Lie algebra $\mathfrak{g}$ is \emph{positively gradable} if there exists a positive grading of it. 

Finally, we say that a Lie algebra is \textit{non-gradable} (or more precisely \textit{non-positively gradable}) if it does not admit any positive gradings.

One should notice that up to dimension 4 all nilpotent Lie algebras are stratifiable, whereas the first examples of non-gradable Lie algebras already appear in dimension 7 (see \cite{finnish,finnish2} for a reference).

\section{An overview of the previous constructions of the Rumin complex}

The starting point for our considerations is the hope to construct a complex of \textit{intrinsic forms} on all nilpotent Lie groups that mirrors the one constructed by Rumin for Carnot groups. 

The Rumin complex was first developed in \cite{rumin1990complexe,rumin1994} as a better choice than the usual de Rham complex in the case of Heisenberg groups $\mathbb{H}^{2n+1}$, especially regarding the homogeneity with respect to the group dilations.

In this first construction on Heisenberg groups, the intrinsic forms were defined by using two  differential ideals:
\begin{itemize}
    \item $\mathcal{I}^\bullet=\lbrace \gamma_1\wedge\theta+\gamma_2\wedge d\theta\mid \gamma_1,\gamma_2\in\Omega^\bullet\rbrace$ the differential ideal generated by the contact form $\theta$, and
    \item $\mathcal{J}^\bullet=\lbrace \beta\in\Omega^\bullet\mid \beta\wedge\theta=\beta\wedge d\theta=0\rbrace$ the annihilator of $\mathcal{I}^\bullet$.
\end{itemize} 

Moreover, since we are working on a contact manifold, we can also consider the map
\begin{align*}
    L\colon \Lambda^k Q^\ast&\to\Lambda^{k+2}Q^\ast\\
    \alpha\;&\mapsto d\theta\wedge \alpha\,,
\end{align*}
where $Q\subset T\mathbb{H}^{2n+1}$ is a subbundle by hyperplanes such that $Ker\,\theta=Q$ and $d\theta\vert_Q$ does not vanish. By exploiting the results on the map $L$ that stem from K\"ahler geometry, one can show that in $\mathbb{H}^{2n+1}$ one has
\begin{align*}
    \Omega^k/\mathcal{I}^k=0\text{ for }k\ge n+1\;,\text{ and }\mathcal{J}^k=0\text{ for }k\le n\,.
\end{align*}

One can then consider the two complexes $(\Omega^\bullet/\mathcal{I}^\bullet,d_Q)$ and $(\mathcal{J}^\bullet,d_Q)$ on $\mathbb{H}^{2n+1}$, where the differential $d_Q$ is simply the usual exterior differential $d$ that descends to the quotients $\Omega^\bullet/\mathcal{I}^\bullet$, and restricts to the subspaces $\mathcal{J}^\bullet$ respectively:
\begin{align*}
    &\Omega^0/\mathcal{I}^0&\xrightarrow[]{d_Q}&\Omega^1/\mathcal{I}^1&\xrightarrow[]{d_Q}&\cdots&\xrightarrow[]{d_Q}&\Omega^{n}/\mathcal{I}^n&\xrightarrow[]{d_Q}&\quad \,0&\xrightarrow[]{d_Q}&\quad \,0&\xrightarrow[]{d_Q}&\cdots&\xrightarrow[]{d_Q}&\quad \,0\\
    & \quad \,0&\xrightarrow[]{d_Q}&\quad \,0&\xrightarrow[]{d_Q}&\cdots&\xrightarrow[]{d_Q}&\quad \,0&\xrightarrow[]{d_Q}&\mathcal{J}^{n+1}&\xrightarrow[]{d_Q}&\mathcal{J}^{n+2}&\xrightarrow[]{d_Q}&\cdots&\xrightarrow[]{d_Q}&\mathcal{J}^{2n+1}\,.
\end{align*}

By exploiting again the properties of the application $L$, one can further construct a second order differential operator, which Rumin denotes as $D$, that links the non-trivial part of the complex $(\Omega^\bullet/\mathcal{I}^\bullet,d_Q)$ to the non-trivial part of the complex $(\mathcal{J}^\bullet,d_Q)$, hence obtaining a new complex of \textit{intrinsic forms} that has the same cohomology as the de Rham complex $(\Omega^\bullet,d)$:
\begin{align}\label{Heisenberg complex}
\Omega^0/\mathcal{I}^0\xrightarrow{d_Q}\Omega^1/\mathcal{I}^1\xrightarrow{d_Q}\cdots\xrightarrow{d_Q}\Omega^n/\mathcal{I}^n\xrightarrow{D}\mathcal{J}^{n+1}\xrightarrow{d_Q}\mathcal{J}^{n+2}\xrightarrow{d_Q}\cdots\xrightarrow{d_Q}\mathcal{J}^{2n+1}\,.
\end{align}

Since these differential ideals create a filtration on smooth forms $0\subset \mathcal{J}^\bullet\subset\mathcal{I}^\bullet\subset\Omega^\bullet$ which is stable under $d$, that is $d(\mathcal{J}^\bullet)\subset\mathcal{J}^\bullet$ and $d(\mathcal{I}^\bullet)\subset\mathcal{I}^\bullet$, one can apply the machinery of spectral sequences to this filtration. Julg in \cite{julg1995complexe} studies, among other things, this spectral sequence on $\mathbb{H}^{2n+1}$ and shows how the non-trivial differentials $d_0$ on the zeroth page quotients $E_0^{k,1}=\Omega^{k}/\mathcal{I}^{k}$ for $k< n$, and $E_0^{k,2}=J^{k}$ for $k>n$ coincide with the first order operators $d_Q$. Moreover, the only non-trivial differential on the second page quotients $d_2\colon E_2^{n,1}\to E_2^{n+1,2}$ coincides with the second order differential operator $D$.

Let us stress that spectral sequences can be used on any filtration of smooth forms that is stable under $d$, in order to compute the (graded part of the) de Rham cohomology of the manifold. 

Since the filtration by differential ideals $\mathcal{I}^\bullet$ and $\mathcal{J}^\bullet$ can only be extended to 2-step nilpotent Lie groups (see \cite{julg1995complexe} for a generalisation of this construction), Rumin introduces a new filtration by weights, still stable under $d$, which can be applied to all Carnot groups $G$ (see \cite{rumin2000around,rumin2000sub,rumin2005Palermo}). 

It should be noticed that in Rumin's construction the concept of weight is strictly linked to the positive grading given by the algebra's stratification
\begin{align*}
    \mathfrak{g}=V_1\oplus\cdots\oplus V_s\,.
\end{align*}
For $\lambda>0$, the dilation of factor $\lambda$ on $\mathfrak{g}$ relative to the associated stratification, is the unique linear map $\delta_{\lambda}\colon \mathfrak{g}\to\mathfrak{g}$ such that $\delta_\lambda(X)=\lambda^t\,X$ for any $X\in V_t$. Rumin then defines the weight function for vectors and covectors as
\begin{align*}
    w(X)=t\Longleftrightarrow X\in V_t\subset\mathfrak{g}\text{ and }w(\theta)=t\Longleftrightarrow \theta\in V_t^\ast\subset\mathfrak{g}^\ast\,,
\end{align*}
and extends it to the whole
differential algebra of $G$ as follows:
\begin{align*}
    w(\theta_1\wedge\cdots\wedge\theta_h)=p\Longleftrightarrow \sum_{i=1}^hw(\theta_i)=p\text{, for any  }\theta_1\wedge\cdots\wedge\theta_h\in\Lambda^h\mathfrak{g}^\ast\,.
\end{align*}

Once the weight function is defined on smooth forms $\Omega^\bullet$ as well, one obtains a decreasing filtration by the spaces $\mathcal{F}^p$ of forms of weight $\ge p$
\begin{align*}
    \mathcal{F}^{Q+1}=0\subset\mathcal{F}^Q\subset\mathcal{F}^{Q-1}\subset\cdots\subset\mathcal{F}^2\subset\mathcal{F}^1\subset\mathcal{F}^0=\Omega^\bullet\,,
\end{align*}
where $Q$ is the homogeneous dimension of the Carnot group $G$. By the definition of the exterior differential $d$, it is easy to check that this filtration is also stable under $d$, that is $d(\mathcal{F}^p)\subset\mathcal{F}^p$.

Following Rumin's notation in \cite{rumin2000around}, the operator $d_0$ denotes the exterior differential $d$ acting on the quotients $\mathcal{F}^p/\mathcal{F}^{p+1}$, and $E_0^\bullet=Ker\,d_0/Im\,d_0$ denotes its cohomology. In the usual notation of spectral sequences we would have had instead $E_0^{\bullet,p}=\mathcal{F}^{p}/\mathcal{F}^{p+1}$ at the zeroth page, and $E_1^{\bullet,p}=Ker\,d_0/Im\,d_0$ at the first page. Moreover, as pointed out by Rumin, $E_0^\bullet=Ker\,d_0/Im\,d_0$ is the bundle whose fibre is the
Lie algebra cohomology of the Carnot group $G$.

When applied to Heisenberg groups, the non-trivial first page quotients $E_0^\bullet$ coincide with the \textit{intrinsic forms} of the complex \eqref{Heisenberg complex}, the first page differentials $d_1\colon E_0^k\to E_0^{k+1}$ with $k\neq n$ coincide with the first order differential operators $d_Q$, and the only non-trivial second page differential $d_2\colon E_0^{n}\to E_0^{n+1}$ coincides with the second order differential operator $D$.

Unfortunately, considering this weight filtration for Carnot groups of arbitrary nilpotency step $s$ is not only complicated to handle analytically, but it also makes it impossible to think of the \textit{intrinsic forms} $E_0^\bullet=Ker\,d_0/Im\,d_0$ as a subspace of $\Omega^\bullet$.

For this reason, in \cite{rumin2005Palermo,rumin2000around} Rumin defines what is now known as the \textit{Rumin complex} $(E_0^\bullet,d_c)$ for an arbitrary Carnot group $G$. In this construction he first introduces a metric on $G$, and maintains the same notation used for the weight spectral sequence to denote the \textit{subspace} $E_0^\bullet=Ker\,d_0\cap(Im\,d_0)^\perp$ of \textit{intrinsic (Rumin) forms}. In this case the operator $d_0$, defined as the part of the differential $d$ that keeps the weight between forms constant, corresponds to the previous operator $d_0$ induced on the zeroth page of the weight spectral sequence, but it is now acting between subspaces of forms on $G$ (and not between quotients). After defining the homotopical equivalence $\Pi_E$ between $(\Omega^\bullet,d)$ and the subcomplex $(E^\bullet,d)$, and the orthogonal projection $\Pi_{E_0}$, we obtain the exact subcomplex $(E_0^\bullet,d_c=\Pi_{E_0}d\Pi_E)$ which is conjugated to $(E^\bullet,d)$. Crucial in this construction is the introduction of the operator $d_0^{-1}$, the inverse map of $d_0$, which can only be defined once we have a metric on $G$.



What was then left as an open question was the possibility of producing a Rumin complex on arbitrary nilpotent Lie groups, and in particular those that do not admit a stratification. This is precisely the aim of this paper (as for the adaptation of Rumin's construction to positively gradable nilpotent Lie groups, this is mentioned in Section 2.2.1 of \cite{rumin2000around}).
To avoid having to rely on gradings in our construction, we will introduce an alternative definition of weights of 1-forms, one that only depends on the lower central series of the Lie algebra $\mathfrak{g}$. We will also see that in the case of Carnot groups the two definitions of weights coincide. Finally, in order to construct the spaces $E_0^\bullet$ of \textit{intrinsic forms}  we will not use $d_0$, defined as the operator that keeps the weight of forms constant, but instead we will define $E_0^\bullet$ as $Ker\,d_\mathfrak{g}\cap(Im\,d_\mathfrak{g})^\perp$, where $d_\mathfrak{g}$ is simply the algebraic part of the exterior differential $d$. In this way we also want to stress the correspondence between $E_0^\bullet$ and the Lie algebra cohomology of the nilpotent Lie group $G$.

\section{Filtration based on the lower central series}

Let $\mathfrak{g}$ be a nilpotent Lie algebra of dimension $n$.
The lower central series $\lbrace\mathfrak{g}^{(i)}\rbrace$ is defined as follows:
\begin{align*}
    \mathfrak{g}^{(0)}=\mathfrak{g}\;\text{ and }\;\mathfrak{g}^{(i)}=[\mathfrak{g},\mathfrak{g}^{(i-1)}]\;,\;i\ge 1\,.
\end{align*}

A nilpotent Lie algebra $\mathfrak{g}$ has nilpotency step $s$ if $\mathfrak{g}^{(s)}=0$ and $\mathfrak{g}^{(s-1)}\neq 0$
. One should notice that for every nilpotent Lie algebra $\mathfrak{g}$ there exists $s\in\mathbb{N}$ such that $\mathfrak{g}^{(s)}=0$, that is the nilpotency step is finite and the lower central series stops after a finite number of consecutive brackets.

In the case of a Lie group $G$, one can consider the subcomplex of the deRham complex $(\Omega^\bullet,d)$ consisting of the left-invariant differential forms. A left-invariant $k$-form is uniquely determined by its value at the identity, where it defines a linear map $\Lambda^k\mathfrak{g}\to\mathbb{R}$, by identifying the tangent space at the identity with the Lie algebra $\mathfrak{g}$. In other words, a left-invariant $k$-form is an element of $\Lambda^k\mathfrak{g}^\ast$. The exterior derivative then induces a linear map on $\Lambda^k\mathfrak{g}^\ast$ which we will denote as $d_\mathfrak{g}$:
\begin{align*}
    d\big\vert_{\Lambda^k\mathfrak{g}^\ast}=:d_\mathfrak{g}\colon\Lambda^k\mathfrak{g}^\ast\to\Lambda^{k+1}\mathfrak{g}^\ast\,.
\end{align*}

\begin{oss}
The exterior differential applied on left-invariant forms, denoted as $
{d}_\mathfrak{g}$ 
\begin{align*}
    d_\mathfrak{g}\colon\mathfrak{g}^\ast=\Lambda^1\mathfrak{g}^\ast\to\Lambda^2\mathfrak{g}^\ast
\end{align*}
coincides with the dual mapping of the Lie bracket
\begin{align*}
    [\;,\;]\colon\Lambda^2\mathfrak{g}\to\mathfrak{g}\,.
\end{align*}

Moreover, given two left-invariant forms $\theta_i\in\Lambda^i\mathfrak{g}^\ast$ and $\theta_j\in\Lambda^j\mathfrak{g}^\ast$, we have that the usual Leibniz rule applies
\begin{align*}
    d_\mathfrak{g}(\theta_i\wedge \theta_j)=d_\mathfrak{g}\theta_i\wedge \theta_j+(-1)^i\theta_i\wedge d_\mathfrak{g}\theta_j\,.
\end{align*}
\end{oss}

Let us consider the following subspaces of $\mathfrak{g}^\ast=\Lambda^1\mathfrak{g}^\ast$:
\begin{align*}
    F_0=0\;,\;F_i=\lbrace \theta_j\in\Lambda^1\mathfrak{g}^\ast\mid d_\mathfrak{g}\theta_j\in\Lambda^2F_{i-1}\rbrace\;,\;i\ge 1\,.
\end{align*}

Notice that $F_1$ is the subspace of left-invariant closed 1-forms. Furthermore, this sequence $F_i$ of subspaces defines a filtration on $\mathfrak{g}^\ast$
\begin{align*}
    0=F_0\subset F_1\subset F_2\subset\cdots\subset F_{s-1}\subset F_{s}=\mathfrak{g}^\ast\,,
\end{align*}
where $s$ is the nilpotency step of $\mathfrak{g}$.

The spaces $F_i$ are dual to the spaces $\mathfrak{g}^{(i)}$ of the lower central series, that is for all $i\ge 0$
\begin{align*}
    F_i=\lbrace \theta_j\in\mathfrak{g}^\ast\mid \theta_j(X)=0\;,\;\forall X\in\mathfrak{g}^{(i)}\rbrace\,.
\end{align*}
\subsection{Filtration on smooth 1-forms}

All the considerations we have made for left-invariant 1-forms can be extended to smooth 1-forms, so one can achieve a filtration of smooth 1-forms too.

In fact, given $G$ a connected Lie group with Lie algebra $\mathfrak{g}$, we know we can identify the tangent space $T_xG$ to $G$ at $x\in G$ with $\mathfrak{g}$ by means of the isomorphism $dL_x$, where $L_x$ denotes the left-translation by $x$. For $\theta\in\Lambda^k\mathfrak{g}^\ast$ and $f\in C^\infty(G)$, we can regard $\theta\otimes f$ as an element of smooth $k$-forms $\Omega^k$ by $(\theta\otimes f)_x=f(x)(dL_x^{-1})\theta$. This gives rise to an isomorphism
\begin{align*}
    Hom_{\mathbb{R}}(\Lambda^k\mathfrak{g},C^\infty(G))\cong\Lambda^k\mathfrak{g}^\ast\otimes C^\infty(G)\to\Omega^k(G)\,.
\end{align*}

Now, given $k+1$ arbitrary left-invariant vector fields $X_0,\ldots,X_k\in\mathfrak{g}$, the formula for the exterior differential is given as follows
\begin{align*}
    d(\theta\otimes f)(X_0,\ldots,X_k)=&\sum_{i=0}^k(-1)^i\theta(X_0,\ldots,\hat{X}_i,\ldots,X_k)\cdot X_if\,+\\&+\sum_{0\le i<j\le k}(-1)^{i+j}\theta([X_i,X_j],X_0,\ldots,\hat{X}_i,\ldots,\hat{X}_j,\ldots,X_k)\cdot f\,.
\end{align*}

We will use the following notation 
\begin{align*}
    \Gamma(\Lambda^k\mathfrak{g}^\ast):=\Lambda^k\mathfrak{g}\otimes C^\infty(G)=\Omega^k
\end{align*}
to denote the space of $k$-forms with coefficients in $C^\infty(G)$. We will also use the same notation if we want to consider subspaces of left-invariant $k$-forms, for example
\begin{align*}
    \Gamma(F_i):=F_i\otimes C^\infty(G)
\end{align*}
where $F_i\subset\Lambda^1\mathfrak{g}^\ast$ is a subspace of left-invariant 1-forms.


Let us express an arbitrary smooth 1-form $\alpha\in\Gamma(\Lambda^1\mathfrak{g}^\ast)=\Omega^1$ in terms of the orthonormal basis $\lbrace \theta_j\rbrace_{1\le j\le n}$ of left-invariant 1-forms in $\Lambda^1\mathfrak{g}^\ast$:
\begin{align*}
    \alpha=\sum_{j=1}^nf_j\theta_j\in\Gamma(\Lambda^1\mathfrak{g}^\ast)\,.
\end{align*}
Then the exterior differential $d$ applied to $\alpha$ will have the following expression
\begin{align*}
    d\alpha=&d\big(\sum_{j=1}^nf_j\theta_j\big)=\sum_{j=1}^nd(f_j\theta_j)=\sum_{j=1}^n\big(df_j\wedge \theta_j+f_jd\theta_j\big)\\=&\sum_{j=1}^n\big(df_j\wedge \theta_j+f_j\,d_\mathfrak{g}\theta_j\big)=\sum_{j=1}^ndf_j\wedge\theta_j+\sum_{j=1}^nf_j\,d_\mathfrak{g}\theta_j\,.
\end{align*}

\begin{defin} \textbf{The algebraic operator $\mathbf{d_\mathfrak{g}}$}

Given an arbitrary smooth 1-form $\alpha=\sum_{j=1}^nf_j\theta_j$ with $\theta_j\in\Lambda^1\mathfrak{g}^\ast$, we will then denote by $d_\mathfrak{g}\alpha$ the ``left-invariant'' part of the smooth 2-form $d\alpha$, that is
\begin{align*}
    d_\mathfrak{g}\alpha:=\sum_{j=1}^nf_j\,d_\mathfrak{g}\theta_j\,.
\end{align*}
\end{defin}

Just like before, one can construct a filtration on $\Omega^1$ using the action of $d_\mathfrak{g}$ (or, equivalently, using the lower central series):
\begin{align*}
    \Gamma(F_0)=0\;,\;\Gamma(F_i)=\lbrace \alpha\in\Gamma(\mathfrak{g}^\ast)\mid d_\mathfrak{g}\alpha\in\Gamma(\Lambda^2F_{i-1})\rbrace\;,\;i\ge 1\,.
\end{align*}

Exactly in the same way, we have the filtration
\begin{align*}
    0=\Gamma(F_0)\subset \Gamma(F_1)\subset\Gamma(F_2)\subset\cdots\subset\Gamma(F_{s-1})\subset\Gamma(F_s)=\Gamma(\Lambda^1\mathfrak{g}^\ast)\,,
\end{align*}
where $s$ is the nilpotency step of $\mathfrak{g}$.

It is trivial to check that for $i\ge 1$ we have the following equality between subspaces of $\Omega^1$
\begin{align*}
    \lbrace \alpha\in\Omega^1\mid d_\mathfrak{g}\alpha\in F_{i-1}\otimes C^\infty(G)\rbrace=F_i\otimes C^\infty(G)\,,
\end{align*}
which justifies using the notation $\Gamma(F_i)$.

\section{Redefining weights}
Just like in the construction proposed by Rumin in \cite{rumin2000around}, we would like to introduce the concept of ``pure weights'' of 1-forms. This will then allow us to produce a quantitative split of the exterior differential $d$ according to the weight increase. In Rumin's case, the weight of a 1-form is strictly linked to the homogeneity stemming from the dilations of the Carnot group considered. Since we are considering arbitrary nilpotent groups which are not necessarily gradable, we will introduce a different notion of weights linked to the lower central series. We will see that the two definitions of weights coincide when considering Carnot groups.

\begin{defin}\label{Weights of 1-forms} \textbf{Asymptotic weights of left-invariant 1-forms}

Using the scalar product $\langle\cdot,\cdot\rangle$ on $\mathfrak{g}^\ast$, we can express the space of left-invariant 1-forms as a direct sum of subspaces using the filtration $\lbrace F_i\rbrace_{i=0}^s$:
\begin{align}
    \mathfrak{g}^\ast=W_1\oplus W_2\oplus\cdots\oplus W_s\,,
\end{align}
where $W_i=F_i\cap(F_{i-1})^\perp$ for $i\ge 1$.

We say that a left-invariant 1-form $\theta\in\mathfrak{g}^\ast$ has pure asymptotic weight $k$ if $\theta\in W_k$, and we write $w(\theta)=k$.
\end{defin}

Let us stress that if $s$ is the nilpotency step of $G$, then $F_j=F_{j+1}$ for any $j\ge s$, so that $W_j=0$ for $j>s$. Moreover, for $1\le j\le s$
\begin{align*}
    F_j=W_1\oplus\cdots\oplus W_j\,.
\end{align*}

\begin{oss}
The concept of asymptotic weights extends directly to any smooth 1-form $\xi\in\Omega^1$ as follows:
\begin{align*}
    w(\xi)=k\iff \alpha\in\Gamma(W_k)=W_k\otimes C^\infty(G)\,.
\end{align*}
\end{oss}

\begin{prop}
Let $G$ be a nilpotent Lie group of dimension $n$ and nilpotency step $s$.

If $G$ is stratifiable with stratification
\begin{align*}
    \mathfrak{g}=V_1\oplus V_2\oplus\cdots\oplus V_s
\end{align*}
then $V_j^\ast=W_j$, where $W_i=F_i\cap(F_{i-1})^\perp$ for $i\ge 1$.

\end{prop}

\begin{proof}
Let us consider an orthonormal basis of $\mathfrak{g}$ adapted to the lower central series so that the following direct sum decomposition holds:
\begin{align*}
    \mathfrak{g}=\big(\mathfrak{g}^{(0)}\cap(\mathfrak{g}^{(1)})^\perp\big)\oplus\big(\mathfrak{g}^{(1)}\cap(\mathfrak{g}^{(2)})^\perp\big)\oplus\cdots\oplus\underbrace{\big(\mathfrak{g}^{(s)}\cap(\mathfrak{g}^{(s+1)})^\perp\big)}_{\mathfrak{g}^{(s)}=0}=\bigoplus_{i=0}^{s-1}\mathfrak{g}^{(i)}\cap (\mathfrak{g}^{(i+1)})^\perp\,.
\end{align*}

Let us first prove that $\mathfrak{g}^{(0)}\cap(\mathfrak{g}^{(1)})^\perp=F_1^\ast=W_1^\ast$:
\begin{align*}
    X\in\mathfrak{g}^{(0)}\cap(\mathfrak{g}^{(1)})^\perp&\Longleftrightarrow \langle X,Y\rangle=0\;\forall\;Y\in\bigoplus_{i=1}^{s-1}\mathfrak{g}^{(i)}\cap (\mathfrak{g}^{(i+1)})^\perp=\mathfrak{g}^{(1)}\\
    &\Longleftrightarrow \langle X^\ast\mid Y\rangle=X^\ast(Y)=0\;\forall\;Y\in\mathfrak{g}^{(1)}\\
    &\Longleftrightarrow X^\ast\in F_1=W_1\,.
\end{align*}

Let us prove that $\big(\mathfrak{g}^{(0)}\cap(\mathfrak{g}^{(1)})^\perp\big)\oplus\big(\mathfrak{g}^{(1)}\cap(\mathfrak{g}^{(2)})^\perp\big)=F_2^\ast$:
\begin{align*}
    X\in\big(\mathfrak{g}^{(0)}\cap(\mathfrak{g}^{(1)})^\perp\big)\oplus\big(\mathfrak{g}^{(1)}\cap(\mathfrak{g}^{(2)})^\perp\big)&\Longleftrightarrow \langle X,Y\rangle=0\;\forall\;Y\in\bigoplus_{i=2}^{s-1}\mathfrak{g}^{(i)}\cap (\mathfrak{g}^{(i+1)})^\perp=\mathfrak{g}^{(2)}\\
    &\Longleftrightarrow \langle X^\ast\mid Y\rangle=X^\ast(Y)=0\;\forall\;Y\in\mathfrak{g}^{(2)}\\
    &\Longleftrightarrow X^\ast\in F_2\,.
\end{align*}

In general we have $\bigoplus_{i=0}^{k-1}\mathfrak{g}^{(i)}\cap(\mathfrak{g}^{(i+1)})^\perp=F_k^\ast$:
\begin{align*}
    X\in\bigoplus_{i=0}^{k-1}\mathfrak{g}^{(i)}\cap(\mathfrak{g}^{(i+1)})^\perp&\Longleftrightarrow \langle X,Y\rangle=0\;\forall\;Y\in\bigoplus_{i=k}^{s-1}\mathfrak{g}^{(i)}\cap (\mathfrak{g}^{(i+1)})^\perp=\mathfrak{g}^{(k)}\\
    &\Longleftrightarrow \langle X^\ast\mid Y\rangle=X^\ast(Y)=0\;\forall\;Y\in\mathfrak{g}^{(k)}\\
    &\Longleftrightarrow X^\ast\in F_k\,.
\end{align*}

Therefore, by definition of the subspaces $W_i$ of left-invariant 1-forms we have:
\begin{align*}
    W_k^\ast=&\big(F_k\cap(F_{k-1})^\perp\big)^\ast=F_k^\ast\cap(F_{k-1}^\ast)^\perp\\=&\bigoplus_{i=0}^{k-1}\mathfrak{g}^{(i)}\cap(\mathfrak{g}^{(i+1)})^\perp\cap\big(\bigoplus_{i=0}^{k-2}\mathfrak{g}^{(i)}\cap(\mathfrak{g}^{(i+1)})^\perp\big)^\perp=\mathfrak{g}^{(k-1)}\cap(\mathfrak{g}^{(k-2)})^\perp\,.
\end{align*}

The proposition then follows directly from the definition of a stratification for a Carnot group.
\end{proof}

\begin{cor}
The two definitions of pure weights of 1-forms in Rumin \cite{rumin2000around} and  asymptotic weights in Definition \ref{Weights of 1-forms} coincide in the case of stratifiable nilpotent Lie groups.
\end{cor}

\begin{oss}
By definition, the Lie algebra $\mathfrak{g}_\infty$ of the asymptotic cone $G_\infty$ of a non-stratifiable nilpotent group $G$ is given by
\begin{align*}
    \mathfrak{g}_\infty=\mathfrak{g}^{(0)}/\mathfrak{g}^{(1)}\oplus\mathfrak{g}^{(1)}/\mathfrak{g}^{(2)}\oplus\cdots\oplus\mathfrak{g}^{(s-1)}/\mathfrak{g}^{(s)}\,.
\end{align*}
The asymptotic cone is a Carnot group with stratification
\begin{align*}
    \mathfrak{g}_{\infty}=V_1^\infty\oplus\cdots\oplus V_s^\infty\,.
\end{align*}
From the proof of the previous proposition, it is clear that there is a correspondence between the decomposition $W_1\oplus\cdots\oplus W_s$ of left-invariant forms given by the lower central series of $G$, and the dual of the stratification $V_1^\infty\oplus\cdots\oplus V_s^\infty$ of $\mathfrak{g}_\infty$, hence the terminology \textit{asymptotic weights}.
\end{oss}


\begin{defin}\textbf{Asymptotic weights of arbitrary smooth forms.}

In general, given a left-invariant $h$-form $\alpha\in\Lambda^h\mathfrak{g}^\ast$, we will say that it has pure asymptotic weight $p$ if it can be expressed as a linear combination of left-invariant $h$-forms $\theta_{i_1}\wedge\cdots\wedge\theta_{i_h}$ such that $w(\theta_{i_1})+\cdots+w(\theta_{i_h})=p$.

Just like in Definition \ref{Weights of 1-forms}, we can extend the concept of asymptotic weights to any smooth $h$-form $\xi\in\Omega^h$.

\end{defin}



\begin{lem}
Let $\alpha\in\Lambda^h\mathfrak{g}^\ast$ be a left-invariant differential form, then $d_{\mathfrak{g}}\alpha=d\alpha$ is still a left-invariant form.
\end{lem}
\begin{proof}
For any $x\in G$ one has $L_x^\ast(d_\mathfrak{g}\alpha)=L_x^\ast(d\alpha)=d(L_x^\ast\alpha)=d(\alpha)=d_\mathfrak{g}\alpha$.
\end{proof}

\begin{prop}\label{dg does not increase}
Let $\alpha\in\Lambda^1\mathfrak{g}^\ast$  be a left-invariant 1-form, then $d_{\mathfrak{g}}\alpha=\beta_1+\cdots+\beta_{l_\alpha}$, where for each $i=1,\ldots,l_{\alpha}$ we have $w(\beta_i)\le w(\alpha)$ .
\end{prop}
\begin{proof}
Let us consider $\lbrace X_{k}\rbrace_{1\le k\le n}$ an orthonormal basis of $\Lambda^1\mathfrak{g}$ adapted to the lower central series, and let $X$ be a left-invariant vector field such that $X=\alpha^\ast$ with $\alpha\in\Lambda^1\mathfrak{g}^\ast$, then



\begin{align*}
    w(\alpha)=j\Longleftrightarrow \alpha\in W_j\Longleftrightarrow X=\alpha^\ast\in\mathfrak{g}^{(j-1)}\cap(\mathfrak{g}^{(j)})^\perp\,.
\end{align*}
Let us assume that $d_{\mathfrak{g}}\alpha=\beta_1+\cdots+\beta_{l_{\alpha}}$, where there is at least a 2-form $\beta_k\in\Lambda^2\mathfrak{g}^\ast$ with $w(\beta_k)=m> j$. Then 
\begin{align*}
    \beta_k=-c\,X^\ast_{i_1}\wedge X_{i_2}^\ast
\end{align*}
with $c\in\mathbb{R}$, $X_{i_1}^\ast\in W_{n_1}$ and $X_{i_2}^\ast\in W_{n_2}$ so that $n_1+n_2=m>j$.

Let us notice that
\begin{align*}
    -c=\langle d_\mathfrak{g}\alpha\mid X_{i_1}\wedge X_{i_2}\rangle=-\langle\alpha\mid [X_{i_1},X_{i_2}]\rangle=-\langle X,[X_{i_1},X_{i_2}]\rangle
\end{align*}
which implies the equality $cX=[X_{i_1},X_{i_{2}}]$ must hold. However, since $X_{i_1}\in\mathfrak{g}^{(n_1-1)}\cap(\mathfrak{g}^{(n_1)})^\perp$ and $X_{i_2}\in\mathfrak{g}^{(n_2-1)}\cap(\mathfrak{g}^{(n_2)})^\perp$  we have
\begin{align*}
    X=[X_{i_1},X_{i_2}]\in [\mathfrak{g}^{(n_1-1)},\mathfrak{g}^{(n_2-1)}]\subset\mathfrak{g}^{n_1+n_2-1}=\mathfrak{g}^{(m-1)}\underbrace{\subset}_{m>j}\mathfrak{g}^{(j)}
\end{align*}
which leads to a contradiction since $X=\alpha^\ast\in\mathfrak{g}^{(j-1)}\cap(\mathfrak{g}^{(j)})^\perp$.

\end{proof}
\begin{cor}
Let $\alpha\in\Lambda^h\mathfrak{g}^\ast$  be a left-invariant 1-form, then $d_{\mathfrak{g}}\alpha=\beta_1+\cdots+\beta_{l_\alpha}$ where for each $i=1,\ldots,l_{\alpha}$ we have $w(\beta_i)\le w(\alpha)$ .
\end{cor}
\begin{proof}
It is sufficient to apply the previous proposition to the Leibniz rule of $d_\mathfrak{g}$.
\end{proof}
\begin{oss}
In the case of Carnot groups we have that $w(\beta_i)=w(\alpha)$ for all $1\le i\le l_{\alpha}$.
\end{oss}

We would like to use this result in order to find a more sophisticated way to express the exterior differential $d$ in terms of the weight increase.

In general, a $h$-form $\alpha\in\Omega^h$ of pure asymptotic weight $p$ will not be a left-invariant form. However, it can be expressed as a combination of the elements of the orthonormal basis $\lbrace \Tilde{\theta}_i\rbrace$ of $\Lambda^h\mathfrak{g}^\ast$ with coefficients in $C^\infty(G)$
\begin{align*}
    \alpha=\sum_{i}f_i\Tilde{\theta}_i\;\text{ with }f_i\in C^\infty(G)\text{ and }\Tilde{\theta}_i\in W^h_p\,,
\end{align*}
where $W_p^h$ denotes the space of left-invariant $h$-form that have asymptotic weight $p$
\begin{align*}
    W_p^h=\bigoplus_{i_1+i_2+\cdots+i_h=p}\Lambda^1W_{i_1}\otimes\Lambda^1W_{i_2}\otimes\cdots\otimes\Lambda^1W_{i_h} \,.
\end{align*}

Therefore, the expression for the differential in local coordinates will be:
\begin{align*}
    d\alpha=\sum_id(f_i\Tilde{\theta}_i)=\sum_idf_i\wedge\Tilde{\theta}_i+\sum_if_id\Tilde{\theta}_i=\sum_i\sum_{j=1}^nX_jf{\theta}_j\wedge\Tilde{\theta}_i+\sum_if_id\Tilde{\theta}_i\,,
\end{align*}
where $\lbrace \theta_j\rbrace_{1\le j\le n}$ is an orthonormal basis for $\Lambda^1\mathfrak{g}^\ast$.

This provides us with a well-posed decomposition of the different operator $d$ by weights.

\begin{defin}
Let $\alpha=\sum_if_i\Tilde{\theta}_i\in\Omega^h$ be an arbitrary $h$-form of pure asymptotic weight $p$, that is $\Tilde{\theta}_i\in W^h_p$, then we can write:
\begin{align*}
    d\alpha=d_{\mathfrak{g}}\alpha+d_1\alpha+d_2\alpha+\cdots+d_s\alpha\,,
\end{align*}
where $s$ is the nilpotency step of $G$.

For $1\le i\le s$, $d_k$ denotes the part of $d$ which increases the weight of the form $\alpha$ by $k$, whereas $d_\mathfrak{g}\alpha$ is the part of $d$ that does not increase the weight, that is:
\begin{itemize}
    \item $d_{\mathfrak{g}}\alpha=\sum_if_id\Tilde{\theta}_i$ by Proposition \ref{dg does not increase};
    \item $d_k\alpha=\sum_i\sum_{X_j^\ast=\theta_j\in W_k}X_jf_i\,\theta_j\wedge\Tilde{\theta}_i\in\Gamma(W_k\otimes W^h_p)$ for $1\le k\le s$.
\end{itemize}
\end{defin}
\section{The Rumin complex}
Now that we have introduced a new definition of asymptotic weights of forms and we have studied how the operator $d_\mathfrak{g}$ behaves with respect to these new weights, we can finally construct the whole complex $(E_0^\bullet,d_c)$. In this presentation, we follow the steps of the construction in \cite{FT}.

\begin{prop}
$(\Omega^\bullet,d_{\mathfrak{g}})$ is a complex.
\end{prop}
\begin{proof}
The claim follows directly from the definition of $d_\mathfrak{g}$. In particular, the cohomology of $(\Omega^\bullet,d_\mathfrak{g})$ is the Lie algebra cohomology of $\mathfrak{g}$ with coefficients in $C^\infty(G)$.
\end{proof}
\begin{defin}\label{defin inverse}
We want to define an inverse of the operator $d_{\mathfrak{g}}$. We can exploit the following map
\begin{align*}
    d_\mathfrak{g}\colon \Lambda^k\mathfrak{g}^\ast/Ker\,d_\mathfrak{g}\xrightarrow{\text{}}\Lambda^{k+1}\mathfrak{g}^\ast\,,
\end{align*}
so that by taking $\beta\in\Lambda^{k+1}\mathfrak{g}^\ast$ with $\beta\neq 0$, there exists a unique $\alpha\perp Ker\,d_\mathfrak{g}$ such that $d_\mathfrak{g}\alpha=\beta+\xi$ with $\xi\in(Im\,d_\mathfrak{g})^\perp$.

In general we have $\beta\notin Im\,d_\mathfrak{g}$, hence there exists $\delta\in (Im\,  d_\mathfrak{g})^\perp$ and $\exists \,\gamma\in Im\,d_\mathfrak{g}$ such that $\beta=\gamma+\delta$ In the expression before we will then have $\xi=-\delta$.

Hence we can define
\begin{align*}
    d_\mathfrak{g}^{-1}\colon \Lambda^{k+1}\mathfrak{g}^\ast&\to\lbrace \alpha\in\Lambda^k\mathfrak{g}^\ast\mid\alpha\perp Ker\,d_\mathfrak{g}\rbrace\\
    \beta\,&\mapsto d_{\mathfrak{g}}^{-1}\beta=\alpha\,.
\end{align*}
\end{defin}
\begin{oss}
$Im\,d_\mathfrak{g}^{-1}\cong(Ker\,d_\mathfrak{g})^\perp$.
\end{oss}

Let us consider the following operator
\begin{align*}
    d_\mathfrak{g}^{-1}d\colon Im\,d_\mathfrak{g}^{-1}\to Im\,d_\mathfrak{g}^{-1}\,.
\end{align*}

We can split this new operator depending on the filtration on $\mathfrak{g}^\ast$:
\begin{align*}
    d_\mathfrak{g}^{-1}d=d_\mathfrak{g}^{-1}(d_\mathfrak{g}+d_1+\cdots+d_s)=d_\mathfrak{g}^{-1}d_\mathfrak{g}+\underbrace{d_\mathfrak{g}^{-1}d_1+d_\mathfrak{g}^{-1}d_2+\cdots+d_\mathfrak{g}^{-1}d_s}_{=:D}\,.
\end{align*}
\begin{oss}
$d_\mathfrak{g}^{-1}d_\mathfrak{g}(\alpha)=\alpha$ for any $\alpha\in Im\,d_\mathfrak{g}^{-1}$.

It is a trivial result which comes directly from the fact that $Im\,d_\mathfrak{g}^{-1}\perp Ker\,d_\mathfrak{g}$.
\end{oss}
\begin{cor}
\begin{align*}
    d_\mathfrak{g}^{-1}d\bigg\vert_{Im\,d_\mathfrak{g}^{-1}}=Id+D\,.
\end{align*}
\end{cor}
\begin{prop}
${d_\mathfrak{g}^{-1}}$ {does not decrease weight.}
\end{prop}

\begin{proof}
Let us consider $\beta\in\Lambda^h\mathfrak{g}^\ast$, with $w(\beta)=p$. 

Let us consider $\beta'\in\Lambda^h\mathfrak{g}^\ast$ the projection of $\beta$ on $Im\,d_\mathfrak{g}$, then we have
\begin{align*}
    \beta'=&c_1(\beta_1^{1}+\cdots+\beta_{m_1}^{1})+c_2(\beta_{1}^{2}+\cdots+\beta_{m_2}^{2})+\cdots+c_l(\beta_{1}^{l}+\cdots+\beta_{m_l}^{l})\\=&c_1d_\mathfrak{g}\alpha_{1}+\cdots+c_ld_\mathfrak{g}\alpha_{l}=d_\mathfrak{g}\alpha
\end{align*}
where $\alpha=c_1\alpha_{1}+\cdots+c_l\alpha_{l}\in\Lambda^{h-1}\mathfrak{g}^\ast$, and for all $1\le k\le l$ we have $w(\alpha_{k})=p_k$ and 
\begin{align*}
    d_\mathfrak{g}\alpha_{k}=\beta_1^{k}+\cdots+\beta_{m_k}^{k}\,.
\end{align*}

Therefore $d_\mathfrak{g}^{-1}\beta=\alpha=c_1\alpha_{1}+\cdots+c_l\alpha_{l}$.
Let us assume that there exists a ${j}\in\lbrace 1,\ldots,l\rbrace$ such that $w(\alpha_{j})<w(\beta)$, then by Proposition \ref{dg does not increase} we have
\begin{align*}
    w(\beta)>w(\alpha_{j})\ge w(\beta_i^{j}) \;\text{ for all }\,1\le i\le m_j\,.
\end{align*}

This implies that the $h$-form $\beta$ is orthogonal to the subspace spanned by $\beta_{i}^{j}$, $1\le i\le m_j$, hence $c_j=0$. Therefore we have
\begin{align*}
    \beta'=\sum_{k=1, k\neq j}^lc_k\big(\beta_{1}^k+\beta_2^k+\cdots+\beta_{m_k}^k\big)
\end{align*}
and $d_\mathfrak{g}^{-1}\beta=c_1\alpha_1+\cdots+\widehat{c_j\alpha_j}+\cdots+c_l\alpha_l$, that is the component of $d_\mathfrak{g}^{-1}\beta$ with asymptotic weight strictly less than $p=w(\beta)$ is zero.
\end{proof}
\begin{cor}\label{operator D}
The operator $D:=d_\mathfrak{g}^{-1}(d-d_\mathfrak{g})$ increases the asymptotic weight of forms by at least 1. The fact that we are working on a nilpotent group and there exists a maximum weight for forms implies that the operator is indeed nilpotent:
\begin{align*}
    \exists\,N\in\mathbb{N}\text{ such that } D^N\equiv 0\,.
\end{align*}
\end{cor}

\begin{oss}
The concept of asymptotic weights of forms has been introduced here in order to create a parallel with the construction in \cite{rumin2000around,rumin2005Palermo}. In particular, these weights provide a direct tool to prove the nilpotency of the differential operator $D$. It should be stressed that the same result (and hence the same construction) can also be reached by studying the action of the algebraic operator $d_\mathfrak{g}$ with respect to the filtration $\lbrace \Lambda^\bullet F_i\rbrace_{1\le i\le s}$ induced by the lower central series on the space of left-invariant forms (see Appendix \ref{Alternative proof} for the intuition behind this argument).
\end{oss}

\begin{prop}
Let us introduce another differential operator:
\begin{align*}
    P:=\sum_{k=0}^N(-D)^k\,.
\end{align*}
$P$ is the inverse of $d_\mathfrak{g}^{-1}d$ on $Im\,d_\mathfrak{g}^{-1}$.
\end{prop}
\begin{proof}
First of all $P$ is well-defined as a differential operator on $Im\,d_\mathfrak{g}^{-1}$ as it is a finite sum of $d_\mathfrak{g}^{-1}d_1+\cdots+d_\mathfrak{g}^{-1}d_s$ and its finite powers.

We need to check that $P(d_\mathfrak{g}^{-1}d)\alpha=\alpha$ for an arbitrary $\alpha\in Im\,  d_\mathfrak{g}^{-1}$:
\begin{align*}
    P(d_\mathfrak{g}^{-1}d)\alpha=&P(d_\mathfrak{g}^{-1}d_\mathfrak{g}\alpha+d_\mathfrak{g}^{-1}(d-d_\mathfrak{g})\alpha)=P\alpha+Pd_\mathfrak{g}^{-1}(d-d_\mathfrak{g})\alpha
    =\sum_{k=0}^N(-D)^k\alpha+\sum_{k=0}^N(-D)^kD\alpha\\=&\sum_{k=0}^N(-D)^k\alpha-\sum_{k=1}^{N+1}(-D)^k\alpha=(-D)^0\alpha-\underbrace{(-D)^{N+1}\alpha}_{=0}=\alpha
\end{align*}
\end{proof}
\begin{oss}
The map $   d_\mathfrak{g}^{-1} {d}$ induces an isomorphism from $( {Ker\,  d_\mathfrak{g}})^\perp$ to itself, whose inverse is the differential operator $P$.

Moreover, this means that, when restricted to this subspace, the de Rham differential itself has a left inverse which we will denote by $ {Q}:= {P} {  d_\mathfrak{g}^{-1}}$, i.e. $ {Q\,d=Id}$ on $( {Ker\,  d_\mathfrak{g}})^\perp$.

Thus this subspace can be cut out from the de Rham complex, using the homotopical equivalence $Id- {Q\,d}$. One can also get rid of $ {d}( {Ker\,  d_\mathfrak{g}})^\perp$ and identify the remaining space.
\end{oss}
\begin{prop}
The operator $\Pi_{ {E}}:= {Id}- {Q\,d}- {d\,Q}$ is a projection operator on the subspace $\lbrace\alpha\in\Omega^\bullet\mid\Pi_ {E}\alpha=\alpha\rbrace=\lbrace\alpha\in\Omega^\bullet\mid {Q\,d}\alpha+ {d\,Q}\alpha=0\rbrace$.
\end{prop}
This new operator $Q$ is given by the following composition:
\begin{align*}
&\Lambda^{k+1}\mathfrak{g}^\ast\quad\xrightarrow{   d_\mathfrak{g}^{-1}}\quad\Lambda^k\mathfrak{g}^\ast\quad\xrightarrow{ {P}}\quad\Lambda^{k+1}\mathfrak{g}^\ast\\
&\;\;\beta\underbrace{\longmapsto}_{   d_\mathfrak{g}\alpha=\beta+\xi\,,\,\alpha\in {Im\,d}_\mathfrak{g}^{-1}}\alpha\;\;\;\mapsto\;\;\sum\limits_{k=0}^N(- {D})^k\alpha\;\in {Im\,d}_\mathfrak{g}^{-1}
\end{align*}
\begin{oss}
$(   d_\mathfrak{g}^{-1})^2=0$.

This follows from the fact that $   d_\mathfrak{g}^2=0$, that is $ {Im\,  d_\mathfrak{g}}\subset {Ker\,  d_\mathfrak{g}}$. Taking the orthogonal complements we obtain the reverse inclusions: $( {Ker\,  d_\mathfrak{g}})^\perp\subset( {Im\,d}_\mathfrak{g})^\perp= {Ker\,  d_\mathfrak{g}^{-1}}$, hence
\begin{align*}
    d_\mathfrak{g}^{-1}\alpha\in {Ker\,d}_\mathfrak{g}^{-1}\;\Rightarrow\;(   d_\mathfrak{g}^{-1})^2=0\,.
\end{align*}
\end{oss}
\begin{proof}
Let's compute $\Pi_{ {E}}^2$:
\begin{align*}
\Pi_ {E}^2&=( {Id}- {Q\,d}- {d\,Q})\,( {Id}- {Q\,d}- {d\,Q})\\
&= {Id}- {Q\,d}- {d\,Q}- {Q\,d}+ {Q\,d\,Q\,d}+ {d\,Q^2\,d}- {d\,Q}+\underbrace{ {Q\,d^2\,Q}}_{=0}+ {d\,Q\,d\,Q}\,.
\end{align*}

Moreover:
\begin{itemize}
\item[i.] $ {d\,Q\,d\,Q}= {d(P\,  d_\mathfrak{g}^{-1})d\,Q}= {d\,\underbrace{ P(  d_\mathfrak{g}^{-1}d)}_\text{$Id$ on $ {Im\,  d_\mathfrak{g}^{-1}}$}Q}= {d\,Q}$;
\item[ii.] $ {Q\,d\,Q\,d}=( {P\,  d_\mathfrak{g}^{-1}d})\, {Q\,d}= {Q\,d}$;
\item[iii.] $ {d\,Q^2\,d}= {d\,P\,  d_\mathfrak{g}^{-1}\underbrace{ P\,  d_\mathfrak{g}^{-1}d}}_{\in {Im\,d}_\mathfrak{g}^{-1}}=0$ since $(   d_\mathfrak{g}^{-1})^2=0$.
\end{itemize}
Thanks to all these simplifications we obtain:
\begin{align*}
\Pi_ {E}^2= {Id}- {Q\,d}- {d\,Q}=\Pi_{ {E}}.
\end{align*}
\end{proof}
\begin{teo}
\begin{align*}
 {E}=\lbrace\alpha\in\Omega^\bullet\mid {Q\,d}\alpha+ {d\,Q}\alpha=0\rbrace=\lbrace\alpha\in\Omega^\bullet\mid   d_\mathfrak{g}^{-1} {d}\alpha=0,\,   d_\mathfrak{g}^{-1}\alpha=0\rbrace.
\end{align*}
\end{teo}
\begin{proof}$\phantom{=}$\\
\fbox{$\supset$} It is just a direct verification.\\
\fbox{$\subset$} By hypothesis $ {  d_\mathfrak{g}^{-1}Q\,d\,}\alpha+   d_\mathfrak{g}^{-1} {d}\,Q\,\alpha=0$.\\
\begin{itemize}
\item First addend: $ {Q\,d\,}\alpha\in {Im\,  d_\mathfrak{g}^{-1}}$, so that $ {  d_\mathfrak{g}^{-1}Q\,d\,}\alpha=0$ since $(   d_\mathfrak{g}^{-1})^2=0$;
\item Second addend: $   d_\mathfrak{g}^{-1} {d}\,Q\,\alpha=\underbrace{   d_\mathfrak{g}^{-1} {d\, P}}_
{= {Id\;on\;Im\,  d_\mathfrak{g}^{-1}}}   d_\mathfrak{g}^{-1}\alpha= {  d_\mathfrak{g}^{-1}}\alpha$.
\end{itemize}

Hence $   d_\mathfrak{g}^{-1}\alpha=0$. Moreover:
\begin{align*}
0= {  d_\mathfrak{g}^{-1}d}( {Q\,d}\alpha+ {d\,Q}\alpha)=\underbrace{ {  d_\mathfrak{g}^{-1}d\,P}}_{ {Id}}\underbrace{ {  d_\mathfrak{g}^{-1}d}\alpha}_{\in {Im\,  d_\mathfrak{g}^{-1}}}+\underbrace{ {  d_\mathfrak{g}^{-1}d^2Q}}_{=0}\alpha= {  d_\mathfrak{g}^{-1}d}\alpha
\end{align*}
\end{proof}
We have just described the projector $\Pi_ {E}$. We can also take into consideration $\Pi_ {F}$, where
\begin{align*}
 {Id}=\Pi_ {E}+\underbrace{ {Id}-\Pi_ {E}}_{=:\Pi_ {F}}\,.
\end{align*}
\begin{lem}
$\Pi_ {F}$ is a projection as well.
\end{lem}
\begin{proof}\begin{align*}
    \Pi_{ {F}}^2=( {Id}-\Pi_{ {E}})( {Id}-\Pi_ {E})= {Id}-\Pi_ {E}-\Pi_ {E}+\Pi_ {E}^2= {Id}-\Pi_ {E}=\Pi_ {F}\,.
\end{align*}
\end{proof}

\begin{prop}
\begin{align*}
F:=\lbrace\alpha\in\Omega\mid\Pi_F\alpha=\alpha\rbrace= {Im\,  d_\mathfrak{g}^{-1}}+ {Im\,( {d\,  d_\mathfrak{g}^{-1}})}\,.
\end{align*}
\end{prop}
\begin{proof}
It is sufficient to re-express F in terms of $\Pi_E$:
\begin{align*}
F=\lbrace\alpha\in\Omega\mid\Pi_F\alpha=\alpha\rbrace=\lbrace\alpha\in\Omega\mid\Pi_E\alpha=0\rbrace\,.
\end{align*}
Hence we need to show that 
\begin{align*}
F=\lbrace\alpha\in\Omega\mid\alpha-Q {d}\alpha- {d}Q\alpha=0\rbrace\,.
\end{align*}
\fbox{$\supset$} It is just a direct verification: let's take $\alpha=   d_\mathfrak{g}^{-1}\beta_1+ {d\,d}_\mathfrak{g}^{-1}\beta_2$, then
\begin{align*}
\Pi_E\alpha&=\alpha-Q {d\,d}_\mathfrak{g}^{-1}\beta_1- {d}Q {d\,d}_\mathfrak{g}^{-1}\beta_2=
\alpha-P\underbrace{ {  d_\mathfrak{g}^{-1}d}}_{= {Id}+D}\,   d_\mathfrak{g}^{-1}\beta_1- {d}P\underbrace{ {  d_\mathfrak{g}^{-1}\,d}}_{= {Id}+D}   d_\mathfrak{g}^{-1}\beta_2\\
&=\alpha-\underbrace{P( {Id}+D)}_{= {Id}}   d_\mathfrak{g}^{-1}\beta_1- {d}\underbrace{P( {Id}+D)}_{= {Id}}   d_\mathfrak{g}^{-1}\beta_2=\alpha-   d_\mathfrak{g}^{-1}\beta_1- {d\,d}_\mathfrak{g}^{-1}\beta_2=0\,.
\end{align*}
\fbox{$\subset$} Given a form $\alpha=Q {d}\alpha+ {d}Q\alpha$ by applying the results already seen:
\begin{align*}
\alpha=P\underbrace{   d_\mathfrak{g}^{-1} {d}\alpha}_{\in {Im\,d}_\mathfrak{g}^{-1}}+ {d}\underbrace{P   d_\mathfrak{g}^{-1}\alpha}_{\in {Im\,d}_\mathfrak{g}^{-1}}\in  {Im\,d}_\mathfrak{g}^{-1}+ {Im\,d\,d}_\mathfrak{g}^{-1}
\end{align*}
since $P$ maps $ {Im\,d}_\mathfrak{g}^{-1}$ to $ {Im\,d}_\mathfrak{g}^{-1}$.

\end{proof}
At this point, we would like to describe some relevant spaces of forms on which we will restrict ourselves for our considerations. We will see that such a result can be achieved by applying the homotopical equivalence $ {Id}- {d\,  d_\mathfrak{g}^{-1}}-   d_\mathfrak{g}^{-1} {d}$.\\
First of all, we need to choose some orthogonal complement
of $ {Im\,d}_\mathfrak{g}^{-1}$ inside the $ {Ker\,d}_\mathfrak{g}$, since we are dealing with a complex ($   d_\mathfrak{g}^2=0$).\\
\begin{align*}
\overbrace{ {Im\,d}_\mathfrak{g}\;+\;  E_0}^{ {Ker\,d}_\mathfrak{g}}\;+\;&\big( {Ker\,d}_\mathfrak{g}\big)^{\perp}\\
\longdownarrow\quad\;\;\longdownarrow\quad&\quad\quad\longdownarrow\;\Leftarrow\text{This isomorphism defines }   d_\mathfrak{g}^{-1}\\
0\quad\quad\; 0\quad\quad\;&\;\; {Im\,d}_\mathfrak{g}\;+\;  E_0\;+\;\big( {Ker\,d}_\mathfrak{g}\big)^\perp\\
&\quad\quad\,\longdownarrow\;\,\quad\longdownarrow\quad\quad\;\;\longdownarrow\\
&\quad\quad 0\;\quad\quad\, 0\quad\quad\quad\;\;\;\; 0
\end{align*}
\begin{defin} For $0\le h\le n$ we set
\begin{align*}
  E_0^{h}:= {Ker\,d}_\mathfrak{g}\cap\big( {Im\,d}_\mathfrak{g}\big)^\perp\,\subset\Omega^h\,.
\end{align*}
The elements of $  E_0^h$ will be called intrinsic $h$-forms on ${G}$, or more simply Rumin $h$-forms on $G$.
\end{defin}


\begin{teo}
Using the same notations and definition above we have:
\begin{itemize}
\item[i.] the de Rham complex $(\Omega^\bullet, {d})$ splits into the direct sum of two sub-complexes $(E^\bullet, {d})$ and $(F^\bullet, {d})$ with
\begin{align*}
E= {Ker\,d}_\mathfrak{g}^{-1}\cap {Ker\,d}_\mathfrak{g}^{-1} {d}\;\text{and}\;F= {Im\,d}_\mathfrak{g}^{-1}+ {Im\,  d_\mathfrak{g}^{-1}d}\,;
\end{align*}
\item[ii.] the operator $\Pi_E= {Id}-Q {d}- {d}Q$, which is the projection on $E$ along $F$, is a homotopical equivalence between the subcomplex $(E^\bullet, {d})$ and the de Rham complex $(\Omega^\bullet, {d})$;
\item[iii.] if we denote by $\Pi_{  E_0}$ the orthogonal projection from $\Omega^\bullet$ to $  E_0^\bullet$, we have:
\begin{align*}
\Pi_{  E_0}= {Id}- {  d_\mathfrak{g}^{-1}  d_\mathfrak{g}}- {  d_\mathfrak{g}  d_\mathfrak{g}^{-1}}\;,\;\Pi_{  E_0^\perp}= {  d_\mathfrak{g}^{-1}  d_\mathfrak{g}}+ {  d_\mathfrak{g}  d_\mathfrak{g}^{-1}}
\end{align*}
and this holds also for covectors, because of the algebraic nature of both $   d_\mathfrak{g}$ and $   d_\mathfrak{g}^{-1}$;
\item[iv.] by defining $ {d}_c:=\Pi_{E_0} {d}\Pi_  {E}$, the exact complex $(  E_0^\bullet, {d}_c)$ is conjugated to the complex $(E^\bullet, {d})$.
\end{itemize}
\end{teo}
The relations which occur between the three complexes $(\Omega^\bullet, {d})$, $(E^\bullet, {d})$ and $(  E_0^\bullet, {d}_c)$ can be gathered in a simple way in the following commutative diagram:
\begin{center}
\begin{tikzpicture}[node distance=2.8cm, auto]

\pgfmathsetmacro{\shift}{0.3ex}

\node (A) {$\Omega^\bullet$};
\node(B)[right of=A] {$\Omega^\bullet$};
\node (C) [below of=A] {$E^\bullet$};
\node (D) [right of=C] {$E^\bullet$};
\node (E) [below of=C] {$  E_0^\bullet$};
\node (F) [below of=D] {$  E_0^\bullet$};

\draw[->](A) to node {\small $ {d}$}(B);
\draw[transform canvas={xshift=-0.5ex},->] (A) --(C) node[left,midway] {\footnotesize $\Pi_E$};
\draw[transform canvas={xshift=.5ex},->](C) -- (A) node[right,midway] {\footnotesize $\iota$};
\draw[transform canvas={xshift=-0.5ex},->] (B) --(D) node[left,midway] {\footnotesize $\Pi_E$};
\draw[transform canvas={xshift=.5ex},->](D) -- (B) node[right,midway] {\footnotesize $\iota$};
\draw[->](C) to node {\small $ {d}$}(D);
\draw[transform canvas={xshift=-0.5ex},->] (C) --(E) node[left,midway] {\footnotesize $\Pi_{  E_0}$};
\draw[transform canvas={xshift=.5ex},->](E) -- (C) node[right,midway] {\footnotesize $\Pi_E$};
\draw[transform canvas={xshift=-0.5ex},->] (D) --(F) node[left,midway] {\footnotesize $\Pi_{  E_0}$};
\draw[transform canvas={xshift=.5ex},->](F) -- (D) node[right,midway] {\footnotesize $\Pi_E$};
\draw[->](E) to node {\small $ {d}_c$}(F);
\end{tikzpicture}
\end{center}
\begin{proof}$\phantom{=}$
\begin{itemize}
\item[i.] Let us first prove that $E= {Ker\,  d_\mathfrak{g}^{-1}}\cap {Ker\,  d_\mathfrak{g}^{-1}d}$ is a complex: one needs to check that given $\xi\in E$, also $ {d}\xi\in E$. By definition, we have $   d_\mathfrak{g}^{-1}\xi= {  d_\mathfrak{g}^{-1}d}=0$, therefore $ {d}\xi\in {Ker\,  d_\mathfrak{g}^{-1}}$ and as $ {d}^2=0$ we have $ {d}\xi\in {Ker\,  d_\mathfrak{g}^{-1}d}$. The fact that $(F^\bullet, {d})$ is a complex can be proved in a similar way.

Moreover, the orthogonal subdivision $\Omega^\bullet=E^\bullet\oplus F^\bullet$ is due to the fact that $\Pi_E(\Omega^\bullet)=E^\bullet$ and $\Pi_F(\Omega^\bullet)= {Id}-\Pi_E(\Omega^\bullet)=F^\bullet$.
\item[ii.] First we need to check that $\Pi_E$ is a chain map. We have already seen that $\forall\,\alpha\in\Omega^\bullet$ we can write $\alpha=\Pi_E\alpha+\Pi_F\alpha$, so that $ {d}\alpha= {d}\Pi_E\alpha+ {d}\Pi_F\alpha$. We have just proved in point i. that both $E$ and $F$ are complexes, therefore $ {d}\Pi_E\alpha\in E$ and $ {d}\Pi_F\alpha\in F$, hence $\forall\,\alpha\in\Omega^\bullet$:
\begin{align*}
\Pi_E {d}\alpha=\Pi_E {d}\big(\Pi_E\alpha+\Pi_F\alpha\big)=\Pi_E {d}\Pi_E\alpha= {d}\Pi_E\alpha\,.
\end{align*}

Moreover, let's consider the retraction $r= {Id}- {  d_\mathfrak{g}^{-1}d}- {d  d_\mathfrak{g}^{-1}}$ on $\Omega^\bullet$, which preserves $ {d}$ and the cohomology (it is a homotopy itself). One can see that the maps $r^k$ stabilize to $\Pi_E$.

In fact, let's compute explicitly the formula for $\Pi_E$:
\begin{align*}
\Pi_E&= {Id}-Q {d}- {d}Q= {Id}-P {  d_\mathfrak{g}^{-1}d}- {d}P   d_\mathfrak{g}^{-1}\\
&= {Id}-\sum\limits_k\big(   d_\mathfrak{g}^{-1}( {d}-   d_\mathfrak{g}^{-1})\big)^k {  d_\mathfrak{g}^{-1}d}- {d}\sum\limits_k\big( {  d_\mathfrak{g}^{-1}}( {d}-   d_\mathfrak{g}^{-1})\big)^k {  d_\mathfrak{g}^{-1}}\\
&= {Id}-\underbrace{ {  d_\mathfrak{g}^{-1}d}- {d  d_\mathfrak{g}^{-1}}}_{k=0}+\underbrace{ {  d_\mathfrak{g}^{-1}( {d}- {  d_\mathfrak{g}^{-1}}) {  d_\mathfrak{g}^{-1}d}}+ {d  d_\mathfrak{g}^{-1}}( {d}- {  d_\mathfrak{g}}) {  d_\mathfrak{g}^{-1}}}_{k=1}+\\
&\phantom{=}\underbrace{- {  d_\mathfrak{g}^{-1}}( {d}-   d_\mathfrak{g}^{-1})   d_\mathfrak{g}^{-1}( {d}-   d_\mathfrak{g})   d_\mathfrak{g}^{-1} {d}- {d  d_\mathfrak{g}^{-1}}( {d}- {  d_\mathfrak{g}^{-1}}) {  d_\mathfrak{g}^{-1}}( {d-  d_\mathfrak{g}})   d_\mathfrak{g}^{-1}}_{k=2}+\cdots
\end{align*}
while
\begin{align*}
r^2&=( {Id-d  d_\mathfrak{g}^{-1}-  d_\mathfrak{g}^{-1}d})( {Id-d  d_\mathfrak{g}^{-1}-  d_\mathfrak{g}^{-1}d})\\
&= {Id-d  d_\mathfrak{g}^{-1}-  d_\mathfrak{g}^{-1}d-d  d_\mathfrak{g}^{-1}+d  d_\mathfrak{g}^{-1}d  d_\mathfrak{g}^{-1}-  d_\mathfrak{g}^{-1}d+  d_\mathfrak{g}^{-1}d  d_\mathfrak{g}^{-1}d}\\
&= {Id-d  d_\mathfrak{g}^{-1}-  d_\mathfrak{g}^{-1}d-d  d_\mathfrak{g}^{-1}+d  d_\mathfrak{g}^{-1}  d_\mathfrak{g}  d_\mathfrak{g}^{-1}+d  d_\mathfrak{g}^{-1}(d-  d_\mathfrak{g})  d_\mathfrak{g}^{-1}}+\\
&\phantom{=}\; {-  d_\mathfrak{g}^{-1}d+  d_\mathfrak{g}^{-1}  d_\mathfrak{g}  d_\mathfrak{g}^{-1}d+  d_\mathfrak{g}^{-1}(d-  d_\mathfrak{g})  d_\mathfrak{g}^{-1}d}\\
&= {Id-d  d_\mathfrak{g}^{-1}-  d_\mathfrak{g}^{-1}d+d  d_\mathfrak{g}^{-1}(d-  d_\mathfrak{g})  d_\mathfrak{g}^{-1}+  d_\mathfrak{g}^{-1}(d-  d_\mathfrak{g})  d_\mathfrak{g}^{-1}d}\,;
\end{align*}
\begin{align*}
r^3&=( {Id-d  d_\mathfrak{g}^{-1}-  d_\mathfrak{g}^{-1}d})( {Id-d  d_\mathfrak{g}^{-1}-  d_\mathfrak{g}^{-1}d+d  d_\mathfrak{g}^{-1}(d-  d_\mathfrak{g})  d_\mathfrak{g}^{-1}+  d_\mathfrak{g}^{-1}(d-  d_\mathfrak{g})  d_\mathfrak{g}^{-1}d})\\
&=r^2- {d  d_\mathfrak{g}^{-1}+d  d_\mathfrak{g}^{-1}d  d_\mathfrak{g}^{-1}-d  d_\mathfrak{g}^{-1}d  d_\mathfrak{g}^{-1}(d-  d_\mathfrak{g})  d_\mathfrak{g}^{-1}-  d_\mathfrak{g}^{-1}d+}
\\&\phantom{=}\;\, {+  d_\mathfrak{g}^{-1}d  d_\mathfrak{g}^{-1}d-  d_\mathfrak{g}^{-1}d  d_\mathfrak{g}^{-1}(d-  d_\mathfrak{g})  d_\mathfrak{g}^{-1}d}\\
&=r^2- {d  d_\mathfrak{g}^{-1}+d  d_\mathfrak{g}^{-1}  d_\mathfrak{g}  d_\mathfrak{g}^{-1}+d  d_\mathfrak{g}^{-1}(d-  d_\mathfrak{g})  d_\mathfrak{g}^{-1}-d  d_\mathfrak{g}^{-1}  d_\mathfrak{g}  d_\mathfrak{g}^{-1}(d-  d_\mathfrak{g})  d_\mathfrak{g}^{-1}}+\\&\phantom{=}\;\,+ {d  d_\mathfrak{g}^{-1}(d-  d_\mathfrak{g})  d_\mathfrak{g}^{-1}(d-  d_\mathfrak{g})  d_\mathfrak{g}^{-1}-  d_\mathfrak{g}^{-1}d+  d_\mathfrak{g}^{-1}  d_\mathfrak{g}  d_\mathfrak{g}^{-1}d+  d_\mathfrak{g}^{-1}(d-  d_\mathfrak{g})  d_\mathfrak{g}^{-1}d+}\\&\phantom{=}\;\, {-  d_\mathfrak{g}^{-1}  d_\mathfrak{g}  d_\mathfrak{g}^{-1}(d-  d_\mathfrak{g})  d_\mathfrak{g}^{-1}d-  d_\mathfrak{g}^{-1}(d-  d_\mathfrak{g})  d_\mathfrak{g}^{-1}(d-  d_\mathfrak{g})  d_\mathfrak{g}^{-1}d}\\&=r^2 {-d  d_\mathfrak{g}^{-1}(d-  d_\mathfrak{g})  d_\mathfrak{g}^{-1}(d-  d_\mathfrak{g})  d_\mathfrak{g}^{-1}-  d_\mathfrak{g}^{-1}(d-  d_\mathfrak{g})  d_\mathfrak{g}^{-1}(d-  d_\mathfrak{g})  d_\mathfrak{g}^{-1}d}
\end{align*}
and so on.

Therefore, since the sum $\sum_k\big( {  d_\mathfrak{g}^{-1}(d-  d_\mathfrak{g})}\big)^k$ is finite, the maps $r^k$ stabilize to $\Pi_E$ in a finite number of steps.
\item[iii.] $\Pi_{  E_0}$ is the zeroth order term of the homotopical equivalence $\Pi_E$. 

It is indeed a projection:
\begin{align*}
\Pi_{  E_0}^2&=( {Id-  d_\mathfrak{g}  d_\mathfrak{g}^{-1}-  d_\mathfrak{g}^{-1}  d_\mathfrak{g}})( {Id-  d_\mathfrak{g}  d_\mathfrak{g}^{-1}-  d_\mathfrak{g}^{-1}  d_\mathfrak{g}})\\&= {Id-  d_\mathfrak{g}  d_\mathfrak{g}^{-1}-  d_\mathfrak{g}^{-1}  d_\mathfrak{g}-  d_\mathfrak{g}  d_\mathfrak{g}^{-1}+  d_\mathfrak{g}  d_\mathfrak{g}^{-1}  d_\mathfrak{g}  d_\mathfrak{g}^{-1}-  d_\mathfrak{g}^{-1}  d_\mathfrak{g}+  d_\mathfrak{g}^{-1}  d_\mathfrak{g}  d_\mathfrak{g}^{-1}  d_\mathfrak{g}}\\&= {Id-  d_\mathfrak{g}  d_\mathfrak{g}^{-1}-  d_\mathfrak{g}^{-1}  d_\mathfrak{g}}=\Pi_{  E_0}\,.
\end{align*}
We need to check that we have:
\begin{align*}
  E_0=\lbrace\alpha\in\Omega^\bullet\mid\Pi_{  E_0}\alpha=\alpha\rbrace=\lbrace\alpha\in\Omega^\bullet\mid {  d_\mathfrak{g}  d_\mathfrak{g}^{-1}}\alpha+ {  d_\mathfrak{g}^{-1}  d_\mathfrak{g}}\alpha=0\rbrace\,.
\end{align*}
\fbox{$\supset$} It is a trivial verification;\\
\fbox{$\subset$} Let's take $\alpha\in\Omega^\bullet$ such that $ {  d_\mathfrak{g}  d_\mathfrak{g}^{-1}}\alpha+ {  d_\mathfrak{g}^{-1}  d_\mathfrak{g}}\alpha=0$, then
\begin{align*}
& {  d_\mathfrak{g}^{-1}}( {  d_\mathfrak{g}  d_\mathfrak{g}^{-1}}\alpha+ {  d_\mathfrak{g}^{-1}  d_\mathfrak{g}}\alpha)=0\;\Rightarrow\; {  d_\mathfrak{g}^{-1}}\alpha=0\;\\ & {  d_\mathfrak{g}}( {  d_\mathfrak{g}  d_\mathfrak{g}^{-1}}\alpha+ {  d_\mathfrak{g}^{-1}  d_\mathfrak{g}}\alpha)=0\;\Rightarrow\; {  d_\mathfrak{g}}\alpha=0\,.
\end{align*}
\item[iv.] One can show that
\begin{align*}
\Pi_E\Pi_{  E_0}\Pi_E=\Pi_E\;\text{and}\;\Pi_{  E_0}\Pi_E\Pi_{  E_0}=\Pi_{  E_0}\,;
\end{align*}
which implies that $E$ and $  E_0$ are in bijection, as $\Pi_E$ restricted to $  E_0$ and $\Pi_{  E_0}$ restricted to $E$ are inverse maps of each other:
\begin{align*}
\Pi_{  E_0}\Pi_E= {Id}\text{ on }  E_0\;\text{ and }\;\Pi_E\Pi_{  E_0}= {Id}\text{ on }E\,.
\end{align*}

Since $  E_0\subset {Ker\,}Q= {Ker\,}P {  d_\mathfrak{g}^{-1}}$ and $ {Im\,}Q\subset {Im\,  d_\mathfrak{g}^{-1}}\subset {Ker\,}\Pi_{  E_0}$, we have that:
\begin{align*}
\Pi_{  E_0}\Pi_E\Pi_{  E_0}=\Pi_{  E_0}( {Id}-Q {d}- {d}Q)\Pi_{  E_0}=\Pi_{  E_0}\,.
\end{align*}

Lastly, we have :
\begin{align*}
\Pi_E\Pi_{  E_0}\Pi_E&=\Pi_E( {Id}-\Pi_{  E_0^\perp})\Pi_E=\Pi_E-\Pi_E\Pi_{  E_0^{\perp}}\Pi_E\\&=\Pi_E-\Pi_E( {  d_\mathfrak{g}  d_\mathfrak{g}^{-1}+  d_\mathfrak{g}^{-1}  d_\mathfrak{g}})\Pi_E\\&=\Pi_E+\Pi_E {  d_\mathfrak{g}}\underbrace{ {  d_\mathfrak{g}^{-1}}\Pi_E}_{=0}+\Pi_E\underbrace{   d_\mathfrak{g}^{-1} {  d_\mathfrak{g}}\Pi_E}_{\in {Im\,d}_\mathfrak{g}^{-1}\subset F}=\Pi_E\,.
\end{align*}

And this implies that the complex $(E^\bullet, {d})$ is conjugated to the complex $(  E_0^\bullet, {d}_c)$, where $ {d}_c=\Pi_{  E_0} {d}\Pi_E$ (see the commutative diagram).
\end{itemize}
\end{proof}
\section{Hodge duality}
We want to show that the so-called Hodge $\ast$-duality holds on the Rumin complex as well and that on the cohomology classes we will have:
\begin{align*}
\ast\,E_0^{h}=E_0^{n-h}\,,
\end{align*}
where $n$ is the dimension of the nilpotent group $G$ we are considering.

To do so we need some technical results concerning the formal adjoint in $L^2({G},\Omega^\bullet)$ of the differential operators ${d_i}$s for $i=1,\ldots,s$ and the algebraic operator $d_\mathfrak{g}$.
\begin{defin}{Formal adjoint of $d$.}

We begin by recalling the definition of the formal adjoint $\delta$ of the de Rham exterior differential $d$ in $L^2({G},\Omega^\bullet)$:
\begin{align*}
\langle {d}\alpha,\beta\rangle=\langle\alpha,\delta\beta\rangle\;,\;\forall\,\alpha\in\Omega^{p-1}\;\forall\,\beta\in\Omega^{p}\,.
\end{align*}
The main result, which is a direct consequence of the definition, is that $\delta:\Omega^{p}\to\Omega^{p-1}$ satisfies the following equality:
\begin{align*}
\delta=(-1)^{n(p+1)+1}\ast{d}\ast\,,
\end{align*}
where $\ast$ denotes the Hodge-$\ast$ operator.
\end{defin}
\begin{prop}\label{Hodge}
Recalling the formal decomposition by asymptotic weights, for $h=0,1,\ldots,n$ and $i=1,\ldots,s$, we have:
\begin{align*}
\delta_i\big(\Gamma(W_p^h)\big)\subset\Gamma(W_{p-i}^{h-1})\;,\;\delta_i=(-1)^{n(h+1)+1}\ast{d}_i\ast\;\text{ and }\delta_\mathfrak{g}=(-1)^{n(h+1)+1}\ast{d}_\mathfrak{g}\ast\,.
\end{align*}
\end{prop}
\begin{proof}
For the first claim, we take $\alpha\in\Gamma(W_p^h)$ and $\beta\in\Omega^{h-1}$ and we have
\begin{align*}
\int\langle\delta_i\alpha,\beta\rangle\,{dV}=\int\langle\alpha,{d_i}\beta\rangle\,{dV}\neq 0
\end{align*}
only if $\beta\in\Gamma(W_{p-i}^{h-1})$.

For the second part, we have to keep in mind the orthogonal decomposition on forms given by the asymptotic weight:
\begin{align*}
\delta_\mathfrak{g}+\sum\limits_{i=1}^s\delta_i=\delta=(-1)^{n(h+1)+1}\ast{d}\ast=(-1)^{n(h+1)+1}\ast{d_\mathfrak{g}}\ast+\sum\limits_{i=1}^s(-1)^{n(h+1)+1}\ast{d_i}\ast\,
\end{align*}
so that $\forall\,i=1,\ldots,s$:
\begin{align*}
\delta_i=(-1)^{n(h+1)+1}\ast{d_i}\ast\,\text{ and }\delta_\mathfrak{g}=(-1)^{n(h+1)+1}\ast{d}_\mathfrak{g}.
\end{align*}
\end{proof}
\begin{cor}
$\delta_\mathfrak{g}$ is an algebraic operator too.
\end{cor}
We are now ready to prove the main result.
\begin{teo}
Let's take $0\le h\le n$ and let $\ast$ denote the the Hodge-$\ast$ operator, then
\begin{align*}
\ast\,E_0^h=E_0^{n-h}\,.
\end{align*}
\end{teo}
\begin{proof}
This result is based on the expression of $\delta_\mathfrak{g}$ in Proposition \ref{Hodge}.

To give an idea of the proof, let us forget about the signs and consider $\delta_\mathfrak{g}=\ast{d_\mathfrak{g}}\ast$. By definition of the cohomology classes, we know that $\forall\,\alpha\in E_0$ we have ${d_\mathfrak{g}}\alpha=\delta_\mathfrak{g}\alpha=0$. Therefore, we are left to prove that also $\ast\alpha$ belongs to a cohomology class, i.e. $\ast\alpha\in{Ker\,d_\mathfrak{g}}\cap\mathrm{Ker\,\delta_\mathfrak{g}}$ as an $(n-h)$-form.

We apply the results in Proposition \ref{Hodge} and obtain:
\begin{align*}
{d_\mathfrak{g}}(\ast\alpha)=\pm\ast\delta_\mathfrak{g}\alpha=0\;\text{ and }\;\delta_\mathfrak{g}(\ast)\alpha=\ast{d_\mathfrak{g}}\ast\ast\alpha=\pm\ast{d_\mathfrak{g}}\alpha=0\,.
\end{align*}
\end{proof}

\section{Examples}

\subsection{$N_{4,2}\times\mathbb{R}^2$}

Let us take into consideration the Carnot group $G$ whose Lie algebra is given by $N_{4,2}\times\mathbb{R}^2$, where $N_{4,2}$ is the step 3 nilpotent Lie algebra of dimension 4 (see \cite{Gong_Thesis}). 

The non-trivial brackets of its Lie algebra are the following:
\begin{align*}
    [X_1,X_2]=X_3\;,\;[X_1,X_3]=X_4\,.
\end{align*}
\subsubsection*{Left-invariant vector fields and 1-forms} $\phantom{=}$

The left-invariant vector fields are:
\begin{itemize}
    \item $X_1=\partial_{x_1}-\frac{x_2}{2}\partial_{x_3}-\big(\frac{x_3}{2}+\frac{x_1x_2}{12}\big)\partial_{x_4}$;
    \item $X_2=\partial_{x_2}+\frac{x_1}{2}\partial_{x_3}+\frac{x_1^2}{12}\partial_{x_4}$;
    \item $X_3=\partial_{x_3}+\frac{x_1}{2}\partial_{x_4}$;
    \item $X_4=\partial_{x_4}$;
    \item $X_5=\partial_{x_5}$;
    \item $X_6=\partial_{x_6}$,
\end{itemize}
and the respective left-invariant 1-forms are:
\begin{itemize}
    \item $\theta_1=dx_1$;
    \item $\theta_2=dx_2$;
    \item $\theta_3=dx_3-\frac{x_1}{2}dx_2+\frac{x_2}{2}dx_1$;
    \item $\theta_4=dx_4-\frac{x_1}{2}dx_3+\frac{x_1^2}{6}dx_2+\big(\frac{x_3}{2}-\frac{x_1x_2}{6}\big)dx_1$;
    \item $\theta_5=dx_5$;
    \item $\theta_6=dx_6$.
\end{itemize}

\subsubsection*{Orthonormal basis for left-invariant forms and asymptotic weights } $\phantom{=}$

We will pick $\lbrace\theta_i\rbrace_{i=1}^6$ to be an orthonormal basis of $\Lambda^1\mathfrak{g}^\ast$. We then have the following orthonormal basis for the space of left-invariant forms:
\begin{itemize}
    \item $\lbrace \theta_i\wedge\theta_i\rbrace_{i,j=1,i<j}^6$ orthonormal basis for $\Lambda^2\mathfrak{g}^\ast$;
    \item $\lbrace \theta_i\wedge\theta_j\wedge\theta_k\rbrace_{i,j,k=1,i<j<k}^6$ orthonormal basis for $\Lambda^3\mathfrak{g}^\ast$;
    \item $\lbrace \theta_i\wedge\theta_j\wedge\theta_k\wedge\theta_l\rbrace_{i,j,k,l=1,i<j<k<l}^6$ orthonormal basis for $\Lambda^4\mathfrak{g}^\ast$;
    \item $\lbrace \theta_1\wedge\theta_2\wedge\theta_3\wedge\theta_4\wedge\theta_5,\theta_1\wedge\theta_2\wedge\theta_3\wedge\theta_4\wedge\theta_6,\theta_1\wedge\theta_2\wedge\theta_3\wedge\theta_5\wedge\theta_6,,\theta_1\wedge\theta_3\wedge\theta_4\wedge\theta_5\wedge\theta_6,\theta_2\wedge\theta_3\wedge\theta_4\wedge\theta_5\wedge\theta_6\rbrace$ orthonormal basis for $\Lambda^5\mathfrak{g}^\ast$.
\end{itemize}

Let us study the action of the differential ${d}_\mathfrak{g}$ on the left-invariant 1-forms $\theta_i$ that generate $\Lambda^1\mathfrak{g}^\ast$:
\begin{itemize}
    \item $d_\mathfrak{g}\theta_1=d_\mathfrak{g}\theta_2=0$;
    \item $d_\mathfrak{g}\theta_3=-\theta_1\wedge \theta_2$;
    \item $d_\mathfrak{g}\theta_4=-\theta_1\wedge \theta_3$;
    \item $d_\mathfrak{g}\theta_5=d_\mathfrak{g}\theta_6=0$.
\end{itemize}

Let us construct the filtration $F_i$ of left-invariant 1-forms:
\begin{itemize}
    \item $F_0=0$;
    \item $F_1=\lbrace \alpha\in\mathfrak{g}^\ast\mid d_\mathfrak{g}\alpha=0\rbrace=span_{\mathbb{R}}\lbrace \theta_1,\theta_2,\theta_5,\theta_6\rbrace$;
    \item $F_2=\lbrace\alpha\in\mathfrak{g}^\ast\mid d_\mathfrak{g}\alpha\in\Lambda^2F_1\rbrace=span_{\mathbb{R}}\lbrace\theta_1,\theta_2,\theta_3,\theta_5,\theta_6\rbrace$;
    \item $F_3=\lbrace \alpha\in\mathfrak{g}^\ast\mid d_\mathfrak{g}\alpha\in\Lambda^2F_2\rbrace=span_{\mathbb{R}}\lbrace \theta_1,\theta_2,\theta_3,\theta_4,\theta_5,\theta_6\rbrace=\mathfrak{g}^\ast$.
\end{itemize}

Let us now define the asymptotic weights of 1-forms using the subspaces $W_i$ of $\mathfrak{g}^\ast$:
\begin{itemize}
    \item $W_1=F_1=span_{\mathbb{R}}\lbrace \theta_1,\theta_2,\theta_5,\theta_6\rbrace$;
    \item $W_2=F_2\cap(F_1)^\perp=span_{\mathbb{R}}\lbrace \theta_1,\theta_2,\theta_3,\theta_5,\theta_6\rbrace\cap(span_{\mathbb{R}}\lbrace\theta_1,\theta_2,\theta_5,\theta_6\rbrace)^\perp=span_{\mathbb{R}}\lbrace \theta_3\rbrace$;
    \item $W_3=F_3\cap(F_2)^\perp=span_{\mathbb{R}}\lbrace \theta_1,\theta_2,\theta_3,\theta_4,\theta_5,\theta_6\rbrace\cap(span_{\mathbb{R}}\lbrace\theta_1,\theta_2,\theta_3,\theta_5,\theta_6\rbrace)^\perp=span_{\mathbb{R}}\lbrace \theta_4\rbrace$.
\end{itemize}

Therefore the left-invariant 1-forms $\theta_i$ have the following asymptotic weights:
\begin{itemize}
    \item $w(\theta_1)=w(\theta_2)=w(\theta_5)=w(\theta_6)=1$;
    \item $w(\theta_3)=2$;
    \item $w(\theta_4)=3$.
\end{itemize}

\begin{oss}
Let us stress that in this case $G$ is a Carnot group with stratification
\begin{align*}
    \mathfrak{g}=V_1\oplus V_2\oplus V_3
\end{align*}
with $V_1=span_\mathbb{R}\lbrace X_1,X_2,X_5,X_6\rbrace$, $V_2=span_\mathbb{R}\lbrace X_3\rbrace$, and $V_3=span_\mathbb{R}\lbrace X_4\rbrace$.
The asymptotic weights given to the left-invariant 1-forms $\theta_i=X_i^\ast$ using the subspaces $W_j$ coincide with the weights that derive from the stratification $V_j$.

Furthermore, let us notice that indeed the differential $d_\mathfrak{g}$ maps 1-forms of asymptotic weight $w$ to 2-forms of same asymptotic weight $w$:
\begin{itemize}
    \item $d_\mathfrak{g}\theta_3=-\theta_1\wedge\theta_2$, where $w(\theta_3)=2=1+1=w(\theta_1)+w(\theta_2)=w(-\theta_1\wedge\theta_2)$, and
    \item $d_\mathfrak{g}\theta_4=-\theta_1\wedge\theta_3$, where $w(\theta_4)=3=1+2=w(\theta_1)+w(\theta_3)=w(-\theta_1\wedge\theta_3)$.
\end{itemize}
\end{oss}

\subsubsection*{Rumin forms}

Let us study the action of the differential $d_\mathfrak{g}$ on the space of all other left-invariant forms in order to compute the space of all Rumin forms $E_0^\bullet$.

We have already studied the action of $d_\mathfrak{g}$ on left-invariant 1-forms, hence we get that
\begin{align*}
    Ker\,d_\mathfrak{g}\cap\Omega^1=span_{C^\infty(G)}\lbrace \theta_1,\theta_2,\theta_5,\theta_6\rbrace\,,
\end{align*}
and since $Im\,d_\mathfrak{g}\cap\Omega^1=0$, we get
\begin{align*}
    E_0^1=span_{C^\infty(G)}\lbrace \theta_1,\theta_2,\theta_5,\theta_6\rbrace\,.
\end{align*}

Let us now study the action of $d_\mathfrak{g}$ on left-invariant 2-forms $\Lambda^2\mathfrak{g}^\ast$:
\begin{itemize}
    \item $d_\mathfrak{g}(\theta_1\wedge\theta_2)=d_\mathfrak{g}(\theta_1\wedge\theta_3)=d_\mathfrak{g}(\theta_1\wedge \theta_4)=d_\mathfrak{g}(\theta_1\wedge\theta_5)=d_\mathfrak{g}(\theta_1\wedge\theta_6)=d_\mathfrak{g}(\theta_2\wedge\theta_3)=0$;
    \item $d_\mathfrak{g}(\theta_2\wedge\theta_4)=\theta_2\wedge\theta_1\wedge\theta_3=-\theta_1\wedge\theta_2\wedge\theta_3$;
    \item $d_\mathfrak{g}(\theta_2\wedge\theta_5)=d_\mathfrak{g}(\theta_2\wedge\theta_6)=d_\mathfrak{g}(\theta_5\wedge\theta_6)=0$;
    \item $d_\mathfrak{g}(\theta_3\wedge\theta_4)=-\theta_1\wedge\theta_2\wedge\theta_4+\theta_3\wedge\theta_1\wedge\theta_3=-\theta_1\wedge\theta_2\wedge\theta_4$;
    \item $d_\mathfrak{g}(\theta_3\wedge\theta_5)=-\theta_1\wedge\theta_2\wedge\theta_5$;
    \item $d_\mathfrak{g}(\theta_3\wedge\theta_6)=-\theta_1\wedge\theta_2\wedge\theta_6$;
    \item $d_\mathfrak{g}(\theta_4\wedge\theta_5 )=-\theta_1\wedge\theta_3\wedge\theta_5$;
    \item $d_\mathfrak{g}(\theta_4\wedge\theta_6)=-\theta_1\wedge\theta_3\wedge\theta_6$.
\end{itemize}

Therefore 
\begin{align*}
    Ker\,d_\mathfrak{g}\cap\Omega^2=span_{C^\infty(G)}\lbrace &\theta_1\wedge\theta_2,\theta_1\wedge\theta_3,\theta_1\wedge\theta_4,\theta_1\wedge\theta_5,\theta_1\wedge\theta_6,\theta_2\wedge\theta_3,\\
    &\theta_2\wedge\theta_5,\theta_2\wedge\theta_6,\theta_5\wedge\theta_6\rbrace
\end{align*}
and since
\begin{align*}
    Im\,d_\mathfrak{g}\cap\Omega^2=span_{C^\infty(G)}\lbrace \theta_1\wedge\theta_2,\theta_1\wedge\theta_3\rbrace
\end{align*}
we get
\begin{align*}
    E_0^2=&span_{C^\infty(G)}\lbrace \theta_1\wedge\theta_4, \theta_1\wedge\theta_5,\theta_1\wedge\theta_6,\theta_2\wedge\theta_3,\theta_2\wedge\theta_5,\theta_2\wedge\theta_6,\theta_5\wedge\theta_6\rbrace\\=&span_{C^\infty(G)}\underbrace{\lbrace \theta_1\wedge\theta_5,\theta_1\wedge\theta_6,\theta_2\wedge\theta_5,\theta_2\wedge\theta_6,\theta_5\wedge\theta_6\rbrace}_{\text{weight }2}\oplus\\&\oplus span_{C^\infty(G)}\underbrace{\lbrace \theta_2\wedge\theta_3\rbrace}_{\text{weight }3}\oplus span_{C^\infty(G)}\underbrace{\lbrace \theta_1\wedge\theta_4\rbrace}_{\text{weight }4}\,.
\end{align*}

Let us now study the action of $d_\mathfrak{g}$ on left-invariant 3-forms $\Lambda^3\mathfrak{g}^\ast$:
\begin{itemize}
    \item $d_\mathfrak{g}(\theta_1\wedge\theta_2\wedge\theta_3)=d_\mathfrak{g}(\theta_1\wedge\theta_2\wedge\theta_4)=d_\mathfrak{g}(\theta_1\wedge\theta_2\wedge\theta_5)=d_\mathfrak{g}(\theta_1\wedge\theta_2\wedge\theta_6)=d_\mathfrak{g}(\theta_1\wedge\theta_3\wedge\theta_4)=0$;
    \item $d_\mathfrak{g}(\theta_1\wedge\theta_3\wedge\theta_5)=d_\mathfrak{g}(\theta_1\wedge\theta_3\wedge\theta_6)=d_\mathfrak{g}(\theta_1\wedge\theta_4\wedge\theta_5)=d_\mathfrak{g}(\theta_1\wedge\theta_4\wedge\theta_6)=d_\mathfrak{g}(\theta_1\wedge\theta_5\wedge\theta_6)=0$;
    \item $d_\mathfrak{g}(\theta_2\wedge\theta_3\wedge\theta_4)=d_\mathfrak{g}(\theta_2\wedge\theta_3\wedge\theta_5)=d_\mathfrak{g}(\theta_2\wedge\theta_3\wedge\theta_6)=d_\mathfrak{g}(\theta_2\wedge\theta_5\wedge\theta_6)=0$;
    \item $d_\mathfrak{g}(\theta_2\wedge\theta_4\wedge\theta_5)=\theta_2\wedge\theta_1\wedge\theta_3\wedge\theta_5=-\theta_1\wedge\theta_2\wedge\theta_3\wedge\theta_5$;
    \item $d_\mathfrak{g}(\theta_2\wedge\theta_4\wedge\theta_6)=\theta_2\wedge\theta_1\wedge\theta_3\wedge\theta_6=-\theta_1\wedge\theta_2\wedge\theta_3\wedge\theta_6$;
    \item $d_\mathfrak{g}(\theta_3\wedge\theta_4\wedge\theta_5)=-\theta_1\wedge\theta_2\wedge\theta_4\wedge\theta_5$;
    \item $d_\mathfrak{g}(\theta_3\wedge\theta_4\wedge\theta_6)=-\theta_1\wedge\theta_2\wedge\theta_4\wedge\theta_6$;
    \item $d_\mathfrak{g}(\theta_3\wedge\theta_5\wedge\theta_6)=-\theta_1\wedge\theta_2\wedge\theta_5\wedge\theta_6$;
    \item $d_\mathfrak{g}(\theta_4\wedge\theta_5\wedge\theta_6)=-\theta_1\wedge\theta_3\wedge\theta_5\wedge\theta_6$.
\end{itemize}

Therefore
\begin{align*}
    Ker\,d_\mathfrak{g}\cap\Omega^3=span_{C^\infty(G)}\lbrace & \theta_1\wedge\theta_2\wedge\theta_3,\theta_1\wedge\theta_2\wedge\theta_4,\theta_1\wedge\theta_2\wedge\theta_5,\theta_1\wedge\theta_2\wedge\theta_6,\theta_1\wedge\theta_3\wedge\theta_4,\\& \theta_1\wedge\theta_3\wedge\theta_5,\theta_1\wedge\theta_3\wedge\theta_6,\theta_1\wedge\theta_4\wedge\theta_5,\theta_1\wedge\theta_4\wedge\theta_6,\theta_1\wedge\theta_5\wedge\theta_6,\\&
    \theta_2\wedge\theta_3\wedge\theta_4,\theta_2\wedge\theta_3\wedge\theta_5,\theta_2\wedge\theta_3\wedge\theta_6,\theta_2\wedge\theta_5\wedge\theta_6\rbrace
\end{align*}
and since
\begin{align*}
    Im\,d_\mathfrak{g}\cap\Omega^3=span_{C^\infty(G)}\lbrace & \theta_1\wedge\theta_2\wedge\theta_3, \theta_1\wedge\theta_2\wedge\theta_4,\theta_1\wedge\theta_2\wedge\theta_5,\theta_1\wedge\theta_2\wedge\theta_6,\\&\theta_1\wedge\theta_3\wedge\theta_5,\theta_1\wedge\theta_3\wedge\theta_6\rbrace
\end{align*}
we get
\begin{align*}
    E_0^3=span_{C^\infty(G)}\lbrace & \theta_1\wedge\theta_3\wedge\theta_4,\theta_1\wedge\theta_4\wedge\theta_5,\theta_1\wedge\theta_4\wedge\theta_6,\theta_1\wedge\theta_5\wedge\theta_6,\\& \theta_2\wedge\theta_3\wedge\theta_4,\theta_2\wedge\theta_3\wedge\theta_5,\theta_2\wedge\theta_3\wedge\theta_6,\theta_2\wedge\theta_5\wedge\theta_6\rbrace\\
    =span_{C^\infty(G)}\lbrace &\underbrace{\theta_1\wedge\theta_5\wedge\theta_6,\theta_2\wedge\theta_5\wedge\theta_6\rbrace}_{\text{weight }3}\oplus span_{C^\infty(G)}\underbrace{\lbrace \theta_2\wedge\theta_3\wedge\theta_5,\theta_2\wedge\theta_3\wedge\theta_6\rbrace}_{\text{weight }4}\oplus\\\oplus span_{C^\infty(G)}\lbrace &\underbrace{\theta_1\wedge\theta_4\wedge\theta_5,\theta_1\wedge\theta_4\wedge\theta_6\rbrace}_{\text{weight }5}\oplus span_{C^\infty(G)}\underbrace{\lbrace \theta_1\wedge\theta_3\wedge\theta_4,\theta_2\wedge\theta_3\wedge\theta_4\rbrace}_{\text{weight }6}\,.
\end{align*}

Let us study the action of $d_\mathfrak{g}$ on left-invariant 4-forms $\Lambda^4 \mathfrak{g}^\ast$:
\begin{itemize}
    \item $d_\mathfrak{g}(\theta_1\wedge\theta_2\wedge\theta_3\wedge\theta_4)=d_\mathfrak{g}(\theta_1\wedge\theta_2\wedge\theta_3\wedge\theta_5)=d_\mathfrak{g}(\theta_1\wedge\theta_2\wedge\theta_3\wedge\theta_6)=d_\mathfrak{g}(\theta_1\wedge\theta_2\wedge\theta_4\wedge\theta_5)=0$;
    \item $d_\mathfrak{g}(\theta_1\wedge\theta_2\wedge\theta_4\wedge\theta_6)=d_\mathfrak{g}(\theta_1\wedge\theta_2\wedge\theta_5\wedge\theta_6)=d_\mathfrak{g}(\theta_1\wedge\theta_3\wedge\theta_4\wedge\theta_5)=d_\mathfrak{g}(\theta_1\wedge\theta_3\wedge\theta_4\wedge\theta_6)=0$;
    \item $d_\mathfrak{g}(\theta_1\wedge\theta_3\wedge\theta_5\wedge\theta_6)=d_\mathfrak{g}(\theta_1\wedge\theta_4\wedge\theta_5\wedge\theta_6)=d_\mathfrak{g}(\theta_2\wedge\theta_3\wedge\theta_4\wedge\theta_5)=d_\mathfrak{g}(\theta_2\wedge\theta_3\wedge\theta_4\wedge\theta_6)=0$;
    \item $d_\mathfrak{g}(\theta_2\wedge\theta_3\wedge\theta_5\wedge\theta_6)=0$;
    \item $d_\mathfrak{g}(\theta_2\wedge\theta_4\wedge\theta_5\wedge\theta_6)=\theta_2\wedge\theta_1\wedge\theta_3\wedge\theta_5\wedge\theta_6=-\theta_1\wedge\theta_2\wedge\theta_3\wedge\theta_5\wedge\theta_6$;
    \item $d_\mathfrak{g}(\theta_3\wedge\theta_4\wedge\theta_5\wedge\theta_6)=-\theta_1\wedge\theta_2\wedge\theta_4\wedge\theta_5\wedge\theta_6$.
\end{itemize}

Therefore
\begin{align*}
    Ker\,d_\mathfrak{g}\cap\Omega^4=span_{C^\infty(G)}\lbrace &\theta_1\wedge\theta_2\wedge\theta_3\wedge\theta_4,\theta_1\wedge\theta_2\wedge\theta_3\wedge\theta_5,\theta_1\wedge\theta_2\wedge\theta_3\wedge\theta_6, \theta_1\wedge\theta_2\wedge\theta_4\wedge\theta_5,\\&\theta_1\wedge\theta_2\wedge\theta_4\wedge\theta_6,\theta_1\wedge\theta_2\wedge\theta_5\wedge\theta_6, \theta_1\wedge\theta_3\wedge\theta_4\wedge\theta_5,\theta_1\wedge\theta_3\wedge\theta_4\wedge\theta_6,\\ &\theta_1\wedge\theta_3\wedge\theta_5\wedge\theta_6,\theta_1\wedge\theta_4\wedge\theta_5\wedge\theta_6,\theta_2\wedge\theta_3\wedge\theta_4\wedge\theta_5,\theta_2\wedge\theta_3\wedge\theta_4\wedge\theta_6,\\&\theta_2\wedge\theta_3\wedge\theta_5\wedge\theta_6\rbrace
\end{align*}
and since
\begin{align*}
    Im\,d_\mathfrak{g}\cap\Omega^4=span_{C^\infty(G)}\lbrace &\theta_1\wedge\theta_2\wedge\theta_3\wedge\theta_5,\theta_1\wedge\theta_2\wedge\theta_3\wedge\theta_6,\theta_1\wedge\theta_2\wedge\theta_4\wedge\theta_5,\theta_1\wedge\theta_2\wedge\theta_4\wedge\theta_6,\\&\theta_1\wedge\theta_2\wedge\theta_5\wedge\theta_6,\theta_1\wedge\theta_3\wedge\theta_5\wedge\theta_6\rbrace
\end{align*}
we get
\begin{align*}
    E_0^4=&span_{C^\infty(G)}\lbrace  \theta_1\wedge\theta_2\wedge\theta_3\wedge\theta_4,\theta_1\wedge\theta_3\wedge\theta_4\wedge\theta_5,\theta_1\wedge\theta_3\wedge\theta_4\wedge\theta_6,\theta_1\wedge\theta_4\wedge\theta_5\wedge\theta_6,\\&\phantom{span_{C^\infty(G)}\lbrace }\theta_2\wedge\theta_3\wedge\theta_4\wedge\theta_5,\theta_2\wedge\theta_3\wedge\theta_4\wedge\theta_6, \theta_2\wedge\theta_3\wedge\theta_5\wedge\theta_6\rbrace\\=&span_{C^\infty(G)}\lbrace \underbrace{\theta_2\wedge\theta_3\wedge\theta_5\wedge\theta_6\rbrace}_{\text{weight }5}\oplus span_{C^\infty(G)}\underbrace{\lbrace \theta_1\wedge\theta_4\wedge\theta_5\wedge\theta_5\rbrace}_{\text{weight }6}\oplus\\ &\oplus span_{C^\infty(G)}\underbrace{\lbrace \theta_1\wedge\theta_2\wedge\theta_3\wedge\theta_4, \theta_1\wedge\theta_3\wedge\theta_4\wedge\theta_5, \theta_1\wedge\theta_3\wedge\theta_4\wedge\theta_6,}_{\text{weight } 7}\\&\phantom{\oplus \;\,span_{C^\infty(G)}\lbrace } \underbrace{\theta_2\wedge\theta_3\wedge\theta_4\wedge\theta_5,\theta_2\wedge\theta_3\wedge\theta_4\wedge\theta_6\rbrace}_{\text{weight }7}\,.
\end{align*}

Let us study the action of $d_\mathfrak{g}$ on left-invariant 5-forms $\Lambda^5\mathfrak{g}^\ast$:
\begin{itemize}
    \item $d_{\mathfrak{g}}(\theta_1\wedge\theta_2\wedge\theta_3\wedge\theta_4\wedge\theta_5)=d_\mathfrak{g}(\theta_1\wedge\theta_2\wedge\theta_3\wedge\theta_4\wedge\theta_6)=d_\mathfrak{g}(\theta_1\wedge\theta_2\wedge\theta_3\wedge\theta_5\wedge\theta_6)=0$;
     \item $d_{\mathfrak{g}}(\theta_1\wedge\theta_2\wedge\theta_4\wedge\theta_5\wedge\theta_6)=d_\mathfrak{g}(\theta_1\wedge\theta_3\wedge\theta_4\wedge\theta_5\wedge\theta_6)=d_\mathfrak{g}(\theta_2\wedge\theta_3\wedge\theta_4\wedge\theta_5\wedge\theta_6)=0$.
\end{itemize}

Therefore
\begin{align*}
    Ker\,d_\mathfrak{g}\cap\Omega^5=\Omega^5
\end{align*}
and since
\begin{align*}
    Im\,d_\mathfrak{g}\cap\Omega^5=span_{C^\infty(G)}\lbrace \theta_1\wedge\theta_2\wedge\theta_3\wedge\theta_5\wedge\theta_6,\theta_1\wedge\theta_2\wedge\theta_4\wedge\theta_5\wedge\theta_6\rbrace
\end{align*}
we get
\begin{align*}
    E_0^5=span_{C^\infty(G)}\lbrace & \theta_1\wedge\theta_2\wedge\theta_3\wedge\theta_4\wedge\theta_5,\theta_1\wedge\theta_2\wedge\theta_3\wedge\theta_4\wedge\theta_6,\\ &\theta_1\wedge\theta_3\wedge\theta_4\wedge\theta_5\wedge\theta_6,\theta_2\wedge\theta_3\wedge\theta_4\wedge\theta_5\wedge\theta_6\rbrace\,.
\end{align*}

Finally, for 6-forms we have:
\begin{align*}
    E_0^6=\Omega^6=span_{C^\infty(G)}\lbrace \theta_1\wedge\theta_2\wedge\theta_3\wedge\theta_4\wedge\theta_5\wedge\theta_6\rbrace\,.
\end{align*}

\subsubsection*{The Rumin differential on 0-forms}

Let us study the differential $d_c$ when applied to the space of Rumin 0-forms $E_0^0=\Omega^0=C^\infty(G)$. Let us then take $f\in C^\infty(G)$, we then have
\begin{align*}
    d_cf=&\Pi_{E_0}d\Pi_Ef=\Pi_{E_0}d(f-Qdf-dQf)=\Pi_{E_0}d(f-P\underbrace{d_\mathfrak{g}^{-1}df}_{=0}-dP\underbrace{d_\mathfrak{g}^{-1}f}_{=0})=\Pi_{E_0}df\\=&\Pi_{E_0}(X_1f\theta_1+X_2f\theta_2+X_3f\theta_3+X_4f\theta_4+X_5f\theta_5+X_6f\theta_6)\\=&(Id-d_\mathfrak{g}^{-1}d_\mathfrak{g}-d_\mathfrak{g}d_\mathfrak{g}^{-1})(X_1f\theta_1+X_2f\theta_2+X_3f\theta_3+X_4f\theta_4+X_5f\theta_5+X_6f\theta_6)\\=&X_1f\theta_1+X_2f\theta_2+X_3f\theta_3+X_4f\theta_4+X_5f\theta_5+X_6f\theta_6-d_\mathfrak{g}^{-1}d_\mathfrak{g}(X_3f\theta_3)-d_\mathfrak{g}^{-1}d_\mathfrak{g}(X_4f\theta_4)\\=&X_1f\theta_1+X_2f\theta_2+X_5f\theta_5+X_6f\theta_6
\end{align*}

\subsubsection*{The Rumin differential on 1-forms} Let us study the differential 
\begin{align*}
    d_c\colon E_0^{1}\to E_0^2\,.
\end{align*}
In order to do so, we will break the computations into steps.

Given an arbitrary Rumin 1-form $\alpha=f_1\theta_1+f_2\theta_2+f_3\theta_5+f_4\theta_6$, we have
\begin{itemize}
    \item let us compute $d_\mathfrak{g}^{-1}d\alpha$:
\end{itemize}
\begin{align*}
    d_\mathfrak{g}^{-1}d\alpha=&d_\mathfrak{g}^{-1}d(f_1\theta_1+f_2\theta_2+f_3\theta_5+f_4\theta_6)\\=&d_\mathfrak{g}^{-1}\big[-X_2f_1\theta_1\wedge\theta_2-X_3f_1\theta_1\wedge\theta_3+X_1f_2\theta_1\wedge\theta_2+(Im\,d_\mathfrak{g})^\perp \big]\\=&d_\mathfrak{g}^{-1}\big[(X_1f_2-X_2f_1)\theta_1\wedge\theta_2-X_3f_1\theta_1\wedge\theta_3+(Im\,d_\mathfrak{g})^\perp\big]=-(X_1f_2-X_2f_1)\theta_3+X_3f_1\theta_4\,;
\end{align*}
\begin{itemize}
    \item let us compute $d_\mathfrak{g}^{-1}(d-d_\mathfrak{g})d_\mathfrak{g}^{-1}d\alpha$:
\end{itemize}
\begin{align*}
    d_\mathfrak{g}^{-1}(d-d_\mathfrak{g})d_\mathfrak{g}^{-1}d\alpha=&d_\mathfrak{g}^{-1}(d-d_\mathfrak{g})\big[-(X_1f_2-X_2f_1)\theta_3+X_3f_1\theta_4\big]\\=&d_\mathfrak{g}^{-1}\big[-X_1(X_1f_2-X_2f_1)\theta_1\wedge\theta_3+(Im\,d_\mathfrak{g})^\perp\big]=X_1(X_1f_2-X_2f_1)\theta_4\,;
\end{align*}
\begin{itemize}
    \item therefore, we have the following expression for $\Pi_E\alpha$:
\end{itemize}
\begin{align*}
    \Pi_E\alpha=&\alpha-d_\mathfrak{g}^{-1}d\alpha+d_\mathfrak{g}^{-1}(d-d_\mathfrak{g})d_\mathfrak{g}^{-1}d\alpha\\=&f_1\theta_1+f_2\theta_2+f_3\theta_5+f_4\theta_6+(X_1f_2-X_2f_1)\theta_3-X_3f_1\theta_4+X_1(X_1f_2-X_2f_1)\theta_4\\=&f_1\theta_1+f_2\theta_2+(X_1f_2-X_2f_1)\theta_3+[X_1(X_1f_2-X_2f_1)-X_3f_1]\theta_4+f_3\theta_5+f_4\theta_6\,;
\end{align*}
\begin{itemize}
    \item let us now compute $d\Pi_E\alpha$:
\end{itemize}
\begin{align*}
    d\Pi_E\alpha=&d\big[f_1\theta_1+f_2\theta_2+(X_1f_2-X_2f_1)\theta_3+[X_1(X_1f_2-X_2f_1)-X_3f_1]\theta_4+f_3\theta_5+f_4\theta_6\big]\\=&-X_2f_1\theta_1\wedge\theta_2-X_3f_1\theta_1\wedge\theta_3-X_4f_1\theta_1\wedge\theta_4-X_5f_1\theta_1\wedge\theta_5-X_6f_1\theta_1\wedge\theta_6+\\
    &+X_1f_2\theta_1\wedge\theta_2-X_3f_2\theta_2\wedge\theta_3-X_4f_2\theta_2\wedge\theta_4-X_5f_2\theta_2\wedge\theta_5-X_6f_2\theta_2\wedge\theta_6+\\
    &+X_1(X_1f_2-X_2f_1)\theta_1\wedge\theta_3+X_2(X_1f_2-X_2f_1)\theta_2\wedge\theta_3-X_4(X_1f_2-X_2f_1)\theta_3\wedge\theta_4+\\
    &-X_5(X_1f_2-X_2f_1)\theta_3\wedge\theta_5-X_6(X_1f_2-X_2f_1)\theta_3\wedge\theta_6+(X_1f_2-X_2f_1)d_\mathfrak{g}\theta_3+\\&+X_1[X_1(X_1f_2-X_2f_1)-X_3f_1]\theta_1\wedge\theta_4+X_2[X_1(X_1f_2-X_2f_1)-X_3f_1]\theta_2\wedge\theta_4+\\
    &+X_3[X_1(X_1f_2-X_2f_1)-X_3f_1]\theta_3\wedge\theta_4-X_5[X_1(X_1f_2-X_2f_1)-X_3f_1]\theta_4\wedge\theta_5+\\
    &-X_6[X_1(X_1f_2-X_2f_1)-X_3f_1]\theta_4\wedge\theta_6+[X_1(X_1f_2-X_2f_1)-X_3f_1]d_\mathfrak{g}\theta_4+\\&+X_1f_3\theta_1\wedge\theta_5+X_2f_3\theta_2\wedge\theta_5+X_3f_3\theta_3\wedge\theta_5+X_4f_3\theta_4\wedge\theta_5-X_6f_3\theta_5\wedge\theta_6+\\
    &+X_1f_4\theta_1\wedge\theta_6+X_2f_4\theta_2\wedge\theta_6+X_3f_4\theta_3\wedge\theta_6+X_4f_4\theta_4\wedge\theta_6+X_5f_4\theta_5\wedge\theta_6\end{align*}
    \begin{align*}
    \phantom{d\Pi_E\alpha}=&(X_1f_2-X_2f_1)\theta_1\wedge\theta_2+[X_1(X_1f_2-X_2f_1)-X_3f_1]\theta_1\wedge\theta_3+(X_1f_2-X_2f_1)d_\mathfrak{g}\theta_3+\\&+[X_1^2(X_1f_2-X_2f_1)-X_1X_3f_1-X_4f_1]\theta_1\wedge\theta_4+(X_1f_3-X_5f_1)\theta_1\wedge\theta_5+\\&+(X_1f_4-X_6f_1)\theta_1\wedge\theta_6+[X_2(X_1f_2-X_2f_1)-X_3f_2]\theta_2\wedge\theta_3+\\&+[X_2X_1(X_1f_2-X_2f_1)-X_2X_3f_1-X_4f_2]\theta_2\wedge\theta_4+(X_2f_3-X_5f_2)\theta_2\wedge\theta_5+\\&+(X_2f_4-X_6f_2)\theta_2\wedge\theta_6+[X_3X_1(X_1f_2-X_2f_1)-X_3^2f_1-X_4(X_1f_2-X_2f_1)]\theta_3\wedge\theta_4+\\&+[X_3f_3-X_5(X_1f_2-X_2f_1)]\theta_3\wedge\theta_5+[X_3f_4-X_6(X_1f_2-X_2f_1)]\theta_3\wedge\theta_6+\\&+[X_4f_3-X_5X_1(X_1f_2-X_2f_1)+X_5X_3f_1]\theta_4\wedge\theta_5+[X_1(X_1f_2-X_2f_1)-X_3f_1]d_\mathfrak{g}\theta_4+\\&+[X_4f_4-X_6X_1(X_1f_2-X_2f_1)+X_6X_3f_1]\theta_4\wedge\theta_6+(X_5f_4-X_6f_3)\theta_5\wedge\theta_6\,;
\end{align*}
\begin{itemize}
    \item before finishing the computations, we will study the action of the projection $\Pi_{E_0}=Id-d_\mathfrak{g}^{-1}d_\mathfrak{g}-d_\mathfrak{g}d_\mathfrak{g}^{-1}$ on each left-invariant 2-form:
    \begin{itemize}
        \item $\Pi_{E_0}(\theta_1\wedge\theta_2)=\theta_1\wedge\theta_2-d_\mathfrak{g}^{-1}d_\mathfrak{g}(\theta_1\wedge\theta_2)-d_\mathfrak{g}d_\mathfrak{g}^{-1}(\theta_1\wedge\theta_2)=0$;
        \item $\Pi_{E_0}(\theta_1\wedge\theta_3)=\theta_1\wedge\theta_3-d_\mathfrak{g}^{-1}d_\mathfrak{g}(\theta_1\wedge\theta_3)-d_\mathfrak{g}d_\mathfrak{g}^{-1}(\theta_1\wedge\theta_3)=0$;
        \item $\Pi_{E_0}(\theta_1\wedge\theta_4)=\theta_1\wedge\theta_4-d_\mathfrak{g}^{-1}d_\mathfrak{g}(\theta_1\wedge\theta_4)-d_\mathfrak{g}d_\mathfrak{g}^{-1}(\theta_1\wedge\theta_4)=\theta_1\wedge\theta_4$;
         \item $\Pi_{E_0}(\theta_1\wedge\theta_5)=\theta_1\wedge\theta_5-d_\mathfrak{g}^{-1}d_\mathfrak{g}(\theta_1\wedge\theta_5)-d_\mathfrak{g}d_\mathfrak{g}^{-1}(\theta_1\wedge\theta_5)=\theta_1\wedge\theta_5$;
          \item $\Pi_{E_0}(\theta_1\wedge\theta_6)=\theta_1\wedge\theta_6-d_\mathfrak{g}^{-1}d_\mathfrak{g}(\theta_1\wedge\theta_6)-d_\mathfrak{g}d_\mathfrak{g}^{-1}(\theta_1\wedge\theta_6)=\theta_1\wedge\theta_6$;
           \item $\Pi_{E_0}(\theta_2\wedge\theta_3)=\theta_2\wedge\theta_3-d_\mathfrak{g}^{-1}d_\mathfrak{g}(\theta_2\wedge\theta_3)-d_\mathfrak{g}d_\mathfrak{g}^{-1}(\theta_2\wedge\theta_3)=\theta_2\wedge\theta_3$;
           \item $\Pi_{E_0}(\theta_2\wedge\theta_4)=\theta_2\wedge\theta_4-d_\mathfrak{g}^{-1}d_\mathfrak{g}(\theta_2\wedge\theta_4)-d_\mathfrak{g}d_\mathfrak{g}^{-1}(\theta_2\wedge\theta_4)=0$;
           \item $\Pi_{E_0}(\theta_2\wedge\theta_5)=\theta_2\wedge\theta_5-d_\mathfrak{g}^{-1}d_\mathfrak{g}(\theta_2\wedge\theta_5)-d_\mathfrak{g}d_\mathfrak{g}^{-1}(\theta_2\wedge\theta_5)=\theta_2\wedge\theta_5$;
           \item $\Pi_{E_0}(\theta_2\wedge\theta_6)=\theta_2\wedge\theta_6-d_\mathfrak{g}^{-1}d_\mathfrak{g}(\theta_2\wedge\theta_6)-d_\mathfrak{g}d_\mathfrak{g}^{-1}(\theta_2\wedge\theta_6)=\theta_2\wedge\theta_6$;
           \item $\Pi_{E_0}(\theta_3\wedge\theta_4)=\theta_3\wedge\theta_4-d_\mathfrak{g}^{-1}d_\mathfrak{g}(\theta_3\wedge\theta_4)-d_\mathfrak{g}d_\mathfrak{g}^{-1}(\theta_3\wedge\theta_4)=0$;
           \item $\Pi_{E_0}(\theta_3\wedge\theta_5)=\theta_3\wedge\theta_5-d_\mathfrak{g}^{-1}d_\mathfrak{g}(\theta_3\wedge\theta_5)-d_\mathfrak{g}d_\mathfrak{g}^{-1}(\theta_3\wedge\theta_5)=0$;
           \item $\Pi_{E_0}(\theta_3\wedge\theta_6)=\theta_3\wedge\theta_6-d_\mathfrak{g}^{-1}d_\mathfrak{g}(\theta_3\wedge\theta_6)-d_\mathfrak{g}d_\mathfrak{g}^{-1}(\theta_3\wedge\theta_6)=0$;
           \item $\Pi_{E_0}(\theta_4\wedge\theta_5)=\theta_4\wedge\theta_5-d_\mathfrak{g}^{-1}d_\mathfrak{g}(\theta_4\wedge\theta_5)-d_\mathfrak{g}d_\mathfrak{g}^{-1}(\theta_4\wedge\theta_5)=0$;
           \item $\Pi_{E_0}(\theta_4\wedge\theta_6)=\theta_4\wedge\theta_6-d_\mathfrak{g}^{-1}d_\mathfrak{g}(\theta_4\wedge\theta_6)-d_\mathfrak{g}d_\mathfrak{g}^{-1}(\theta_4\wedge\theta_6)=0$;
            \item $\Pi_{E_0}(\theta_5\wedge\theta_6)=\theta_5\wedge\theta_6-d_\mathfrak{g}^{-1}d_\mathfrak{g}(\theta_5\wedge\theta_6)-d_\mathfrak{g}d_\mathfrak{g}^{-1}(\theta_5\wedge\theta_6)=\theta_5\wedge\theta_6$.
    \end{itemize}
    \item Finally, using all the previous computations, we can express explicitly the differential operator $d_c\alpha=\Pi_{E_0}d\Pi_E$, so that
    \begin{align*}
        d_c\alpha=&d_c(f_1\theta_1+f_2\theta_2+f_3\theta_5+f_4\theta_6)\\
        =&[X_1^2(X_1f_2-X_2f_1)-X_1X_3f_1-X_4f_1]\theta_1\wedge\theta_4+(X_1f_3-X_5f_1)\theta_1\wedge\theta_5+\\&+(X_1f_4-X_6f_1)\theta_1\wedge\theta_6+[X_2(X_1f_2-X_2f_1)-X_3f_2]\theta_2\wedge\theta_3+\\&+(X_2f_3-X_5f_2)\theta_2\wedge\theta_5+(X_2f_4-X_6f_2)\theta_2\wedge\theta_6+(X_5f_4-X_6f_3)\theta_5\wedge\theta_6\,.
    \end{align*}
\end{itemize}


\subsubsection*{The Rumin differential on 2-forms} Let us study the differential 
\begin{align*}
    d_c\colon E_0^{2}\to E_0^3\,.
\end{align*}
In order to do so, we will break the computations into steps.

Given an arbitrary Rumin 2-form $\beta=g_1\theta_1\wedge\theta_5+g_2\theta_1\wedge\theta_6+g_3\theta_2\wedge\theta_5+g_4\theta_2\wedge\theta_6+g_5\theta_5\wedge\theta_6+g_6\theta_2\wedge\theta_3+g_7\theta_1\wedge\theta_4$, we have
\begin{itemize}
    \item let us compute $d_\mathfrak{g}^{-1}d\beta$:
\end{itemize}
\begin{align*}
    d_\mathfrak{g}^{-1}d\beta=&d_\mathfrak{g}^{-1}d(g_1\theta_1\wedge\theta_5+g_2\theta_1\wedge\theta_6+g_3\theta_2\wedge\theta_5+g_4\theta_2\wedge\theta_6+g_5\theta_5\wedge\theta_6+g_6\theta_2\wedge\theta_3+g_7\theta_1\wedge\theta_4)\\=&d_\mathfrak{g}^{-1}\big[-X_2g_1\theta_1\wedge\theta_2\wedge\theta_5-X_3g_1\theta_1\wedge\theta_3\wedge\theta_5-X_2g_2\theta_1\wedge\theta_2\wedge\theta_6-X_3g_2\theta_1\wedge\theta_3\wedge\theta_6+\\+&X_1g_3\theta_1\wedge\theta_2\wedge\theta_5+X_1g_4\theta_1\wedge\theta_2\wedge\theta_6+X_1g_6\theta_1\wedge\theta_2\wedge\theta_3-X_2g_7\theta_1\wedge\theta_2\wedge\theta_4+(Im\,d_\mathfrak{g})^\perp \big]\\=&d_\mathfrak{g}^{-1}\big[(X_1g_3-X_2g_1)\theta_1\wedge\theta_2\wedge\theta_5-X_3g_1\theta_1\wedge\theta_3\wedge\theta_5+(X_1g_4-X_2g_2)\theta_1\wedge\theta_2\wedge\theta_6+\\&-X_3g_2\theta_1\wedge\theta_3\wedge\theta_6+X_1g_6\theta_1\wedge\theta_2\wedge\theta_3-X_2g_7\theta_1\wedge\theta_2\wedge\theta_4+(Im\,d_\mathfrak{g})^\perp \big]\\=&-(X_1g_3-X_2g_1)\theta_3\wedge\theta_5+X_3g_1\theta_4\wedge\theta_5-(X_1g_4-X_2g_2)\theta_3\wedge\theta_6+X_3g_2\theta_4\wedge\theta_6+\\&-X_1g_6\theta_2\wedge\theta_4+X_2g_7\theta_3\wedge\theta_4\,;
\end{align*}
\begin{itemize}
    \item let us compute $d_\mathfrak{g}^{-1}(d-d_\mathfrak{g})d_\mathfrak{g}^{-1}d\beta$:
\end{itemize}
\begin{align*}
    d_\mathfrak{g}^{-1}(d-d_\mathfrak{g})d_\mathfrak{g}^{-1}d\beta=&d_\mathfrak{g}^{-1}(d-d_\mathfrak{g})\big[-(X_1g_3-X_2g_1)\theta_3\wedge\theta_5+X_3g_1\theta_4\wedge\theta_5-(X_1g_4-X_2g_2)\theta_3\wedge\theta_6+\\&+X_3g_2\theta_4\wedge\theta_6-X_1g_6\theta_2\wedge\theta_4+X_2g_7\theta_3\wedge\theta_4\big]\\=&d_\mathfrak{g}^{-1}\big[-X_1(X_1g_3-X_2g_1)\theta_1\wedge\theta_3\wedge\theta_5-X_1(X_1g_4-X_2g_2)\theta_1\wedge\theta_3\wedge\theta_6+\\&-X_1^2g_6\theta_1\wedge\theta_2\wedge\theta_4
    +(Im\,d_\mathfrak{g})^\perp\big]\\=&X_1(X_1g_3-X_2g_1)\theta_4\wedge\theta_5+X_1(X_1g_4-X_2g_2)\theta_4\wedge\theta_6+X_1^2g_6\theta_3\wedge\theta_4\,;
\end{align*}
\begin{itemize}
    \item therefore, we have the following expression for $\Pi_E\beta$:
\end{itemize}
\begin{align*}
    \Pi_E\beta=&\beta-d_\mathfrak{g}^{-1}d\beta+d_\mathfrak{g}^{-1}(d-d_\mathfrak{g})d_\mathfrak{g}^{-1}d\beta\\=&g_1\theta_1\wedge\theta_5+g_2\theta_1\wedge\theta_6+g_3\theta_2\wedge\theta_5+g_4\theta_2\wedge\theta_6+g_5\theta_5\wedge\theta_6+g_6\theta_2\wedge\theta_3+g_7\theta_1\wedge\theta_4+\\&+(X_1g_3-X_2g_1)\theta_3\wedge\theta_5-X_3g_1\theta_4\wedge\theta_5+(X_1g_4-X_2g_2)\theta_3\wedge\theta_6-X_3g_2\theta_4\wedge\theta_6+\\&+X_1g_6\theta_2\wedge\theta_4-X_2g_7\theta_3\wedge\theta_4+X_1(X_1g_3-X_2g_1)\theta_4\wedge\theta_5+X_1(X_1g_4-X_2g_2)\theta_4\wedge\theta_6+\\&+X_1^2g_6\theta_3\wedge\theta_4\\=&g_1\theta_1\wedge\theta_5+g_2\theta_1\wedge\theta_6+g_3\theta_2\wedge\theta_5+g_4\theta_2\wedge\theta_6+g_5\theta_5\wedge\theta_6+g_6\theta_2\wedge\theta_3+g_7\theta_1\wedge\theta_4+\\&+(X_1g_3-X_2g_1)\theta_3\wedge\theta_5+[X_1(X_1g_3-X_2g_1)-X_3g_1]\theta_4\wedge\theta_5+(X_1g_4-X_2g_2)\theta_3\wedge\theta_6+\\&+[X_1(X_1g_4-X_2g_2)-X_3g_2]\theta_4\wedge\theta_6+X_1g_6\theta_2\wedge\theta_4+(X_1^2g_6-X_2g_7)\theta_3\wedge\theta_4\,;
\end{align*}
\begin{itemize}
    \item let us now compute $d\Pi_E\beta$:
\end{itemize}
\begin{align*}
    d\Pi_E\beta=&d\big[g_1\theta_1\wedge\theta_5+g_2\theta_1\wedge\theta_6+g_3\theta_2\wedge\theta_5+g_4\theta_2\wedge\theta_6+g_5\theta_5\wedge\theta_6+g_6\theta_2\wedge\theta_3+g_7\theta_1\wedge\theta_4+\\&+(X_1g_3-X_2g_1)\theta_3\wedge\theta_5+[X_1(X_1g_3-X_2g_1)-X_3g_1]\theta_4\wedge\theta_5+(X_1g_4-X_2g_2)\theta_3\wedge\theta_6+\\&+[X_1(X_1g_4-X_2g_2)-X_3g_2]\theta_4\wedge\theta_6+X_1g_6\theta_2\wedge\theta_4+(X_1^2g_6-X_2g_7)\theta_3\wedge\theta_4\big]\\
    =&-X_2g_1\theta_1\wedge\theta_2\wedge\theta_5-X_3g_1\theta_1\wedge\theta_3\wedge\theta_5-X_4g_1\theta_1\wedge\theta_4\wedge\theta_5+X_6g_1\theta_1\wedge\theta_5\wedge\theta_6+\\&
    -X_2g_2\theta_1\wedge\theta_2\wedge\theta_6-X_3g_2\theta_1\wedge\theta_3\wedge\theta_6-X_4g_2\theta_1\wedge\theta_4\wedge\theta_6-X_5g_2\theta_1\wedge\theta_5\wedge\theta_6+\\&
    +X_1g_3\theta_1\wedge\theta_2\wedge\theta_5-X_3g_3\theta_2\wedge\theta_3\wedge\theta_5-X_4g_3\theta_2\wedge\theta_4\wedge\theta_5+X_6g_3\theta_2\wedge\theta_5\wedge\theta_6+\\&
    +X_1g_4\theta_1\wedge\theta_2\wedge\theta_6-X_3g_4\theta_2\wedge\theta_3\wedge\theta_6-X_4g_4\theta_2\wedge\theta_4\wedge\theta_6-X_5g_4\theta_2\wedge\theta_5\wedge\theta_6+\\&
    +X_1g_5\theta_1\wedge\theta_5\wedge\theta_6+X_2g_5\theta_2\wedge\theta_5\wedge\theta_6+X_3g_5\theta_3\wedge\theta_5\wedge\theta_6+X_4g_5\theta_4\wedge\theta_5\wedge\theta_6+\\&+X_1g_6\theta_1\wedge\theta_2\wedge\theta_3+X_4g_6\theta_2\wedge\theta_3\wedge\theta_4+X_5g_6\theta_2\wedge\theta_3\wedge\theta_5+X_6g_6\theta_2\wedge\theta_3\wedge\theta_6+\\&-X_2g_7\theta_1\wedge\theta_2\wedge\theta_4-X_3g_7\theta_1\wedge\theta_3\wedge\theta_4+X_5g_7\theta_1\wedge\theta_4\wedge\theta_5+X_6g_7\theta_1\wedge\theta_4\wedge\theta_6+\\&+
    X_1(X_1g_3-X_2g_1)\theta_1\wedge\theta_3\wedge\theta_5+X_2(X_1g_3-X_2g_1)\theta_2\wedge\theta_3\wedge\theta_5+\\&
    -X_4(X_1g_3-X_2g_1)\theta_3\wedge\theta_4\wedge\theta_5+X_6(X_1g_3-X_2g_1)\theta_3\wedge\theta_5\wedge\theta_6+\\&
    +[X_1^2(X_1g_3-X_2g_1)-X_1X_3g_1]\theta_1\wedge\theta_4\wedge\theta_5+X_2[X_1(X_1g_3-X_2g_1)-X_3g_1]\theta_2\wedge\theta_4\wedge\theta_5+\\&+
    X_3[X_1(X_1g_3-X_2g_1)-X_3g_1]\theta_3\wedge\theta_4\wedge\theta_5+X_6[X_1(X_1g_3-X_2g_1)-X_3g_1]\theta_4\wedge\theta_5\wedge\theta_6+\\&
    +X_1(X_1g_4-X_2g_2)\theta_1\wedge\theta_3\wedge\theta_6+X_2(X_1g_4-X_2g_2)\theta_2\wedge\theta_3\wedge\theta_6+\\&-X_4(X_1g_4-X_2g_2)\theta_3\wedge\theta_4\wedge\theta_6-X_5(X_1g_4-X_2g_2)\theta_3\wedge\theta_5\wedge\theta_6+\\&+X_1[X_1(X_1g_4-X_2g_2)-X_3g_2]\theta_1\wedge\theta_4\wedge\theta_6+X_2[X_1(X_1g_4-X_2g_2)-X_3g_2]\theta_2\wedge\theta_4\wedge\theta_6+\\&+X_3[X_1(X_1g_4-X_2g_2)-X_3g_2]\theta_3\wedge\theta_4\wedge\theta_6-X_5[X_1(X_1g_4-X_2g_2)-X_3g_2]\theta_4\wedge\theta_5\wedge\theta_6+\\&+X_1^2g_6\theta_1\wedge\theta_2\wedge\theta_4-X_3X_1g_6\theta_2\wedge\theta_3\wedge\theta_4+X_5X_1g_6\theta_2\wedge\theta_4\wedge\theta_5+X_6X_1g_6\theta_2\wedge\theta_4\wedge\theta_6+\\&+X_1(X_1^2g_6-X_2g_7)\theta_1\wedge\theta_3\wedge\theta_4+X_2(X_1^2g_6-X_2g_7)\theta_2\wedge\theta_3\wedge\theta_4+\\&+X_5(X_1^2g_6-X_2g_7)\theta_3\wedge\theta_4\wedge\theta_5+X_6(X_1^2g_6-X_2g_7)\theta_3\wedge\theta_4\wedge\theta_6+\\&+(X_1g_3-X_2g_1)d_\mathfrak{g}(\theta_3\wedge\theta_5)+[X_1(X_1g_3-X_2g_1)-X_3g_1]d_\mathfrak{g}(\theta_4\wedge\theta_5)+(X_1g_4-X_2g_2)d_\mathfrak{g}(\theta_3\wedge\theta_6)+\\&+[X_1(X_1g_4-X_2g_2)-X_3g_2]d_\mathfrak{g}(\theta_4\wedge\theta_6)+X_1g_6d_\mathfrak{g}(\theta_2\wedge\theta_4)+(X_1^2g_6-X_2g_7)d_\mathfrak{g}(\theta_3\wedge\theta_4)\,;
\end{align*}
\begin{itemize}
    \item before finishing the computations, we will study the action of the projection $\Pi_{E_0}=Id-d_\mathfrak{g}^{-1}d_\mathfrak{g}-d_\mathfrak{g}d_\mathfrak{g}^{-1}$ on each left-invariant 3-form:
    \begin{itemize}
        \item $\Pi_{E_0}(\theta_1\wedge\theta_2\wedge\theta_3)=\theta_1\wedge\theta_2\wedge\theta_3-d_\mathfrak{g}^{-1}d_\mathfrak{g}(\theta_1\wedge\theta_2\wedge\theta_3)-d_\mathfrak{g}d_\mathfrak{g}^{-1}(\theta_1\wedge\theta_2\wedge\theta_3)=0$;
        \item $\Pi_{E_0}(\theta_1\wedge\theta_2\wedge\theta_4)=\theta_1\wedge\theta_2\wedge\theta_4-d_\mathfrak{g}^{-1}d_\mathfrak{g}(\theta_1\wedge\theta_2\wedge\theta_4)-d_\mathfrak{g}d_\mathfrak{g}^{-1}(\theta_1\wedge\theta_2\wedge\theta_4)=0$;
        \item $\Pi_{E_0}(\theta_1\wedge\theta_2\wedge\theta_5)=\theta_1\wedge\theta_2\wedge\theta_5-d_\mathfrak{g}^{-1}d_\mathfrak{g}(\theta_1\wedge\theta_2\wedge\theta_5)-d_\mathfrak{g}d_\mathfrak{g}^{-1}(\theta_1\wedge\theta_2\wedge\theta_5)=0$;
        \item $\Pi_{E_0}(\theta_1\wedge\theta_2\wedge\theta_6)=\theta_1\wedge\theta_2\wedge\theta_6-d_\mathfrak{g}^{-1}d_\mathfrak{g}(\theta_1\wedge\theta_2\wedge\theta_6)-d_\mathfrak{g}d_\mathfrak{g}^{-1}(\theta_1\wedge\theta_2\wedge\theta_6)=0$;
        \item $\Pi_{E_0}(\theta_1\wedge\theta_3\wedge\theta_4)=\theta_1\wedge\theta_3\wedge\theta_4-d_\mathfrak{g}^{-1}d_\mathfrak{g}(\theta_1\wedge\theta_3\wedge\theta_4)-d_\mathfrak{g}d_\mathfrak{g}^{-1}(\theta_1\wedge\theta_3\wedge\theta_4)=\theta_1\wedge\theta_3\wedge\theta_4$;
         \item $\Pi_{E_0}(\theta_1\wedge\theta_3\wedge\theta_5)=\theta_1\wedge\theta_3\wedge\theta_5-d_\mathfrak{g}^{-1}d_\mathfrak{g}(\theta_1\wedge\theta_3\wedge\theta_5)-d_\mathfrak{g}d_\mathfrak{g}^{-1}(\theta_1\wedge\theta_3\wedge\theta_5)=0$;
          \item $\Pi_{E_0}(\theta_1\wedge\theta_3\wedge\theta_6)=\theta_1\wedge\theta_3\wedge\theta_6-d_\mathfrak{g}^{-1}d_\mathfrak{g}(\theta_1\wedge\theta_3\wedge\theta_6)-d_\mathfrak{g}d_\mathfrak{g}^{-1}(\theta_1\wedge\theta_3\wedge\theta_6)=0$;
        \item $\Pi_{E_0}(\theta_1\wedge\theta_4\wedge\theta_5)=\theta_1\wedge\theta_4\wedge\theta_5-d_\mathfrak{g}^{-1}d_\mathfrak{g}(\theta_1\wedge\theta_4\wedge\theta_5)-d_\mathfrak{g}d_\mathfrak{g}^{-1}(\theta_1\wedge\theta_4\wedge\theta_5)=\theta_1\wedge\theta_4\wedge\theta_5$;
          \item $\Pi_{E_0}(\theta_1\wedge\theta_4\wedge\theta_6)=\theta_1\wedge\theta_4\wedge\theta_6-d_\mathfrak{g}^{-1}d_\mathfrak{g}(\theta_1\wedge\theta_4\wedge\theta_6)-d_\mathfrak{g}d_\mathfrak{g}^{-1}(\theta_1\wedge\theta_4\wedge\theta_6)=\theta_1\wedge\theta_4\wedge\theta_6$;
         \item $\Pi_{E_0}(\theta_1\wedge\theta_5\wedge\theta_6)=\theta_1\wedge\theta_5\wedge\theta_6-d_\mathfrak{g}^{-1}d_\mathfrak{g}(\theta_1\wedge\theta_5\wedge\theta_6)-d_\mathfrak{g}d_\mathfrak{g}^{-1}(\theta_1\wedge\theta_5\wedge\theta_6)=\theta_1\wedge\theta_5\wedge\theta_6$;
           \item $\Pi_{E_0}(\theta_2\wedge\theta_3\wedge\theta_4)=\theta_2\wedge\theta_3\wedge\theta_4-d_\mathfrak{g}^{-1}d_\mathfrak{g}(\theta_2\wedge\theta_3\wedge\theta_4)-d_\mathfrak{g}d_\mathfrak{g}^{-1}(\theta_2\wedge\theta_3\wedge\theta_4)=\theta_2\wedge\theta_3\wedge\theta_4$;
            \item $\Pi_{E_0}(\theta_2\wedge\theta_3\wedge\theta_5)=\theta_2\wedge\theta_3\wedge\theta_5-d_\mathfrak{g}^{-1}d_\mathfrak{g}(\theta_2\wedge\theta_3\wedge\theta_5)-d_\mathfrak{g}d_\mathfrak{g}^{-1}(\theta_2\wedge\theta_3\wedge\theta_5)=\theta_2\wedge\theta_3\wedge\theta_5$;
             \item $\Pi_{E_0}(\theta_2\wedge\theta_3\wedge\theta_6)=\theta_2\wedge\theta_3\wedge\theta_6-d_\mathfrak{g}^{-1}d_\mathfrak{g}(\theta_2\wedge\theta_3\wedge\theta_6)-d_\mathfrak{g}d_\mathfrak{g}^{-1}(\theta_2\wedge\theta_3\wedge\theta_6)=\theta_2\wedge\theta_3\wedge\theta_6$;
           \item $\Pi_{E_0}(\theta_2\wedge\theta_4\wedge\theta_5)=\theta_2\wedge\theta_4\wedge\theta_5-d_\mathfrak{g}^{-1}d_\mathfrak{g}(\theta_2\wedge\theta_4\wedge\theta_5)-d_\mathfrak{g}d_\mathfrak{g}^{-1}(\theta_2\wedge\theta_4\wedge\theta_5)=0$;
            \item $\Pi_{E_0}(\theta_2\wedge\theta_4\wedge\theta_6)=\theta_2\wedge\theta_4\wedge\theta_6-d_\mathfrak{g}^{-1}d_\mathfrak{g}(\theta_2\wedge\theta_4\wedge\theta_6)-d_\mathfrak{g}d_\mathfrak{g}^{-1}(\theta_2\wedge\theta_4\wedge\theta_6)=0$;
           \item $\Pi_{E_0}(\theta_2\wedge\theta_5\wedge\theta_6)=\theta_2\wedge\theta_5\wedge\theta_6-d_\mathfrak{g}^{-1}d_\mathfrak{g}(\theta_2\wedge\theta_5\wedge\theta_6)-d_\mathfrak{g}d_\mathfrak{g}^{-1}(\theta_2\wedge\theta_5\wedge\theta_6)=\theta_2\wedge\theta_5\wedge\theta_6$;
           \item $\Pi_{E_0}(\theta_3\wedge\theta_4\wedge\theta_5)=\theta_3\wedge\theta_4\wedge\theta_5-d_\mathfrak{g}^{-1}d_\mathfrak{g}(\theta_3\wedge\theta_4\wedge\theta_5)-d_\mathfrak{g}d_\mathfrak{g}^{-1}(\theta_3\wedge\theta_4\wedge\theta_5)=0$;
           \item $\Pi_{E_0}(\theta_3\wedge\theta_4\wedge\theta_6)=\theta_3\wedge\theta_4\wedge\theta_6-d_\mathfrak{g}^{-1}d_\mathfrak{g}(\theta_3\wedge\theta_4\wedge\theta_6)-d_\mathfrak{g}d_\mathfrak{g}^{-1}(\theta_3\wedge\theta_4\wedge\theta_6)=0$;
           \item $\Pi_{E_0}(\theta_3\wedge\theta_5\wedge\theta_6)=\theta_3\wedge\theta_5\wedge\theta_6-d_\mathfrak{g}^{-1}d_\mathfrak{g}(\theta_3\wedge\theta_5\wedge\theta_6)-d_\mathfrak{g}d_\mathfrak{g}^{-1}(\theta_3\wedge\theta_5\wedge\theta_6)=0$;
           \item $\Pi_{E_0}(\theta_4\wedge\theta_5\wedge\theta_6)=\theta_4\wedge\theta_5\wedge\theta_6-d_\mathfrak{g}^{-1}d_\mathfrak{g}(\theta_4\wedge\theta_5\wedge\theta_6)-d_\mathfrak{g}d_\mathfrak{g}^{-1}(\theta_4\wedge\theta_5\wedge\theta_6)=0$.
    \end{itemize}
    \item Finally, using all the previous computations, we can express explicitly the differential operator $d_c\beta=\Pi_{E_0}d\Pi_E\beta$, so that
    \begin{align*}
        d_c\beta=&d_c(g_1\theta_1\wedge\theta_5+g_2\theta_1\wedge\theta_6+g_3\theta_2\wedge\theta_5+g_4\theta_2\wedge\theta_6+g_5\theta_5\wedge\theta_6+g_6\theta_2\wedge\theta_3+g_7\theta_1\wedge\theta_4)\\
        =&-X_4g_1\theta_1\wedge\theta_4\wedge\theta_5+X_6g_1\theta_1\wedge\theta_5\wedge\theta_6
   -X_4g_2\theta_1\wedge\theta_4\wedge\theta_6-X_5g_2\theta_1\wedge\theta_5\wedge\theta_6+\\&
    -X_3g_3\theta_2\wedge\theta_3\wedge\theta_5+X_6g_3\theta_2\wedge\theta_5\wedge\theta_6
    -X_3g_4\theta_2\wedge\theta_3\wedge\theta_6-X_5g_4\theta_2\wedge\theta_5\wedge\theta_6+\\&
    +X_1g_5\theta_1\wedge\theta_5\wedge\theta_6+X_2g_5\theta_2\wedge\theta_5\wedge\theta_6+X_4g_6\theta_2\wedge\theta_3\wedge\theta_4+X_5g_6\theta_2\wedge\theta_3\wedge\theta_5+\\&+X_6g_6\theta_2\wedge\theta_3\wedge\theta_6-X_3g_7\theta_1\wedge\theta_3\wedge\theta_4+X_5g_7\theta_1\wedge\theta_4\wedge\theta_5+X_6g_7\theta_1\wedge\theta_4\wedge\theta_6+\\&+X_2(X_1g_3-X_2g_1)\theta_2\wedge\theta_3\wedge\theta_5    +[X_1^2(X_1g_3-X_2g_1)-X_1X_3g_1]\theta_1\wedge\theta_4\wedge\theta_5+\\&
    +X_2(X_1g_4-X_2g_2)\theta_2\wedge\theta_3\wedge\theta_6+X_1[X_1(X_1g_4-X_2g_2)-X_3g_2]\theta_1\wedge\theta_4\wedge\theta_6+\\&+X_1(X_1^2g_6-X_2g_7)\theta_1\wedge\theta_3\wedge\theta_4+X_2(X_1^2g_6-X_2g_7)\theta_2\wedge\theta_3\wedge\theta_4-X_3X_1g_6\theta_2\wedge\theta_3\wedge\theta_4\\=&
    (X_6g_1-X_5g_2+X_1g_5)\theta_1\wedge\theta_5\wedge\theta_6+(X_6g_3-X_5g_4+X_2g_5)\theta_2\wedge\theta_5\wedge\theta_6+\\&+(X_5g_6+X_2(X_1g_3-X_2g_1)-X_3g_3)\theta_2\wedge\theta_3\wedge\theta_5+\\&+(X_6g_6+X_2(X_1g_4-X_2g_2)-X_3g_4)\theta_2\wedge\theta_3\wedge\theta_6+\\&+(X_5g_7+X_1^2(X_1g_3-X_2g_1)-X_1X_3g_1-X_4g_1)\theta_1\wedge\theta_4\wedge\theta_5+\\&+(X_6g_7+X_1^2(X_1g_4-X_2g_2)-X_1X_3g_2-X_4g_2)\theta_1\wedge\theta_4\wedge\theta_6+\\&+(X_1^3g_6-X_1X_2g_7-X_3g_7)\theta_1\wedge\theta_3\wedge\theta_4+(X_4g_6+X_2X_1^2g_6-X_2^2g_7-X_3X_1g_6)\theta_2\wedge\theta_3\wedge\theta_4\,.
    \end{align*}
\end{itemize}


\subsection{$N_{6,3,2}$}\label{Esempio non-Carnot}

Let us take into consideration the nilpotent group $G$ whose Lie algebra is denoted as $N_{6,3,2}$ in \cite{Gong_Thesis}, with the following non-trivial Lie brackets: 
\begin{align*}
    [X_1,X_2]=X_3\;,\;[X_1,X_3]=X_4\;,\;[X_5,X_6]=X_4\,.
\end{align*}

Let us notice that its asymptotic cone has Lie algebra isomorphic to $N_{4,2}\times\mathbb{R}^2$. 

The non-trivial brackets of its Lie algebra are the following:
\subsubsection*{Left-invariant vector fields and 1-forms} $\phantom{=}$

The left-invariant vector fields are:
\begin{itemize}
    \item $X_1=\partial_{x_1}-\frac{x_2}{2}\partial_{x_3}-\big(\frac{x_3}{2}+\frac{x_1x_2}{12}\big)\partial_{x_4}$;
    \item $X_2=\partial_{x_2}+\frac{x_1}{2}\partial_{x_3}+\frac{x_1^2}{12}\partial_{x_4}$;
    \item $X_3=\partial_{x_3}+\frac{x_1}{2}\partial_{x_4}$;
    \item $X_4=\partial_{x_4}$;
    \item $X_5=\partial_{x_5}-\frac{x_6}{2}\partial_{x_4}$;
    \item $X_6=\partial_{x_6}+\frac{x_5}{2}\partial_{x_4}$,
\end{itemize}
and the respective left-invariant 1-forms are:
\begin{itemize}
    \item $\theta_1=dx_1$;
    \item $\theta_2=dx_2$;
    \item $\theta_3=dx_3-\frac{x_1}{2}dx_2+\frac{x_2}{2}dx_1$;
    \item $\theta_4=dx_4-\frac{x_1}{2}dx_3+\frac{x_1^2}{6}dx_2+\big(\frac{x_3}{2}-\frac{x_1x_2}{6}\big)dx_1+\frac{x_6}{2}dx_5-\frac{x_5}{2}dx_6$;
    \item $\theta_5=dx_5$;
    \item $\theta_6=dx_6$.
\end{itemize}

\subsubsection*{Orthonormal basis for left-invariant forms and asymptotic weights } $\phantom{=}$

We will pick $\lbrace\theta_i\rbrace_{i=1}^6$ to be an orthonormal basis of $\Lambda^1\mathfrak{g}^\ast$. We then have the following orthonormal bases for the space of left-invariant forms:
\begin{itemize}
    \item $\lbrace \theta_i\wedge\theta_i\rbrace_{i,j=1,i<j}^6$ orthonormal basis for $\Lambda^2\mathfrak{g}^\ast$;
    \item $\lbrace \theta_i\wedge\theta_j\wedge\theta_k\rbrace_{i,j,k=1,i<j<k}^6$ orthonormal basis for $\Lambda^3\mathfrak{g}^\ast$;
    \item $\lbrace \theta_i\wedge\theta_j\wedge\theta_k\wedge\theta_l\rbrace_{i,j,k,l=1,i<j<k<l}^6$ orthonormal basis for $\Lambda^4\mathfrak{g}^\ast$;
    \item $\lbrace \theta_1\wedge\theta_2\wedge\theta_3\wedge\theta_4\wedge\theta_5,\theta_1\wedge\theta_2\wedge\theta_3\wedge\theta_4\wedge\theta_6,\theta_1\wedge\theta_2\wedge\theta_3\wedge\theta_5\wedge\theta_6,,\theta_1\wedge\theta_3\wedge\theta_4\wedge\theta_5\wedge\theta_6,\theta_2\wedge\theta_3\wedge\theta_4\wedge\theta_5\wedge\theta_6\rbrace$ orthonormal basis for $\Lambda^5\mathfrak{g}^\ast$.
\end{itemize}

Let us study the action of the differential ${d}_\mathfrak{g}$ on the left-invariant 1-forms that form the orthonormal basis of $\Lambda^1\mathfrak{g}^\ast$:
\begin{itemize}
    \item $d_\mathfrak{g}\theta_1=d_\mathfrak{g}\theta_2=0$;
    \item $d_\mathfrak{g}\theta_3=-\theta_1\wedge \theta_2$;
    \item $d_\mathfrak{g}\theta_4=-\theta_1\wedge \theta_3-\theta_5\wedge\theta_6$;
    \item $d_\mathfrak{g}\theta_5=d_\mathfrak{g}\theta_6=0$.
\end{itemize}

Let us construct the filtration $F_i$ of left-invariant 1-forms:
\begin{itemize}
    \item $F_0=0$;
    \item $F_1=\lbrace \alpha\in\mathfrak{g}^\ast\mid d_\mathfrak{g}\alpha=0\rbrace=span_{\mathbb{R}}\lbrace \theta_1,\theta_2,\theta_5,\theta_6\rbrace$;
    \item $F_2=\lbrace\alpha\in\mathfrak{g}^\ast\mid d_\mathfrak{g}\alpha\in\Lambda^2F_1\rbrace=span_{\mathbb{R}}\lbrace\theta_1,\theta_2,\theta_3,\theta_5,\theta_6\rbrace$;
    \item $F_3=\lbrace \alpha\in\mathfrak{g}^\ast\mid d_\mathfrak{g}\alpha\in\Lambda^2F_2\rbrace=span_{\mathbb{R}}\lbrace \theta_1,\theta_2,\theta_3,\theta_4,\theta_5,\theta_6\rbrace=\mathfrak{g}^\ast$.
\end{itemize}

Let us now define the asymptotic weights of 1-forms using the subspaces $W_i$ of $\mathfrak{g}^\ast$:
\begin{itemize}
    \item $W_1=F_1=span_{\mathbb{R}}\lbrace \theta_1,\theta_2,\theta_5,\theta_6\rbrace$;
    \item $W_2=F_2\cap(F_1)^\perp=span_{\mathbb{R}}\lbrace \theta_1,\theta_2,\theta_3,\theta_5,\theta_6\rbrace\cap(span_{\mathbb{R}}\lbrace\theta_1,\theta_2,\theta_5,\theta_6\rbrace)^\perp=span_{\mathbb{R}}\lbrace \theta_3\rbrace$;
    \item $W_3=F_3\cap(F_2)^\perp=span_{\mathbb{R}}\lbrace \theta_1,\theta_2,\theta_3,\theta_4,\theta_5,\theta_6\rbrace\cap(span_{\mathbb{R}}\lbrace\theta_1,\theta_2,\theta_3,\theta_5,\theta_6\rbrace)^\perp=span_{\mathbb{R}}\lbrace \theta_4\rbrace$.
\end{itemize}

Therefore the left-invariant 1-forms of the orthonormal basis $\lbrace \theta_i\rbrace$ have the following asymptotic weights:
\begin{itemize}
    \item $w(\theta_1)=w(\theta_2)=w(\theta_5)=w(\theta_6)=1$;
    \item $w(\theta_3)=2$;
    \item $w(\theta_4)=3$.
\end{itemize}

\begin{oss}

Let us notice that the differential $d_\mathfrak{g}$ on left-invariant 1-forms does not increase the asymptotic weight:
\begin{itemize}
    \item $d_\mathfrak{g}\theta_3=-\theta_1\wedge\theta_2$, where $w(\theta_3)=2=1+1=w(\theta_1)+w(\theta_2)=w(-\theta_1\wedge\theta_2)$, and
    \item $d_\mathfrak{g}\theta_4=-\theta_1\wedge\theta_3-\theta_5\wedge\theta_6$, where $w(\theta_4)=3=1+2=w(\theta_1)+w(\theta_3)=w(-\theta_1\wedge\theta_3)$ and $w(\theta_4)=3>1+1=w(\theta_5)+w(\theta_6)=w(-\theta_5\wedge\theta_6)$.
\end{itemize}
\end{oss}

\subsubsection*{Rumin forms}

Let us study the action of the differential $d_\mathfrak{g}$ on the space of all other left-invariant forms in order to compute the space of all Rumin forms $E_0^\bullet$.

We have already studied the action of $d_\mathfrak{g}$ on left-invariant 1-forms, hence we get that
\begin{align*}
    Ker\,d_\mathfrak{g}\cap\Omega^1=span_{C^\infty(G)}\lbrace \theta_1,\theta_2,\theta_5,\theta_6\rbrace\,,
\end{align*}
and since $Im\,d_\mathfrak{g}\cap\Omega^1=0$, we get
\begin{align*}
    E_0^1=span_{C^\infty(G)}\lbrace \theta_1,\theta_2,\theta_5,\theta_6\rbrace\,.
\end{align*}

Let us now study the action of $d_\mathfrak{g}$ on left-invariant 2-forms $\Lambda^2\mathfrak{g}^\ast$:
\begin{itemize}
    \item $d_\mathfrak{g}(\theta_1\wedge\theta_2)=d_\mathfrak{g}(\theta_1\wedge\theta_3)=d_\mathfrak{g}(\theta_1\wedge\theta_5)=d_\mathfrak{g}(\theta_1\wedge\theta_6)=d_\mathfrak{g}(\theta_2\wedge\theta_3)=0$;
    \item $d_\mathfrak{g}(\theta_1\wedge \theta_4)=\theta_1\wedge(\theta_1\wedge\theta_3+\theta_5\wedge\theta_6)=\theta_1\wedge\theta_5\wedge\theta_6$;
    \item $d_\mathfrak{g}(\theta_2\wedge\theta_4)=\theta_2\wedge(\theta_1\wedge\theta_3+\theta_5\wedge\theta_6)=-\theta_1\wedge\theta_2\wedge\theta_3+\theta_2\wedge\theta_5\wedge\theta_6$;
    \item $d_\mathfrak{g}(\theta_2\wedge\theta_5)=d_\mathfrak{g}(\theta_2\wedge\theta_6)=0$;
    \item $d_\mathfrak{g}(\theta_3\wedge\theta_4)=-\theta_1\wedge\theta_2\wedge\theta_4+\theta_3\wedge(\theta_1\wedge\theta_3+\theta_5\wedge\theta_6)=-\theta_1\wedge\theta_2\wedge\theta_4+\theta_3\wedge\theta_5\wedge\theta_6$;
    \item $d_\mathfrak{g}(\theta_3\wedge\theta_5)=-\theta_1\wedge\theta_2\wedge\theta_5$;
    \item $d_\mathfrak{g}(\theta_3\wedge\theta_6)=-\theta_1\wedge\theta_2\wedge\theta_6$;
    \item $d_\mathfrak{g}(\theta_4\wedge\theta_5 )=-\theta_1\wedge\theta_3\wedge\theta_5$;
    \item $d_\mathfrak{g}(\theta_4\wedge\theta_6)=-\theta_1\wedge\theta_3\wedge\theta_6$;
    \item $d_\mathfrak{g}(\theta_5\wedge\theta_6)=0$.
\end{itemize}

Therefore 
\begin{align*}
    Ker\,d_\mathfrak{g}\cap\Omega^2=span_{C^\infty(G)}\lbrace &\theta_1\wedge\theta_2,\theta_1\wedge\theta_3,\theta_1\wedge\theta_5,\theta_1\wedge\theta_6,\theta_2\wedge\theta_3,\theta_2\wedge\theta_5,\theta_2\wedge\theta_6,\theta_5\wedge\theta_6\rbrace
\end{align*}
and since
\begin{align*}
    Im\,d_\mathfrak{g}\cap\Omega^2=span_{C^\infty(G)}\lbrace \theta_1\wedge\theta_2,\theta_1\wedge\theta_3+\theta_5\wedge\theta_6\rbrace
\end{align*}
we get
\begin{align*}
    E_0^2=&span_{C^\infty(G)}\lbrace \theta_5\wedge\theta_6-\theta_1\wedge\theta_3, \theta_1\wedge\theta_5,\theta_1\wedge\theta_6,\theta_2\wedge\theta_3,\theta_2\wedge\theta_5,\theta_2\wedge\theta_6\rbrace\\=&span_{C^\infty(G)}\underbrace{\lbrace \theta_1\wedge\theta_5,\theta_1\wedge\theta_6,\theta_2\wedge\theta_5,\theta_2\wedge\theta_6\rbrace}_{\text{weight }2}\oplus\\&\oplus span_{C^\infty(G)}\underbrace{\lbrace \theta_5\wedge\theta_6-\theta_1\wedge\theta_3\rbrace}_{\text{weight }2+\text{weight }3}\oplus span_{C^\infty(G)}\underbrace{\lbrace \theta_2\wedge\theta_3\rbrace}_{\text{weight }3}\,.
\end{align*}

Let us now study the action of $d_\mathfrak{g}$ on left-invariant 3-forms $\Lambda^3\mathfrak{g}^\ast$:
\begin{itemize}
    \item $d_\mathfrak{g}(\theta_1\wedge\theta_2\wedge\theta_3)=d_\mathfrak{g}(\theta_1\wedge\theta_2\wedge\theta_5)=d_\mathfrak{g}(\theta_1\wedge\theta_2\wedge\theta_6)=0$;
    \item $d_\mathfrak{g}(\theta_1\wedge\theta_3\wedge\theta_5)=d_\mathfrak{g}(\theta_1\wedge\theta_3\wedge\theta_6)=d_\mathfrak{g}(\theta_1\wedge\theta_4\wedge\theta_5)=d_\mathfrak{g}(\theta_1\wedge\theta_4\wedge\theta_6)=d_\mathfrak{g}(\theta_1\wedge\theta_5\wedge\theta_6)=0$;
    \item $d_\mathfrak{g}(\theta_2\wedge\theta_3\wedge\theta_5)=d_\mathfrak{g}(\theta_2\wedge\theta_3\wedge\theta_6)=d_\mathfrak{g}(\theta_2\wedge\theta_5\wedge\theta_6)=0$;
    \item $d_\mathfrak{g}(\theta_1\wedge\theta_2\wedge\theta_4)=-\theta_1\wedge\theta_2\wedge(\theta_1\wedge\theta_3+\theta_5\wedge\theta_6)=-\theta_1\wedge\theta_2\wedge\theta_5\wedge\theta_6$; 
    \item $d_\mathfrak{g}(\theta_1\wedge\theta_3\wedge\theta_4)=-\theta_1\wedge\theta_3\wedge(\theta_1\wedge\theta_3+\theta_5\wedge\theta_6)=-\theta_1\wedge\theta_3\wedge\theta_5\wedge\theta_6 $;
    \item $d_\mathfrak{g}(\theta_2\wedge\theta_3\wedge\theta_4)=-\theta_2\wedge\theta_3\wedge(\theta_1\wedge\theta_3+\theta_5\wedge\theta_6)=-\theta_2\wedge\theta_3\wedge\theta_5\wedge\theta_6$;
    \item $d_\mathfrak{g}(\theta_2\wedge\theta_4\wedge\theta_5)=\theta_2\wedge\theta_1\wedge\theta_3\wedge\theta_5=-\theta_1\wedge\theta_2\wedge\theta_3\wedge\theta_5$;
    \item $d_\mathfrak{g}(\theta_2\wedge\theta_4\wedge\theta_6)=\theta_2\wedge\theta_1\wedge\theta_3\wedge\theta_6=-\theta_1\wedge\theta_2\wedge\theta_3\wedge\theta_6$;
    \item $d_\mathfrak{g}(\theta_3\wedge\theta_4\wedge\theta_5)=-\theta_1\wedge\theta_2\wedge\theta_4\wedge\theta_5$;
    \item $d_\mathfrak{g}(\theta_3\wedge\theta_4\wedge\theta_6)=-\theta_1\wedge\theta_2\wedge\theta_4\wedge\theta_6$;
    \item $d_\mathfrak{g}(\theta_3\wedge\theta_5\wedge\theta_6)=-\theta_1\wedge\theta_2\wedge\theta_5\wedge\theta_6$;
    \item $d_\mathfrak{g}(\theta_4\wedge\theta_5\wedge\theta_6)=-\theta_1\wedge\theta_3\wedge\theta_5\wedge\theta_6$.
\end{itemize}

Therefore
\begin{align*}
    Ker\,d_\mathfrak{g}\cap\Omega^3=span_{C^\infty(G)}\lbrace & \theta_1\wedge\theta_2\wedge\theta_3, \theta_1\wedge\theta_2\wedge\theta_5,\theta_1\wedge\theta_2\wedge\theta_6, \theta_1\wedge\theta_3\wedge\theta_5,\theta_1\wedge\theta_3\wedge\theta_6,\\&\theta_1\wedge\theta_4\wedge\theta_5,\theta_1\wedge\theta_4\wedge\theta_6,\theta_1\wedge\theta_5\wedge\theta_6, \theta_2\wedge\theta_3\wedge\theta_5,\theta_2\wedge\theta_3\wedge\theta_6,\\&
    \theta_2\wedge\theta_5\wedge\theta_6,\theta_1\wedge\theta_2\wedge\theta_4-\theta_3\wedge\theta_5\wedge\theta_6,\theta_1\wedge\theta_3\wedge\theta_4-\theta_4\wedge\theta_5\wedge\theta_6\rbrace
\end{align*}
and since
\begin{align*}
    Im\,d_\mathfrak{g}\cap\Omega^3=span_{C^\infty(G)}\lbrace & \theta_1\wedge\theta_2\wedge\theta_3-\theta_2\wedge\theta_5\wedge\theta_6, \theta_1\wedge\theta_2\wedge\theta_4-\theta_3\wedge\theta_5\wedge\theta_6,\theta_1\wedge\theta_2\wedge\theta_5,\\&\theta_1\wedge\theta_2\wedge\theta_6,\theta_1\wedge\theta_3\wedge\theta_5,\theta_1\wedge\theta_3\wedge\theta_6, \theta_1\wedge\theta_5\wedge\theta_6\rbrace
\end{align*}
we get
\begin{align*}
    E_0^3=span_{C^\infty(G)}\lbrace & \theta_1\wedge\theta_2\wedge\theta_3+\theta_2\wedge\theta_5\wedge\theta_6,\theta_1\wedge\theta_4\wedge\theta_5,\theta_1\wedge\theta_4\wedge\theta_6,\\&\theta_2\wedge\theta_3\wedge\theta_5,\theta_2\wedge\theta_3\wedge\theta_6,\theta_1\wedge\theta_3\wedge\theta_4-\theta_4\wedge\theta_5\wedge\theta_6\rbrace\\
    =span_{C^\infty(G)}\lbrace &\underbrace{\theta_2\wedge\theta_5\wedge\theta_6+\theta_1\wedge\theta_2\wedge\theta_3\rbrace}_{\text{weight }3+\text{weight }4}\oplus span_{C^\infty(G)}\underbrace{\lbrace \theta_2\wedge\theta_3\wedge\theta_5,\theta_2\wedge\theta_3\wedge\theta_6\rbrace}_{\text{weight }4}\oplus\\\oplus span_{C^\infty(G)}\lbrace &\underbrace{\theta_1\wedge\theta_4\wedge\theta_5,\theta_1\wedge\theta_4\wedge\theta_6\rbrace}_{\text{weight }5}\oplus span_{C^\infty(G)}\underbrace{\lbrace \theta_1\wedge\theta_3\wedge\theta_4-\theta_4\wedge\theta_5\wedge\theta_6\rbrace}_{\text{weight }6+\text{weight }5}\,.
\end{align*}

Let us study the action of $d_\mathfrak{g}$ on left-invariant 4-forms $\Lambda^4\mathfrak{g}^\ast$:
\begin{itemize}
\item $d_\mathfrak{g}(\theta_1\wedge\theta_2\wedge\theta_3\wedge\theta_4)=\theta_1\wedge\theta_2\wedge\theta_3\wedge\theta_5\wedge\theta_6$;
    \item $d_\mathfrak{g}(\theta_1\wedge\theta_2\wedge\theta_3\wedge\theta_5)=d_\mathfrak{g}(\theta_1\wedge\theta_2\wedge\theta_3\wedge\theta_6)=d_\mathfrak{g}(\theta_1\wedge\theta_2\wedge\theta_4\wedge\theta_5)=d_\mathfrak{g}(\theta_2\wedge\theta_3\wedge\theta_5\wedge\theta_6)=0$;
    \item $d_\mathfrak{g}(\theta_1\wedge\theta_2\wedge\theta_4\wedge\theta_6)=d_\mathfrak{g}(\theta_1\wedge\theta_2\wedge\theta_5\wedge\theta_6)=d_\mathfrak{g}(\theta_1\wedge\theta_3\wedge\theta_4\wedge\theta_5)=d_\mathfrak{g}(\theta_1\wedge\theta_3\wedge\theta_4\wedge\theta_6)=0$;
    \item $d_\mathfrak{g}(\theta_1\wedge\theta_3\wedge\theta_5\wedge\theta_6)=d_\mathfrak{g}(\theta_1\wedge\theta_4\wedge\theta_5\wedge\theta_6)=d_\mathfrak{g}(\theta_2\wedge\theta_3\wedge\theta_4\wedge\theta_5)=d_\mathfrak{g}(\theta_2\wedge\theta_3\wedge\theta_4\wedge\theta_6)=0$;
    \item $d_\mathfrak{g}(\theta_2\wedge\theta_4\wedge\theta_5\wedge\theta_6)=\theta_2\wedge\theta_1\wedge\theta_3\wedge\theta_5\wedge\theta_6=-\theta_1\wedge\theta_2\wedge\theta_3\wedge\theta_5\wedge\theta_6$;
    \item $d_\mathfrak{g}(\theta_3\wedge\theta_4\wedge\theta_5\wedge\theta_6)=-\theta_1\wedge\theta_2\wedge\theta_4\wedge\theta_5\wedge\theta_6$.
\end{itemize}

Therefore
\begin{align*}
    Ker\,d_\mathfrak{g}\cap\Omega^4=span_{C^\infty(G)}\lbrace &\theta_1\wedge\theta_2\wedge\theta_3\wedge\theta_4+\theta_2\wedge\theta_4\wedge\theta_5\wedge\theta_6,\theta_1\wedge\theta_2\wedge\theta_3\wedge\theta_5,\theta_1\wedge\theta_2\wedge\theta_3\wedge\theta_6,\\& \theta_1\wedge\theta_2\wedge\theta_4\wedge\theta_5,\theta_1\wedge\theta_2\wedge\theta_4\wedge\theta_6,\theta_1\wedge\theta_2\wedge\theta_5\wedge\theta_6, \theta_1\wedge\theta_3\wedge\theta_4\wedge\theta_5,\\&\theta_1\wedge\theta_3\wedge\theta_4\wedge\theta_6,\theta_1\wedge\theta_3\wedge\theta_5\wedge\theta_6,\theta_1\wedge\theta_4\wedge\theta_5\wedge\theta_6,\theta_2\wedge\theta_3\wedge\theta_4\wedge\theta_5,\\&\theta_2\wedge\theta_3\wedge\theta_4\wedge\theta_6,\theta_2\wedge\theta_3\wedge\theta_5\wedge\theta_6\rbrace
\end{align*}
and since
\begin{align*}
    Im\,d_\mathfrak{g}\cap\Omega^4=span_{C^\infty(G)}\lbrace &\theta_1\wedge\theta_2\wedge\theta_3\wedge\theta_5,\theta_1\wedge\theta_2\wedge\theta_3\wedge\theta_6,\theta_1\wedge\theta_2\wedge\theta_4\wedge\theta_5,\theta_1\wedge\theta_2\wedge\theta_4\wedge\theta_6,\\&\theta_1\wedge\theta_2\wedge\theta_5\wedge\theta_6,\theta_1\wedge\theta_3\wedge\theta_5\wedge\theta_6,\theta_2\wedge\theta_3\wedge\theta_5\wedge\theta_6\rbrace
\end{align*}
we get
\begin{align*}
    E_0^4=&span_{C^\infty(G)}\lbrace  \theta_1\wedge\theta_2\wedge\theta_3\wedge\theta_4+\theta_2\wedge\theta_4\wedge\theta_5\wedge\theta_6,\theta_1\wedge\theta_3\wedge\theta_4\wedge\theta_5,\theta_1\wedge\theta_3\wedge\theta_4\wedge\theta_6,\\&\phantom{span_{C^\infty(G)}\lbrace }\theta_1\wedge\theta_4\wedge\theta_5\wedge\theta_6,\theta_2\wedge\theta_3\wedge\theta_4\wedge\theta_5,\theta_2\wedge\theta_3\wedge\theta_4\wedge\theta_6\rbrace\\=& span_{C^\infty(G)}\underbrace{\lbrace \theta_1\wedge\theta_4\wedge\theta_5\wedge\theta_6\rbrace}_{\text{weight }6}\oplus span_{C^\infty(G)}\underbrace{\lbrace\theta_1\wedge\theta_2\wedge\theta_3\wedge\theta_4+\theta_2\wedge\theta_4\wedge\theta_5\wedge\theta_6\rbrace}_{\text{weight }7+\text{weight }6}\oplus\\ &\oplus span_{C^\infty(G)}\underbrace{\lbrace  \theta_1\wedge\theta_3\wedge\theta_4\wedge\theta_5, \theta_1\wedge\theta_3\wedge\theta_4\wedge\theta_6,\theta_2\wedge\theta_3\wedge\theta_4\wedge\theta_5,\theta_2\wedge\theta_3\wedge\theta_4\wedge\theta_6\rbrace}_{\text{weight } 7}\,.
\end{align*}

Let us study the action of $d_\mathfrak{g}$ on left-invariant 5-forms $\Lambda^5\mathfrak{g}^\ast$:
\begin{itemize}
    \item $d_{\mathfrak{g}}(\theta_1\wedge\theta_2\wedge\theta_3\wedge\theta_4\wedge\theta_5)=d_\mathfrak{g}(\theta_1\wedge\theta_2\wedge\theta_3\wedge\theta_4\wedge\theta_6)=d_\mathfrak{g}(\theta_1\wedge\theta_2\wedge\theta_3\wedge\theta_5\wedge\theta_6)=0$;
     \item $d_{\mathfrak{g}}(\theta_1\wedge\theta_2\wedge\theta_4\wedge\theta_5\wedge\theta_6)=d_\mathfrak{g}(\theta_1\wedge\theta_3\wedge\theta_4\wedge\theta_5\wedge\theta_6)=d_\mathfrak{g}(\theta_2\wedge\theta_3\wedge\theta_4\wedge\theta_5\wedge\theta_6)=0$.
\end{itemize}

Therefore
\begin{align*}
    Ker\,d_\mathfrak{g}\cap\Omega^5=\Omega^5
\end{align*}
and since
\begin{align*}
    Im\,d_\mathfrak{g}\cap\Omega^5=span_{C^\infty(G)}\lbrace \theta_1\wedge\theta_2\wedge\theta_3\wedge\theta_5\wedge\theta_6,\theta_1\wedge\theta_2\wedge\theta_4\wedge\theta_5\wedge\theta_6\rbrace
\end{align*}
we get
\begin{align*}
    E_0^5=span_{C^\infty(G)}\lbrace & \theta_1\wedge\theta_2\wedge\theta_3\wedge\theta_4\wedge\theta_5,\theta_1\wedge\theta_2\wedge\theta_3\wedge\theta_4\wedge\theta_6,\\ &\theta_1\wedge\theta_3\wedge\theta_4\wedge\theta_5\wedge\theta_6,\theta_2\wedge\theta_3\wedge\theta_4\wedge\theta_5\wedge\theta_6\rbrace\,.
\end{align*}

Finally, for 6-forms we have:
\begin{align*}
    E_0^6=\Omega^6=span_{C^\infty(G)}\lbrace \theta_1\wedge\theta_2\wedge\theta_3\wedge\theta_4\wedge\theta_5\wedge\theta_6\rbrace\,.
\end{align*}

\subsubsection*{The Rumin differential on 0-forms}

Let us study the differential $d_c$ when applied to the space of Rumin 0-forms $E_0^0=\Omega^0=C^\infty(G)$. Let us then take $f\in C^\infty(G)$, we then have
\begin{align*}
    d_cf=&\Pi_{E_0}d\Pi_Ef=\Pi_{E_0}d(f-Qdf-dQf)=\Pi_{E_0}d(f-P\underbrace{d_\mathfrak{g}^{-1}df}_{=0}-dP\underbrace{d_\mathfrak{g}^{-1}f}_{=0})=\Pi_{E_0}df\\=&\Pi_{E_0}(X_1f\theta_1+X_2f\theta_2+X_3f\theta_3+X_4f\theta_4+X_5f\theta_5+X_6f\theta_6)\\=&(Id-d_\mathfrak{g}^{-1}d_\mathfrak{g}-d_\mathfrak{g}d_\mathfrak{g}^{-1})(X_1f\theta_1+X_2f\theta_2+X_3f\theta_3+X_4f\theta_4+X_5f\theta_5+X_6f\theta_6)\\=&X_1f\theta_1+X_2f\theta_2+X_3f\theta_3+X_4f\theta_4+X_5f\theta_5+X_6f\theta_6-d_\mathfrak{g}^{-1}d_\mathfrak{g}(X_3f\theta_3)-d_\mathfrak{g}^{-1}d_\mathfrak{g}(X_4f\theta_4)\\=&X_1f\theta_1+X_2f\theta_2+X_5f\theta_5+X_6f\theta_6
\end{align*}

\subsubsection*{The Rumin differential on 1-forms} Let us study the differential 
\begin{align*}
    d_c\colon E_0^{1}\to E_0^2\,.
\end{align*}
In order to do so, we will break the computations into steps.

Given an arbitrary Rumin 1-form $\alpha=f_1\theta_1+f_2\theta_2+f_3\theta_5+f_4\theta_6$, we have
\begin{itemize}
    \item let us compute $d_\mathfrak{g}^{-1}d\alpha$:
\end{itemize}
\begin{align*}
    d_\mathfrak{g}^{-1}d\alpha=&d_\mathfrak{g}^{-1}d(f_1\theta_1+f_2\theta_2+f_3\theta_5+f_4\theta_6)\\=&d_\mathfrak{g}^{-1}\big[-X_2f_1\theta_1\wedge\theta_2-X_3f_1\theta_1\wedge\theta_3+X_1f_2\theta_1\wedge\theta_2-X_6f_3\theta_5\wedge\theta_6+X_5f_4\theta_5\wedge\theta_6+(Im\,d_\mathfrak{g})^\perp \big]\\=&d_\mathfrak{g}^{-1}\big[(X_1f_2-X_2f_1)\theta_1\wedge\theta_2-X_3f_1\theta_1\wedge\theta_3+(X_5f_4-X_6f_3)\theta_5\wedge\theta_6+(Im\,d_\mathfrak{g})^\perp\big]\\=&-(X_1f_2-X_2f_1)\theta_3+\frac{X_3f_1-(X_5f_4-X_6f_3)}{2}\theta_4\,;
\end{align*}
\begin{itemize}
    \item let us compute $d_\mathfrak{g}^{-1}(d-d_\mathfrak{g})d_\mathfrak{g}^{-1}d\alpha$:
\end{itemize}
\begin{align*}
    d_\mathfrak{g}^{-1}(d-d_\mathfrak{g})d_\mathfrak{g}^{-1}d\alpha=&d_\mathfrak{g}^{-1}(d-d_\mathfrak{g})\big[-(X_1f_2-X_2f_1)\theta_3+\frac{X_3f_1-X_5f_4+X_6f_3}{2}\theta_4\big]\\=&d_\mathfrak{g}^{-1}\big[-X_1(X_1f_2-X_2f_1)\theta_1\wedge\theta_3+(Im\,d_\mathfrak{g})^\perp\big]=\frac{X_1(X_1f_2-X_2f_1)}{2}\theta_4\,;
\end{align*}
\begin{itemize}
    \item therefore, we have the following expression for $\Pi_E\alpha$:
\end{itemize}
\begin{align*}
    \Pi_E\alpha=&\alpha-d_\mathfrak{g}^{-1}d\alpha+d_\mathfrak{g}^{-1}(d-d_\mathfrak{g})d_\mathfrak{g}^{-1}d\alpha\\=&f_1\theta_1+f_2\theta_2+f_3\theta_5+f_4\theta_6+(X_1f_2-X_2f_1)\theta_3-\frac{X_3f_1-X_5f_4+X_6f_3}{2}\theta_4+\\&+\frac{X_1(X_1f_2-X_2f_1)}{2}\theta_4\\=&f_1\theta_1+f_2\theta_2+(X_1f_2-X_2f_1)\theta_3+\frac{X_1(X_1f_2-X_2f_1)-X_3f_1+X_5f_4-X_6f_3}{2}\theta_4+\\&+f_3\theta_5+f_4\theta_6\,;
\end{align*}
\begin{itemize}
    \item let us now compute $d\Pi_E\alpha$:
\end{itemize}
\begin{align*}
    d\Pi_E\alpha=&d\big[f_1\theta_1+f_2\theta_2+(X_1f_2-X_2f_1)\theta_3+\frac{X_1(X_1f_2-X_2f_1)-X_3f_1+X_5f_4-X_6f_3}{2}\theta_4+\\&+f_3\theta_5+f_4\theta_6\big]\\=&-X_2f_1\theta_1\wedge\theta_2-X_3f_1\theta_1\wedge\theta_3-X_4f_1\theta_1\wedge\theta_4-X_5f_1\theta_1\wedge\theta_5-X_6f_1\theta_1\wedge\theta_6+\\
    &+X_1f_2\theta_1\wedge\theta_2-X_3f_2\theta_2\wedge\theta_3-X_4f_2\theta_2\wedge\theta_4-X_5f_2\theta_2\wedge\theta_5-X_6f_2\theta_2\wedge\theta_6+\\
    &+X_1(X_1f_2-X_2f_1)\theta_1\wedge\theta_3+X_2(X_1f_2-X_2f_1)\theta_2\wedge\theta_3-X_4(X_1f_2-X_2f_1)\theta_3\wedge\theta_4+\\&-X_5(X_1f_2-X_2f_1)\theta_3\wedge\theta_5-X_6(X_1f_2-X_2f_1)\theta_3\wedge\theta_6+\\&+\frac{X_1[X_1(X_1f_2-X_2f_1)-X_3f_1+X_5f_4-X_6f_3]}{2}\theta_1\wedge\theta_4+\\&+\frac{X_2[X_1(X_1f_2-X_2f_1)-X_3f_1+X_5f_4-X_6f_3]}{2}\theta_2\wedge\theta_4+\\&+\frac{X_3[X_1(X_1f_2-X_2f_1)-X_3f_1+X_5f_4-X_6f_3]}{2}\theta_3\wedge\theta_4+\\&-\frac{X_5[X_1(X_1f_2-X_2f_1)-X_3f_1+X_5f_4-X_6f_3]}{2}\theta_4\wedge\theta_5+\\&-\frac{X_6[X_1(X_1f_2-X_2f_1)-X_3f_1+X_5f_4-X_6f_3]}{2}\theta_4\wedge\theta_6+X_1f_3\theta_1\wedge\theta_5+\\&+X_2f_3\theta_2\wedge\theta_5+X_3f_3\theta_3\wedge\theta_5+X_4f_3\theta_4\wedge\theta_5-X_6f_3\theta_5\wedge\theta_6+X_1f_4\theta_1\wedge\theta_6+\\&+X_2f_4\theta_2\wedge\theta_6+X_3f_4\theta_3\wedge\theta_6+X_4f_4\theta_4\wedge\theta_6+X_5f_4\theta_5\wedge\theta_6+(X_1f_2-X_2f_1)d_\mathfrak{g}\theta_3+\\&+\frac{X_1(X_1f_2-X_2f_1)-X_3f_1+X_5f_4-X_6f_3}{2}d_\mathfrak{g}\theta_4\\=&(X_1f_2-X_2f_1)\theta_1\wedge\theta_2+[X_1(X_1f_2-X_2f_1)-X_3f_1]\theta_1\wedge\theta_3+\\&+\big[\frac{X_1^2(X_1f_2-X_2f_1)-X_1X_3f_1+X_1X_5f_4-X_1X_6f_3}{2}-X_4f_1\big]\theta_1\wedge\theta_4+\\&+(X_1f_3-X_5f_1)\theta_1\wedge\theta_5+(X_1f_4-X_6f_1)\theta_1\wedge\theta_6+[X_2(X_1f_2-X_2f_1)-X_3f_2]\theta_2\wedge\theta_3+\\&+\big[\frac{X_2X_1(X_1f_2-X_2f_1)-X_2X_3f_1+X_2X_5f_4-X_2X_6f_3}{2}-X_4f_2\big]\theta_2\wedge\theta_4+\\&+(X_2f_3-X_5f_2)\theta_2\wedge\theta_5+(X_2f_4-X_6f_2)\theta_2\wedge\theta_6+\\&+\big[\frac{X_3X_1(X_1f_2-X_2f_1)-X_3^2f_1+X_3X_5f_4-X_3X_6f_3}{2}-X_4(X_1f_2-X_2f_1)\big]\theta_3\wedge\theta_4+\\&+[X_3f_3-X_5(X_1f_2-X_2f_1)]\theta_3\wedge\theta_5+[X_3f_4-X_6(X_1f_2-X_2f_1)]\theta_3\wedge\theta_6+\\&+\big[X_4f_3-\frac{X_5X_1(X_1f_2-X_2f_1)-X_5X_3f_1+X_5^2f_4-X_5X_6f_3}{2}\big]\theta_4\wedge\theta_5+\\&+\big[X_4f_4-\frac{X_6X_1(X_1f_2-X_2f_1)-X_6X_3f_1+X_6X_5f_4-X_6^2f_3}{2}\big]\theta_4\wedge\theta_6+\end{align*}
    \begin{align*}
    \phantom{d\Pi_E\alpha=}&+(X_5f_4-X_6f_3)\theta_5\wedge\theta_6+(X_1f_2-X_2f_1)d_\mathfrak{g}\theta_3+\\&+\frac{X_1(X_1f_2-X_2f_1)-X_3f_1+X_5f_4-X_6f_3}{2}d_\mathfrak{g}\theta_4\,;
\end{align*}
\begin{itemize}
    \item before finishing the computations, we will study the action of the projection $\Pi_{E_0}=Id-d_\mathfrak{g}^{-1}d_\mathfrak{g}-d_\mathfrak{g}d_\mathfrak{g}^{-1}$ on each left-invariant 2-form:
    \begin{itemize}
        \item $\Pi_{E_0}(\theta_1\wedge\theta_2)=\theta_1\wedge\theta_2-d_\mathfrak{g}^{-1}d_\mathfrak{g}(\theta_1\wedge\theta_2)-d_\mathfrak{g}d_\mathfrak{g}^{-1}(\theta_1\wedge\theta_2)=0$;
        \item $\Pi_{E_0}(\theta_1\wedge\theta_3)=\theta_1\wedge\theta_3-d_\mathfrak{g}^{-1}d_\mathfrak{g}(\theta_1\wedge\theta_3)-d_\mathfrak{g}d_\mathfrak{g}^{-1}(\theta_1\wedge\theta_3)=-\frac{\theta_5\wedge\theta_6-\theta_1\wedge\theta_3}{2}$;
        \item $\Pi_{E_0}(\theta_1\wedge\theta_4)=\theta_1\wedge\theta_4-d_\mathfrak{g}^{-1}d_\mathfrak{g}(\theta_1\wedge\theta_4)-d_\mathfrak{g}d_\mathfrak{g}^{-1}(\theta_1\wedge\theta_4)=0$;
         \item $\Pi_{E_0}(\theta_1\wedge\theta_5)=\theta_1\wedge\theta_5-d_\mathfrak{g}^{-1}d_\mathfrak{g}(\theta_1\wedge\theta_5)-d_\mathfrak{g}d_\mathfrak{g}^{-1}(\theta_1\wedge\theta_5)=\theta_1\wedge\theta_5$;
          \item $\Pi_{E_0}(\theta_1\wedge\theta_6)=\theta_1\wedge\theta_6-d_\mathfrak{g}^{-1}d_\mathfrak{g}(\theta_1\wedge\theta_6)-d_\mathfrak{g}d_\mathfrak{g}^{-1}(\theta_1\wedge\theta_6)=\theta_1\wedge\theta_6$;
           \item $\Pi_{E_0}(\theta_2\wedge\theta_3)=\theta_2\wedge\theta_3-d_\mathfrak{g}^{-1}d_\mathfrak{g}(\theta_2\wedge\theta_3)-d_\mathfrak{g}d_\mathfrak{g}^{-1}(\theta_2\wedge\theta_3)=\theta_2\wedge\theta_3$;
           \item $\Pi_{E_0}(\theta_2\wedge\theta_4)=\theta_2\wedge\theta_4-d_\mathfrak{g}^{-1}d_\mathfrak{g}(\theta_2\wedge\theta_4)-d_\mathfrak{g}d_\mathfrak{g}^{-1}(\theta_2\wedge\theta_4)=0$;
           \item $\Pi_{E_0}(\theta_2\wedge\theta_5)=\theta_2\wedge\theta_5-d_\mathfrak{g}^{-1}d_\mathfrak{g}(\theta_2\wedge\theta_5)-d_\mathfrak{g}d_\mathfrak{g}^{-1}(\theta_2\wedge\theta_5)=\theta_2\wedge\theta_5$;
           \item $\Pi_{E_0}(\theta_2\wedge\theta_6)=\theta_2\wedge\theta_6-d_\mathfrak{g}^{-1}d_\mathfrak{g}(\theta_2\wedge\theta_6)-d_\mathfrak{g}d_\mathfrak{g}^{-1}(\theta_2\wedge\theta_6)=\theta_2\wedge\theta_6$;
           \item $\Pi_{E_0}(\theta_3\wedge\theta_4)=\theta_3\wedge\theta_4-d_\mathfrak{g}^{-1}d_\mathfrak{g}(\theta_3\wedge\theta_4)-d_\mathfrak{g}d_\mathfrak{g}^{-1}(\theta_3\wedge\theta_4)=0$;
           \item $\Pi_{E_0}(\theta_3\wedge\theta_5)=\theta_3\wedge\theta_5-d_\mathfrak{g}^{-1}d_\mathfrak{g}(\theta_3\wedge\theta_5)-d_\mathfrak{g}d_\mathfrak{g}^{-1}(\theta_3\wedge\theta_5)=0$;
           \item $\Pi_{E_0}(\theta_3\wedge\theta_6)=\theta_3\wedge\theta_6-d_\mathfrak{g}^{-1}d_\mathfrak{g}(\theta_3\wedge\theta_6)-d_\mathfrak{g}d_\mathfrak{g}^{-1}(\theta_3\wedge\theta_6)=0$;
           \item $\Pi_{E_0}(\theta_4\wedge\theta_5)=\theta_4\wedge\theta_5-d_\mathfrak{g}^{-1}d_\mathfrak{g}(\theta_4\wedge\theta_5)-d_\mathfrak{g}d_\mathfrak{g}^{-1}(\theta_4\wedge\theta_5)=0$;
           \item $\Pi_{E_0}(\theta_4\wedge\theta_6)=\theta_4\wedge\theta_6-d_\mathfrak{g}^{-1}d_\mathfrak{g}(\theta_4\wedge\theta_6)-d_\mathfrak{g}d_\mathfrak{g}^{-1}(\theta_4\wedge\theta_6)=0$;
            \item $\Pi_{E_0}(\theta_5\wedge\theta_6)=\theta_5\wedge\theta_6-d_\mathfrak{g}^{-1}d_\mathfrak{g}(\theta_5\wedge\theta_6)-d_\mathfrak{g}d_\mathfrak{g}^{-1}(\theta_5\wedge\theta_6)=\frac{\theta_5\wedge\theta_6-\theta_1\wedge\theta_3}{2}$.
    \end{itemize}
    \item Finally, using all the previous computations, we can express explicitly the differential operator $d_c\alpha=\Pi_{E_0}d\Pi_E$, so that
    \begin{align*}
        d_c\alpha=&d_c(f_1\theta_1+f_2\theta_2+f_3\theta_5+f_4\theta_6)\\
        =&-[X_1(X_1f_2-X_2f_1)-X_3f_1]\frac{\theta_5\wedge\theta_6-\theta_1\wedge\theta_3}{2}+\\&+(X_1f_3-X_5f_1)\theta_1\wedge\theta_5+(X_1f_4-X_6f_1)\theta_1\wedge\theta_6+[X_2(X_1f_2-X_2f_1)-X_3f_2]\theta_2\wedge\theta_3+\\&+(X_2f_3-X_5f_2)\theta_2\wedge\theta_5+(X_2f_4-X_6f_2)\theta_2\wedge\theta_6+\\&+(X_5f_4-X_6f_3)\frac{\theta_5\wedge\theta_6-\theta_1\wedge\theta_3}{2}\\
        =&(X_1f_3-X_5f_1)\theta_1\wedge\theta_5+(X_1f_4-X_6f_1)\theta_1\wedge\theta_6+[X_2(X_1f_2-X_2f_1)-X_3f_2]\theta_2\wedge\theta_3+\\&+(X_2f_3-X_5f_2)\theta_2\wedge\theta_5+(X_2f_4-X_6f_2)\theta_2\wedge\theta_6+\\&+\frac{X_5f_4-X_6f_3-X_1(X_1f_2-X_2f_1)+X_3f_1}{2}(\theta_5\wedge\theta_6-\theta_1\wedge\theta_3)\,;
    \end{align*}
\end{itemize}


\subsubsection*{The Rumin differential on 2-forms} Let us study the differential 
\begin{align*}
    d_c\colon E_0^{2}\to E_0^3\,.
\end{align*}
In order to do so, we will break the computations into steps.

Given an arbitrary Rumin 2-form $\beta=g_1\theta_1\wedge\theta_5+g_2\theta_1\wedge\theta_6+g_3\theta_2\wedge\theta_5+g_4\theta_2\wedge\theta_6+g_5(\theta_5\wedge\theta_6-\theta_1\wedge\theta_3)+g_6\theta_2\wedge\theta_3$, we have
\begin{itemize}
    \item let us compute $d_\mathfrak{g}^{-1}d\beta$:
\end{itemize}
\begin{align*}
    d_\mathfrak{g}^{-1}d\beta=&d_\mathfrak{g}^{-1}d\big[g_1\theta_1\wedge\theta_5+g_2\theta_1\wedge\theta_6+g_3\theta_2\wedge\theta_5+g_4\theta_2\wedge\theta_6+g_5(\theta_5\wedge\theta_6-\theta_1\wedge\theta_3)+g_6\theta_2\wedge\theta_3\big]\\=&d_\mathfrak{g}^{-1}\big[-X_2g_1\theta_1\wedge\theta_2\wedge\theta_5-X_3g_1\theta_1\wedge\theta_3\wedge\theta_5+X_6g_1\theta_1\wedge\theta_5\wedge\theta_6-X_2g_2\theta_1\wedge\theta_2\wedge\theta_6+\\&-X_3g_2\theta_1\wedge\theta_3\wedge\theta_6-X_5g_2\theta_1\wedge\theta_5\wedge\theta_6+X_1g_3\theta_1\wedge\theta_2\wedge\theta_5+X_6g_3\theta_2\wedge\theta_5\wedge\theta_6+\\&+X_1g_4\theta_1\wedge\theta_2\wedge\theta_6-X_5g_4\theta_2\wedge\theta_5\wedge\theta_6+X_1g_5\theta_1\wedge\theta_5\wedge\theta_6+X_2g_5\theta_2\wedge\theta_5\wedge\theta_6+\\&+X_3g_5\theta_3\wedge\theta_5\wedge\theta_6+X_2g_5\theta_1\wedge\theta_2\wedge\theta_3-X_5g_5\theta_1\wedge\theta_3\wedge\theta_5-X_6g_5\theta_1\wedge\theta_3\wedge\theta_6+\\&+X_1g_6\theta_1\wedge\theta_2\wedge\theta_3+(Im\,d_\mathfrak{g})^\perp\big]\\
    =&d_\mathfrak{g}^{-1}\big[(X_1g_6+X_2g_5)\theta_1\wedge\theta_2\wedge\theta_3+(X_1g_3-X_2g_1)\theta_1\wedge\theta_2\wedge\theta_5+\\&+(X_1g_4-X_2g_2)\theta_1\wedge\theta_2\wedge\theta_6-(X_5g_5+X_3g_1)\theta_1\wedge\theta_3\wedge\theta_5-(X_3g_2+X_6g_5)\theta_1\wedge\theta_3\wedge\theta_6+\\&+(X_6g_1-X_5g_2+X_1g_5)\theta_1\wedge\theta_5\wedge\theta_6+(X_6g_3+X_2g_5-X_5g_4)\theta_2\wedge\theta_5\wedge\theta_6+\\&+X_3g_5\theta_3\wedge\theta_5\wedge\theta_6+(Im\,d_\mathfrak{g})^\perp\big]
     \end{align*}
    \begin{align*}
    \phantom{d_\mathfrak{g}^{-1}d\beta}
    =&-\frac{X_1g_6+X_2g_5}{2}\theta_2\wedge\theta_4-(X_1g_3-X_2g_1)\theta_3\wedge\theta_5-(X_1g_4-X_2g_2)\theta_3\wedge\theta_6+\\&+(X_5g_5+X_3g_1)\theta_4\wedge\theta_5+(X_3g_2+X_6g_5)\theta_4\wedge\theta_6+(X_6g_1-X_5g_2+X_1g_5)\theta_1\wedge\theta_4+\\& \frac{X_6g_3+X_2g_5-X_5g_4}{2}\theta_2\wedge\theta_4+\frac{X_3g_5}{2}\theta_3\wedge\theta_4\\
    =&-(X_1g_3-X_2g_1)\theta_3\wedge\theta_5-(X_1g_4-X_2g_2)\theta_3\wedge\theta_6+\\&+(X_5g_5+X_3g_1)\theta_4\wedge\theta_5+(X_3g_2+X_6g_5)\theta_4\wedge\theta_6+(X_6g_1-X_5g_2+X_1g_5)\theta_1\wedge\theta_4+\\& +\frac{X_6g_3+X_2g_5-X_5g_4-X_1g_6-X_2g_5}{2}\theta_2\wedge\theta_4+\frac{X_3g_5}{2}\theta_3\wedge\theta_4
    \,;
\end{align*}
\begin{itemize}
    \item let us compute $d_\mathfrak{g}^{-1}(d-d_\mathfrak{g})d_\mathfrak{g}^{-1}d\beta$:
\end{itemize}
\begin{align*}
    d_\mathfrak{g}^{-1}(d-d_\mathfrak{g})d_\mathfrak{g}^{-1}d\beta=&d_\mathfrak{g}^{-1}(d-d_\mathfrak{g})\big[-(X_1g_3-X_2g_1)\theta_3\wedge\theta_5-(X_1g_4-X_2g_2)\theta_3\wedge\theta_6+\\&+(X_5g_5+X_3g_1)\theta_4\wedge\theta_5+(X_3g_2+X_6g_5)\theta_4\wedge\theta_6+\\&+(X_6g_1-X_5g_2+X_1g_5)\theta_1\wedge\theta_4+\\&+ \frac{X_6g_3+X_2g_5-X_5g_4-X_1g_6-X_2g_5}{2}\theta_2\wedge\theta_4+\frac{X_3g_5}{2}\theta_3\wedge\theta_4\big]\\=&d_\mathfrak{g}^{-1}\big[
    -X_1(X_1g_3-X_2g_1)\theta_1\wedge\theta_3\wedge\theta_5-X_6(X_1g_3-X_2g_1)\theta_3\wedge\theta_5\wedge\theta_6+\\&-X_1(X_1g_4-X_2g_2)\theta_1\wedge\theta_3\wedge\theta_6+X_5(X_1g_4-X_2g_2)\theta_3\wedge\theta_5\wedge\theta_6+\\&-X_2(X_6g_1-X_5g_2+X_1g_5)\theta_1\wedge\theta_2\wedge\theta_4+\\&+\frac{X_1(X_6g_3+X_2g_5-X_5g_4-X_1g_6-X_2g_5)}{2}\theta_1\wedge\theta_2\wedge\theta_4+(Im\,d_\mathfrak{g})^\perp\big]\\=&X_1(X_1g_3-X_2g_1)\theta_4\wedge\theta_5-\frac{X_6(X_1g_3-X_2g_1)}{2}\theta_3\wedge\theta_4+\\&+X_1(X_1g_4-X_2g_2)\theta_4\wedge\theta_6+\frac{X_5(X_1g_4-X_2g_2)}{2}\theta_3\wedge\theta_4+\\&+\frac{X_2(X_6g_1-X_5g_2+X_1g_5)}{2}\theta_3\wedge\theta_4+\\&-\frac{X_1(X_6g_3+X_2g_5-X_5g_4-X_1g_6-X_2g_5)}{4}\theta_3\wedge\theta_4\\=&X_1(X_1g_3-X_2g_1)\theta_4\wedge\theta_5+X_1(X_1g_4-X_2g_2)\theta_4\wedge\theta_6+A\theta_3\wedge\theta_4\,;
\end{align*}
where
\begin{align*}
    A=&-\frac{X_6(X_1g_3-X_2g_1)}{2}+\frac{X_5(X_1g_4-X_2g_2)}{2}-\frac{X_1(X_6g_3+X_2g_5-X_5g_4-X_1g_6-X_2g_5)}{4}+\\&+\frac{X_2(X_6g_1-X_5g_2+X_1g_5)}{2}
\end{align*}
\begin{itemize}
    \item therefore, we have the following expression for $\Pi_E\beta$:
\end{itemize}
\begin{align*}
    \Pi_E\beta=&\beta-d_\mathfrak{g}^{-1}d\beta+d_\mathfrak{g}^{-1}(d-d_\mathfrak{g})d_\mathfrak{g}^{-1}d\beta\\=&g_1\theta_1\wedge\theta_5+g_2\theta_1\wedge\theta_6+g_3\theta_2\wedge\theta_5+g_4\theta_2\wedge\theta_6+g_5(\theta_5\wedge\theta_6-\theta_1\wedge\theta_3)+g_6\theta_2\wedge\theta_3+\\&+(X_1g_3-X_2g_1)\theta_3\wedge\theta_5+(X_1g_4-X_2g_2)\theta_3\wedge\theta_6+\\&-(X_5g_5+X_3g_1)\theta_4\wedge\theta_5-(X_3g_2+X_6g_5)\theta_4\wedge\theta_6+(X_6g_1-X_5g_2+X_1g_5)\theta_1\wedge\theta_4+\\& -\frac{X_6g_3+X_2g_5-X_5g_4-X_1g_6-X_2g_5}{2}\theta_2\wedge\theta_4-\frac{X_3g_5}{2}\theta_3\wedge\theta_4+\\&+X_1(X_1g_3-X_2g_1)\theta_4\wedge\theta_5+X_1(X_1g_4-X_2g_2)\theta_4\wedge\theta_6+A\theta_3\wedge\theta_4\\=&g_1\theta_1\wedge\theta_5+g_2\theta_1\wedge\theta_6+g_3\theta_2\wedge\theta_5+g_4\theta_2\wedge\theta_6+g_5(\theta_5\wedge\theta_6-\theta_1\wedge\theta_3)+g_6\theta_2\wedge\theta_3+\\&+(X_1g_3-X_2g_1)\theta_3\wedge\theta_5+(X_1g_4-X_2g_2)\theta_3\wedge\theta_6+(A-\frac{X_3g_5}{2})\theta_3\wedge\theta_4+\\&+[X_1(X_1g_3-X_2g_1)-X_5g_5-X_3g_1]\theta_4\wedge\theta_5+\\&+[X_1(X_1g_4-X_2g_2)-X_3g_2-X_6g_5]\theta_4\wedge\theta_6+\\&+(X_6g_1-X_5g_2+X_1g_5)\theta_1\wedge\theta_4 -\frac{X_6g_3+X_2g_5-X_5g_4-X_1g_6-X_2g_5}{2}\theta_2\wedge\theta_4;
\end{align*}
\begin{itemize}
    \item let us now compute $d\Pi_E\beta$:
\end{itemize}
\begin{align*}
    d\Pi_E\beta=&d\big[g_1\theta_1\wedge\theta_5+g_2\theta_1\wedge\theta_6+g_3\theta_2\wedge\theta_5+g_4\theta_2\wedge\theta_6+g_5(\theta_5\wedge\theta_6-\theta_1\wedge\theta_3)+g_6\theta_2\wedge\theta_3+\\&+(X_1g_3-X_2g_1)\theta_3\wedge\theta_5+(X_1g_4-X_2g_2)\theta_3\wedge\theta_6+(A-\frac{X_3g_5}{2})\theta_3\wedge\theta_4+\\&+[X_1(X_1g_3-X_2g_1)-X_5g_5-X_3g_1]\theta_4\wedge\theta_5+\\&+[X_1(X_1g_4-X_2g_2)-X_3g_2-X_6g_5]\theta_4\wedge\theta_6+\\&+(X_6g_1-X_5g_2+X_1g_5)\theta_1\wedge\theta_4 -\frac{X_6g_3+X_2g_5-X_5g_4-X_1g_6-X_2g_5}{2}\theta_2\wedge\theta_4\big]\\
    &-X_2g_1\theta_1\wedge\theta_2\wedge\theta_5-X_3g_1\theta_1\wedge\theta_3\wedge\theta_5-X_4g_1\theta_1\wedge\theta_4\wedge\theta_5+X_6g_1\theta_1\wedge\theta_5\wedge\theta_6+\\&-X_2g_2\theta_1\wedge\theta_2\wedge\theta_6-X_3g_2\theta_1\wedge\theta_3\wedge\theta_6-X_4g_2\theta_1\wedge\theta_4\wedge\theta_6-X_5g_2\theta_1\wedge\theta_5\wedge\theta_6+\\&+X_1g_3\theta_1\wedge\theta_2\wedge\theta_5-X_3g_3\theta_2\wedge\theta_3\wedge\theta_5-X_4g_3\theta_2\wedge\theta_4\wedge\theta_5+X_6g_3\theta_2\wedge\theta_5\wedge\theta_6+\\&+X_1g_4\theta_1\wedge\theta_2\wedge\theta_6-X_3g_4\theta_2\wedge\theta_3\wedge\theta_6-X_4g_4\theta_2\wedge\theta_4\wedge\theta_6-X_5g_4\theta_2\wedge\theta_5\wedge\theta_6+\\&+X_1g_5\theta_1\wedge\theta_5\wedge\theta_6+X_2g_5\theta_2\wedge\theta_5\wedge\theta_6+X_3g_5\theta_3\wedge\theta_5\wedge\theta_6+X_4g_5\theta_4\wedge\theta_5\wedge\theta_6+\\&+X_2g_5\theta_1\wedge\theta_2\wedge\theta_3-X_4g_5\theta_1\wedge\theta_3\wedge\theta_4-X_5g_5\theta_1\wedge\theta_3\wedge\theta_5-X_6g_5\theta_1\wedge\theta_3\wedge\theta_6+\\&+X_1g_6\theta_1\wedge\theta_2\wedge\theta_3+X_4g_6\theta_2\wedge\theta_3\wedge\theta_4+X_5g_6\theta_2\wedge\theta_3\wedge\theta_5+X_6g_6\theta_2\wedge\theta_3\wedge\theta_6+\\&+X_1(X_1g_3-X_2g_1)\theta_1\wedge\theta_3\wedge\theta_5+X_2(X_1g_3-X_2g_1)\theta_2\wedge\theta_3\wedge\theta_5+\\&-X_4(X_1g_3-X_2g_1)\theta_3\wedge\theta_4\wedge\theta_5+X_6(X_1g_3-X_2g_1)\theta_3\wedge\theta_5\wedge\theta_6+\\&+X_1(X_1g_4-X_2g_2)\theta_1\wedge\theta_3\wedge\theta_6+X_2(X_1g_4-X_2g_2)\theta_2\wedge\theta_3\wedge\theta_6+\\&-X_4(X_1g_4-X_2g_2)\theta_3\wedge\theta_4\wedge\theta_6-X_5(X_1g_4-X_2g_2)\theta_3\wedge\theta_5\wedge\theta_6\\&+X_1(A-\frac{X_3g_5}{2})\theta_1\wedge\theta_3\wedge\theta_4+X_2(A-\frac{X_3g_5}{2})\theta_2\wedge\theta_3\wedge\theta_4+X_5(A-\frac{X_3g_5}{2})\theta_3\wedge\theta_4\wedge\theta_5+\\&+X_6(A-\frac{X_3g_5}{2})\theta_3\wedge\theta_4\wedge\theta_6+X_1[X_1(X_1g_3-X_2g_1)-X_5g_5-X_3g_1]\theta_1\wedge\theta_4\wedge\theta_5+\\&+X_2[X_1(X_1g_3-X_2g_1)-X_5g_5-X_3g_1]\theta_2\wedge\theta_4\wedge\theta_5+\\&+X_3[X_1(X_1g_3-X_2g_1)-X_5g_5-X_3g_1]\theta_3\wedge\theta_4\wedge\theta_5+\\&+X_6[X_1(X_1g_3-X_2g_1)-X_5g_5-X_3g_1]\theta_4\wedge\theta_5\wedge\theta_6+\\&+X_1[X_1(X_1g_4-X_2g_2)-X_3g_2-X_6g_5]\theta_1\wedge\theta_4\wedge\theta_6+\\&+X_2[X_1(X_1g_4-X_2g_2)-X_3g_2-X_6g_5]\theta_2\wedge\theta_4\wedge\theta_6+\\&
    +X_3[X_1(X_1g_4-X_2g_2)-X_3g_2-X_6g_5]\theta_3\wedge\theta_4\wedge\theta_6+\\&-X_5[X_1(X_1g_4-X_2g_2)-X_3g_2-X_6g_5]\theta_4\wedge\theta_5\wedge\theta_6+\\&-X_2(X_6g_1-X_5g_2+X_1g_5)\theta_1\wedge\theta_2\wedge\theta_4-X_3(X_6g_1-X_5g_2+X_1g_5)\theta_1\wedge\theta_3\wedge\theta_4+\\&+X_5(X_6g_1-X_5g_2+X_1g_5)\theta_1\wedge\theta_4\wedge\theta_5+X_6(X_6g_1-X_5g_2+X_1g_5)\theta_1\wedge\theta_4\wedge\theta_6+\\&-\frac{X_1(X_6g_3+X_2g_5-X_5g_4-X_1g_6-X_2g_5)}{2}\theta_1\wedge\theta_2\wedge\theta_4+\end{align*}
    \begin{align*}
        \phantom{ d\Pi_E\beta=}&
    +\frac{X_3(X_6g_3+X_2g_5-X_5g_4-X_1g_6-X_2g_5)}{2}\theta_2\wedge\theta_3\wedge\theta_4+\phantom{blablablalallalablblblblbblbl}\\&-\frac{X_5(X_6g_3+X_2g_5-X_5g_4-X_1g_6-X_2g_5)}{2}\theta_2\wedge\theta_4\wedge\theta_5+\\&-\frac{X_6(X_6g_3+X_2g_5-X_5g_4-X_1g_6-X_2g_5)}{2}\theta_2\wedge\theta_4\wedge\theta_6
    \end{align*}

\begin{itemize}
    \item before finishing the computations, we will study the action of the projection $\Pi_{E_0}=Id-d_\mathfrak{g}^{-1}d_\mathfrak{g}-d_\mathfrak{g}d_\mathfrak{g}^{-1}$ on each left-invariant 3-form:
    \begin{itemize}
        \item $\Pi_{E_0}(\theta_1\wedge\theta_2\wedge\theta_3)=\theta_1\wedge\theta_2\wedge\theta_3-d_\mathfrak{g}^{-1}d_\mathfrak{g}(\theta_1\wedge\theta_2\wedge\theta_3)-d_\mathfrak{g}d_\mathfrak{g}^{-1}(\theta_1\wedge\theta_2\wedge\theta_3)=\frac{\theta_1\wedge\theta_2\wedge\theta_3+\theta_2\wedge\theta_5\wedge\theta_6}{2}$;
        \item $\Pi_{E_0}(\theta_1\wedge\theta_2\wedge\theta_4)=\theta_1\wedge\theta_2\wedge\theta_4-d_\mathfrak{g}^{-1}d_\mathfrak{g}(\theta_1\wedge\theta_2\wedge\theta_4)-d_\mathfrak{g}d_\mathfrak{g}^{-1}(\theta_1\wedge\theta_2\wedge\theta_4)=0$;
        \item $\Pi_{E_0}(\theta_1\wedge\theta_2\wedge\theta_5)=\theta_1\wedge\theta_2\wedge\theta_5-d_\mathfrak{g}^{-1}d_\mathfrak{g}(\theta_1\wedge\theta_2\wedge\theta_5)-d_\mathfrak{g}d_\mathfrak{g}^{-1}(\theta_1\wedge\theta_2\wedge\theta_5)=0$;
        \item $\Pi_{E_0}(\theta_1\wedge\theta_2\wedge\theta_6)=\theta_1\wedge\theta_2\wedge\theta_6-d_\mathfrak{g}^{-1}d_\mathfrak{g}(\theta_1\wedge\theta_2\wedge\theta_6)-d_\mathfrak{g}d_\mathfrak{g}^{-1}(\theta_1\wedge\theta_2\wedge\theta_6)=0$;
        \item $\Pi_{E_0}(\theta_1\wedge\theta_3\wedge\theta_4)=\theta_1\wedge\theta_3\wedge\theta_4-d_\mathfrak{g}^{-1}d_\mathfrak{g}(\theta_1\wedge\theta_3\wedge\theta_4)-d_\mathfrak{g}d_\mathfrak{g}^{-1}(\theta_1\wedge\theta_3\wedge\theta_4)=\frac{\theta_1\wedge\theta_3\wedge\theta_4-\theta_4\wedge\theta_5\wedge\theta_6}{2}$;
         \item $\Pi_{E_0}(\theta_1\wedge\theta_3\wedge\theta_5)=\theta_1\wedge\theta_3\wedge\theta_5-d_\mathfrak{g}^{-1}d_\mathfrak{g}(\theta_1\wedge\theta_3\wedge\theta_5)-d_\mathfrak{g}d_\mathfrak{g}^{-1}(\theta_1\wedge\theta_3\wedge\theta_5)=0$;
          \item $\Pi_{E_0}(\theta_1\wedge\theta_3\wedge\theta_6)=\theta_1\wedge\theta_3\wedge\theta_6-d_\mathfrak{g}^{-1}d_\mathfrak{g}(\theta_1\wedge\theta_3\wedge\theta_6)-d_\mathfrak{g}d_\mathfrak{g}^{-1}(\theta_1\wedge\theta_3\wedge\theta_6)=0$;
        \item $\Pi_{E_0}(\theta_1\wedge\theta_4\wedge\theta_5)=\theta_1\wedge\theta_4\wedge\theta_5-d_\mathfrak{g}^{-1}d_\mathfrak{g}(\theta_1\wedge\theta_4\wedge\theta_5)-d_\mathfrak{g}d_\mathfrak{g}^{-1}(\theta_1\wedge\theta_4\wedge\theta_5)=\theta_1\wedge\theta_4\wedge\theta_5$;
          \item $\Pi_{E_0}(\theta_1\wedge\theta_4\wedge\theta_6)=\theta_1\wedge\theta_4\wedge\theta_6-d_\mathfrak{g}^{-1}d_\mathfrak{g}(\theta_1\wedge\theta_4\wedge\theta_6)-d_\mathfrak{g}d_\mathfrak{g}^{-1}(\theta_1\wedge\theta_4\wedge\theta_6)=\theta_1\wedge\theta_4\wedge\theta_6$;
         \item $\Pi_{E_0}(\theta_1\wedge\theta_5\wedge\theta_6)=\theta_1\wedge\theta_5\wedge\theta_6-d_\mathfrak{g}^{-1}d_\mathfrak{g}(\theta_1\wedge\theta_5\wedge\theta_6)-d_\mathfrak{g}d_\mathfrak{g}^{-1}(\theta_1\wedge\theta_5\wedge\theta_6)=0$;
           \item $\Pi_{E_0}(\theta_2\wedge\theta_3\wedge\theta_4)=\theta_2\wedge\theta_3\wedge\theta_4-d_\mathfrak{g}^{-1}d_\mathfrak{g}(\theta_2\wedge\theta_3\wedge\theta_4)-d_\mathfrak{g}d_\mathfrak{g}^{-1}(\theta_2\wedge\theta_3\wedge\theta_4)=0$;
            \item $\Pi_{E_0}(\theta_2\wedge\theta_3\wedge\theta_5)=\theta_2\wedge\theta_3\wedge\theta_5-d_\mathfrak{g}^{-1}d_\mathfrak{g}(\theta_2\wedge\theta_3\wedge\theta_5)-d_\mathfrak{g}d_\mathfrak{g}^{-1}(\theta_2\wedge\theta_3\wedge\theta_5)=\theta_2\wedge\theta_3\wedge\theta_5$;
             \item $\Pi_{E_0}(\theta_2\wedge\theta_3\wedge\theta_6)=\theta_2\wedge\theta_3\wedge\theta_6-d_\mathfrak{g}^{-1}d_\mathfrak{g}(\theta_2\wedge\theta_3\wedge\theta_6)-d_\mathfrak{g}d_\mathfrak{g}^{-1}(\theta_2\wedge\theta_3\wedge\theta_6)=\theta_2\wedge\theta_3\wedge\theta_6$;
           \item $\Pi_{E_0}(\theta_2\wedge\theta_4\wedge\theta_5)=\theta_2\wedge\theta_4\wedge\theta_5-d_\mathfrak{g}^{-1}d_\mathfrak{g}(\theta_2\wedge\theta_4\wedge\theta_5)-d_\mathfrak{g}d_\mathfrak{g}^{-1}(\theta_2\wedge\theta_4\wedge\theta_5)=0$;
            \item $\Pi_{E_0}(\theta_2\wedge\theta_4\wedge\theta_6)=\theta_2\wedge\theta_4\wedge\theta_6-d_\mathfrak{g}^{-1}d_\mathfrak{g}(\theta_2\wedge\theta_4\wedge\theta_6)-d_\mathfrak{g}d_\mathfrak{g}^{-1}(\theta_2\wedge\theta_4\wedge\theta_6)=0$;
           \item $\Pi_{E_0}(\theta_2\wedge\theta_5\wedge\theta_6)=\theta_2\wedge\theta_5\wedge\theta_6-d_\mathfrak{g}^{-1}d_\mathfrak{g}(\theta_2\wedge\theta_5\wedge\theta_6)-d_\mathfrak{g}d_\mathfrak{g}^{-1}(\theta_2\wedge\theta_5\wedge\theta_6)=\frac{\theta_1\wedge\theta_2\wedge\theta_3+\theta_2\wedge\theta_5\wedge\theta_6}{2}$;
           \item $\Pi_{E_0}(\theta_3\wedge\theta_4\wedge\theta_5)=\theta_3\wedge\theta_4\wedge\theta_5-d_\mathfrak{g}^{-1}d_\mathfrak{g}(\theta_3\wedge\theta_4\wedge\theta_5)-d_\mathfrak{g}d_\mathfrak{g}^{-1}(\theta_3\wedge\theta_4\wedge\theta_5)=0$;
           \item $\Pi_{E_0}(\theta_3\wedge\theta_4\wedge\theta_6)=\theta_3\wedge\theta_4\wedge\theta_6-d_\mathfrak{g}^{-1}d_\mathfrak{g}(\theta_3\wedge\theta_4\wedge\theta_6)-d_\mathfrak{g}d_\mathfrak{g}^{-1}(\theta_3\wedge\theta_4\wedge\theta_6)=0$;
           \item $\Pi_{E_0}(\theta_3\wedge\theta_5\wedge\theta_6)=\theta_3\wedge\theta_5\wedge\theta_6-d_\mathfrak{g}^{-1}d_\mathfrak{g}(\theta_3\wedge\theta_5\wedge\theta_6)-d_\mathfrak{g}d_\mathfrak{g}^{-1}(\theta_3\wedge\theta_5\wedge\theta_6)=0$;
           \item $\Pi_{E_0}(\theta_4\wedge\theta_5\wedge\theta_6)=\theta_4\wedge\theta_5\wedge\theta_6-d_\mathfrak{g}^{-1}d_\mathfrak{g}(\theta_4\wedge\theta_5\wedge\theta_6)-d_\mathfrak{g}d_\mathfrak{g}^{-1}(\theta_4\wedge\theta_5\wedge\theta_6)=-\frac{\theta_1\wedge\theta_3\wedge\theta_4-\theta_4\wedge\theta_5\wedge\theta_6}{2}$.
    \end{itemize}
    \item Finally, using all the previous computations, we can express explicitly the differential operator $d_c\beta=\Pi_{E_0}d\Pi_E\beta$, so that
    \begin{align*}
        d_c\beta=&d_c(g_1\theta_1\wedge\theta_5+g_2\theta_1\wedge\theta_6+g_3\theta_2\wedge\theta_5+g_4\theta_2\wedge\theta_6+g_5(\theta_5\wedge\theta_6-\theta_1\wedge\theta_3)+g_6\theta_2\wedge\theta_3)\\
        =&-X_4g_1\theta_1\wedge\theta_4\wedge\theta_5-X_4g_2\theta_1\wedge\theta_4\wedge\theta_6-X_3g_3\theta_2\wedge\theta_3\wedge\theta_5-X_3g_4\theta_2\wedge\theta_3\wedge\theta_6+\\&+(X_6g_3-X_5g_4+X_2g_5)\frac{\theta_1\wedge\theta_2\wedge\theta_3+\theta_2\wedge\theta_5\wedge\theta_6}{2}-X_4g_5\frac{\theta_1\wedge\theta_3\wedge\theta_4-\theta_4\wedge\theta_5\wedge\theta_6}{2}+\\&+(X_2g_5+X_1g_6)\frac{\theta_1\wedge\theta_2\wedge\theta_3+\theta_2\wedge\theta_5\wedge\theta_6}{2}-X_4g_5\frac{\theta_1\wedge\theta_3\wedge\theta_4-\theta_4\wedge\theta_5\wedge\theta_6}{2}+\\&+X_5g_6\theta_2\wedge\theta_3\wedge\theta_5+X_6g_6\theta_2\wedge\theta_3\wedge\theta_6+X_2(X_1g_3-X_2g_1)\theta_2\wedge\theta_3\wedge\theta_5+\\&+X_2(X_1g_4-X_2g_2)\theta_2\wedge\theta_3\wedge\theta_6+X_1(A-\frac{X_3g_5}{2})\frac{\theta_1\wedge\theta_3\wedge\theta_4-\theta_4\wedge\theta_5\wedge\theta_6}{2}+\\&+X_1[X_1(X_1g_3-X_2g_1)-X_5g_5-X_3g_1]\theta_1\wedge\theta_4\wedge\theta_5+\\&-X_6[X_1(X_1g_3-X_2g_1)-X_5g_5-X_3g_1]\frac{\theta_1\wedge\theta_3\wedge\theta_4-\theta_4\wedge\theta_5\wedge\theta_6}{2}+\\&+X_1[X_1(X_1g_4-X_2g_2)-X_3g_2-X_6g_5]\theta_1\wedge\theta_4\wedge\theta_6+\\&+X_5[X_1(X_1g_4-X_2g_2)-X_3g_2-X_6g_5]\frac{\theta_1\wedge\theta_3\wedge\theta_4-\theta_4\wedge\theta_5\wedge\theta_6}{2}+\\&-X_3(X_6g_1-X_5g_2+X_1g_5)\frac{\theta_1\wedge\theta_3\wedge\theta_4-\theta_4\wedge\theta_5\wedge\theta_6}{2}+\\&+X_5(X_6g_1-X_5g_2+X_1g_5)\theta_1\wedge\theta_4\wedge\theta_5+X_6(X_6g_1-X_5g_2+X_1g_5)\theta_1\wedge\theta_4\wedge\theta_6\end{align*}
    \begin{align*}
        \phantom{ =}
        =&(X_6g_3-X_5g_4+X_2g_5+X_2g_5+X_1g_6)\frac{\theta_1\wedge\theta_2\wedge\theta_3+\theta_2\wedge\theta_5\wedge\theta_6}{2}+\\&+[X_5g_6+X_1(X_1g_3-X_2g_1)-X_3g_3]\theta_2\wedge\theta_3\wedge\theta_5+\\&+[X_6g_6+X_2(X_1g_4-X_2g_2)-X_3g_4]\theta_2\wedge\theta_3\wedge\theta_6+\\&+[X_1^2(X_1g_3-X_2g_1)-X_1X_5g_5-X_1X_3g_1-X_4g_1+X_5(X_6g_1-X_5g_2+X_1g_5)]\theta_1\wedge\theta_4\wedge\theta_5+\\&+[X_1^2(X_1g_4-X_2g_2)-X_1X_3g_2-X_1X_6g_5-X_4g_2+X_6(X_6g_1-X_5g_2+X_1g_5)]\theta_1\wedge\theta_4\wedge\theta_6+\\&+[X_1(A-\frac{X_3g_5}{2})-X_4g_5-X_4g_5-X_6X_1(X_1g_3-X_2g_1)-X_6X_5g_5-X_6X_3g_1+\\&+X_5X_1(X_1g_4-X_2g_2)-X_5X_3g_2-X_5X_6g_5-X_3(X_6g_1-X_5g_2+X_1g_5)]\\&\frac{\theta_1\wedge\theta_3\wedge\theta_4-\theta_4\wedge\theta_5\wedge\theta_6}{2}.
    \end{align*}
\end{itemize}


\section{Applications to $\ell^{q,p}$ cohomology}

By definition, the $\ell^{q,p}$ cohomology of a Riemannian manifold with bounded geometry is the $\ell^{q,p}$ cohomology of every bounded geometry simplicial complex quasi-isometric to it. A crucial result of \cite{pansu-rumin} shows that in the case of manifolds with bounded geometry, and in particular contractible Lie groups, the $\ell^{q,p}$ cohomology of the manifold is isomorphic to its $L^{q,p}$ cohomology.

\begin{defin}
Let $G$ be a nilpotent Lie group, then the $L^{q,p}$ cohomology is defined as
\begin{align*}
    L^{q,p}H^\bullet(G)=\frac{\lbrace \text{closed forms in }L^p\rbrace}{d\big(\lbrace \text{forms in }L^q\rbrace\big)\cap L^p}\,.
\end{align*}
\end{defin}

In \cite{pansu-rumin} it is shown that the Rumin complex constructed on a Carnot group enables sharper estimates for vanishing and non-vanishing results of $L^{q,p}H^\bullet(G)$, and hence of $\ell^{q,p}H^\bullet(G)$. In particular the following theorem is presented in \cite{pansu-rumin}:

\vspace{0.3cm}
\textbf{Theorem 1.1} \textit{Let $G$ be a Carnot group of dimension $n$ and homogeneous dimension $Q$ and take $k=1,\ldots,n$. Let us define the numbers $\delta N_{max}(k)$ and $\delta N_{min}(k)$ as follows:}
\begin{align*}
    \delta N_{max}(k)=&\textit{maximal weight of Rumin k-forms}-\textit{minimal weight of Rumin ($k-1$)-forms}\,;\\
    \delta N_{min}(k)=&\textit{minimal weight of Rumin k-forms}-\textit{maximal weight of Rumin ($k-1$)-forms}\,;
\end{align*}
\textit{Then the following holds:}
\begin{itemize}
    \item [\textit{i.}] \textit{if}
    \begin{align*}
        1<p,q<\infty\textit{ and } \frac{1}{p}-\frac{1}{q}\ge\frac{\delta N_{max}(k)}{Q}
    \end{align*}
    \textit{then the $L^{q,p}$ cohomology of $G$ in degree $k$ vanishes.}
     \item [\textit{ii.}] \textit{if}
    \begin{align*}
        1\le p,q\le\infty\textit{ and } \frac{1}{p}-\frac{1}{q}<\frac{\delta N_{min}(k)}{Q}
    \end{align*}
    \textit{then the $L^{q,p}$ cohomology of $G$ in degree $k$ does not vanish.}
\end{itemize}
\vspace{0.3cm}

It is also shown in Theorem 9.2 of \cite{pansu-rumin} how this final non-vanishing statement has a wider scope, as it holds in particular on Carnot groups equipped with a homogeneous structure that gives only a positive grading. 

Since it is now possible to construct the Rumin complex on arbitrary nilpotent Lie groups, we are able to directly apply Theorem 9.2 in \cite{pansu-rumin} to arbitrary homogeneous nilpotent Lie groups that only admit positive gradings (and are not stratifiable).
\begin{prop}\label{teoteo}\textbf{Non-vanishing result for positively gradable nilpotent Lie groups}\\
Let $G$ be a homogeneous Lie group with homogeneous dimension $T$, then $\ell^{q,p}H^k(G)\neq 0$ provided
\begin{align}\label{proposizione}
    1\le p,q\le\infty\text{ and  }\,\frac{1}{p}-\frac{1}{q}<\frac{\delta N_{min}(k)}{T}\,,
\end{align}
where this time $\delta N_{min}(k)$ is computed using the weights given by the positive grading considered.
\end{prop}

\begin{ese}
Let us take into consideration the nilpotent Lie group $G$ whose Lie algebra is denoted as $N_{6,3,2}$ in \cite{Gong_Thesis} (see Section \ref{Esempio non-Carnot}), whose non-trivial brackets are the following:
\begin{align*}
    [X_1,X_2]=X_3\;,\;[X_1,X_3]=[X_5,X_6]=X_4\;.
\end{align*}

This time we are not interested in the asymptotic weights of forms. On the contrary, since this Lie group is positively gradable, we are interested in those weights that stem from the homogeneous structure given by the positive grading.

We will now consider the following two positive gradings:
\begin{align*}
    \mathcal{V}^1&=\lbrace V^1_i\rbrace_{i=1}^4\text{ with } V_1^1=\span\lbrace X_1\rbrace\,,\,V_2^1=\span\lbrace X_2,X_5,X_6\rbrace\,,\,V^1_3=\span\lbrace X_3\rbrace\,,\,V_4^1=\span\lbrace X_4\rbrace\,;\\
    \mathcal{V}^2&=\lbrace V^2_i\rbrace_{i=1}^3\text{ with } V_1^2=\span\lbrace X_1,X_2,X_5\rbrace\,,\,V_2^2=\span\lbrace X_3,X_6\rbrace\,,\,V^2_3=\span\lbrace X_4\rbrace\,.
\end{align*}
One can check that these are indeed positive gradings for $N_{6,3,2}$.

Let us study all the different options of weights for 1-forms on $G$:
\begin{center}
 \begin{tabular}{c| c| c| c} 
  & Aymptotic weight& Weight from $\mathcal{V}_1$ & Weight from $\mathcal{V}_2$ \\ [0.5ex] 
 \hline\hline
 $\theta_1$ & 1 & 1 & 1 \\ 
 \hline
 $\theta_2$ & 1 & 2 & 1 \\
 \hline
 $\theta_3$ & 2 & 3 & 2 \\
 \hline
 $\theta_4$ & 3 & 4 & 3 \\
 \hline
 $\theta_5$ & 1 & 2 & 1 \\  
 \hline
 $\theta_6$ & 1 & 2 & 2\\
\end{tabular}
\end{center}

Let us apply Proposition \ref{teoteo} to study the $\ell^{q,p}$ cohomology in degree 1 and 2.

The space of Rumin 1-forms is given by the following 1-forms
\begin{align*}
    E_0^1=span_{C^{\infty}(G)}\lbrace \theta_1,\theta_2,\theta_5,\theta_6\rbrace\,;
\end{align*}
whereas the space of Rumin 2-forms is given by
\begin{align*}
    E_0^2=span_{C^\infty(G}\lbrace \theta_5\wedge\theta_6-\theta_1\wedge\theta_3, \theta_1\wedge\theta_5,\theta_1\wedge\theta_6,\theta_2\wedge\theta_3,\theta_2\wedge\theta_5,\theta_2\wedge\theta_6\rbrace\,.
\end{align*}

Therefore, we get
\begin{center}
 \begin{tabular}{c|  c| c} 
  &  Weight from $\mathcal{V}_1$ & Weight from $\mathcal{V}_2$ \\ [0.5ex] 
 \hline\hline
 $\theta_1\wedge\theta_5$  & 3 & 2 \\ 
 \hline
 $\theta_1\wedge\theta_6$  & 3 & 3 \\
 \hline
 $\theta_2\wedge\theta_3$  & 5 & 3 \\
 \hline
 $\theta_2\wedge\theta_5$  & 4 & 2 \\
 \hline
 $\theta_2\wedge\theta_6$  & 4 & 3 \\  
 \hline
 $\theta_1\wedge\theta_3-\theta_5\wedge\theta_6$  & 4 & 3\\
\end{tabular}
\end{center}

Finally, considering the homogeneous dimensions $T_1=14$ and $T_2=10$, we have that using the first grading $\mathcal{V}_1$ we get
\begin{align*}
    \ell^{q,p}H^1(G)\neq 0&\text{ for }\frac{1}{p}-\frac{1}{q}<\frac{1}{14}\,,\\
    \ell^{q,p}H^2(G)\neq 0&\text{ for }\frac{1}{p}-\frac{1}{q}<\frac{1}{14}\,;
\end{align*}
whereas if we use the second grading $\mathcal{V}_2$ we get
\begin{align*}
    \ell^{q,p}H^1(G)\neq 0 &\text{ for }\frac{1}{p}-\frac{1}{q}<\frac{1}{10}\,,\\
    \ell^{q,p}H^2(G)\neq 0&\text{ for }\frac{1}{p}-\frac{1}{q}<0\,.
\end{align*}

Therefore, in degree 1 the second choice of grading $\mathcal{V}_2$ yields a better interval for non-vanishing cohomology, whereas in degree 2 the first grading $\mathcal{V}_1$ gives a better interval instead.
\end{ese}

\newpage

\appendix

\section{An alternative to asymptotic weights}\label{Alternative proof}

The scalar product on $\Lambda^1\mathfrak{g}^\ast$ extends canonically to smooth 1-forms $\Gamma(\Lambda^1\mathfrak{g}^\ast)$, so we can express the space of smooth 1-forms as a direct sum of subspaces using the filtration $\lbrace\Gamma(F_i)\rbrace$:
\begin{align*}
    \Gamma(\mathfrak{g}^\ast)=\Gamma(W_1)\oplus\Gamma(W_2)\oplus\cdots\oplus\Gamma(W_s)\,,
\end{align*}
where $\Gamma(W_i)=\Gamma(F_i)\cap(\Gamma(F_{i-1}))^\perp \;,i\ge 1$.

Let us stress that if $s$ is the nilpotency step, then $\Gamma(F_j)=\Gamma(F_{j+1})=\Gamma(\Lambda^1\mathfrak{g}^\ast)$ for any $j\ge s$, so that $\Gamma(W_j)=0$ for any $j>s$. Moreover, for $1\le j\le s$
\begin{align*}
    \Gamma(F_j)=\Gamma(W_1)\oplus \cdots\oplus\Gamma(W_j)\,.
\end{align*}

Let us now study the filtration $\lbrace \Gamma(\Lambda^2F_i)\rbrace$ on smooth 2-forms:
\begin{itemize}
    \item $\Gamma(\Lambda^2F_0)=0$;
    \item $\Gamma(\Lambda^2F_1)=\Gamma(\Lambda^2W_1)$;
    \item $\Gamma(\Lambda^2F_2)=\Gamma(\Lambda^2W_1)\oplus \Gamma(\Lambda^1W_1)\otimes\Gamma(\Lambda^1W_2)\oplus \Gamma(\Lambda^2W_2)=\Gamma(\Lambda^2F_1)\oplus \Gamma(\Lambda^2F_2)\cap(\Gamma(\Lambda^2F_1))^\perp$;
    \item $\Gamma(\Lambda^2F_3)=\Gamma(\Lambda^2W_1)\oplus \Gamma(\Lambda^1W_1)\otimes \Gamma(\Lambda^1W_2)\oplus\Gamma(\Lambda^2W_2)\oplus \Gamma(\Lambda^1W_1)\otimes \Gamma(\Lambda^1W_3)\oplus \Gamma(\Lambda^1W_2)\otimes \Gamma(\Lambda^1W_3)\oplus\Gamma(\Lambda^2W_3)=\Gamma(\Lambda^2F_1)\oplus\Gamma(\Lambda^2F_2)\cap(\Gamma(\Lambda^2F_1))^\perp\oplus\Gamma(\Lambda^2F_3\cap(\Gamma(\Lambda^2F_2))^\perp$;
    \item $\ldots$
    \item $\Gamma(\Lambda^2F_j)=\Gamma(\Lambda^2W_1)\oplus\Gamma(\Lambda^1W_1)\otimes\Gamma(\Lambda^1W_2)\oplus\Gamma(\Lambda^2W_2)\oplus\cdots\oplus\Gamma(\Lambda^1W_{j-2})\otimes\Gamma(\Lambda^1W_j)\oplus\Gamma(\Lambda^1W_{j-1})\otimes\Gamma(\Lambda^1W_j)\oplus\Gamma(\Lambda^2W_j)=\Gamma(\Lambda^2F_1)\oplus\Gamma(\Lambda^2F_2)\cap(\Gamma(\Lambda^2F_1))^\perp\oplus\cdots\oplus\Gamma(\Lambda^2F_j)\cap(\Gamma(\Lambda^2F_{j-1}))^\perp$;
    \item $\ldots$
    \item $\Gamma(\Lambda^2F_s)=\Gamma(\Lambda^2F_1)\oplus\Gamma(\Lambda^2F_2)\cap(\Gamma(\Lambda^2F_1))^\perp\oplus\cdots\Gamma(\Lambda^2F_s)\cap(\Gamma(\Lambda^2F_{s-1}))^\perp$.
\end{itemize}

One should notice that for $j>1$, we have
\begin{align*}
    \Gamma(\Lambda^2F_j)\cap(\Gamma(\Lambda^2F_{j-1}))^\perp=\bigoplus_{i=1}^{j-1}\Gamma(\Lambda^1W_i)\otimes\Gamma(\Lambda^1W_j)\oplus\Gamma(\Lambda^2W_j)\,,
\end{align*}
and for $k<j-1$
\begin{align*}
    \Gamma(\Lambda^2F_j)\cap(\Gamma(\Lambda^2F_{k}))^\perp=&\bigoplus_{i=0}^{j-k-1}\Gamma(\Lambda^2F_{j-i})\cap(\Gamma(\Lambda^2F_{j-i-1}))^\perp\\=&\bigoplus_{i=0}^{j-k-1}\bigoplus_{l=1}^{j-i-1}\Gamma(\Lambda^1W_l)\wedge\Gamma(\Lambda^1W_{j-i})\oplus\Gamma(\Lambda^2W_{j-i})\,.
\end{align*}

By repeating the same steps as in Definition \ref{defin inverse}, we can again construct the algebraic operator $d_\mathfrak{g}^{-1}$. Moreover, by definition of the filtration $\lbrace F_i\rbrace$, we have indeed that $d_\mathfrak{g}\Gamma(F_i)\subset\Gamma(\Lambda^2F_{i-1})$, hence
\begin{align*}
    d_\mathfrak{g}^{-1}(\Gamma(\Lambda^2F_i)\cap(\Gamma(\Lambda^2F_{i-1}))^\perp)\subset \Gamma(F_s)\cap(\Gamma(F_{i}))^\perp\,.
\end{align*}

Therefore we have the following inclusions:
\begin{itemize}
    \item $(d-d_\mathfrak{g})(\Gamma(W_k))\subset\Gamma(\Lambda^2F_s)\cap(\Gamma(\Lambda^2F_{k-1}))^\perp$;
    \item $d_\mathfrak{g}^{-1}(d-d_\mathfrak{g})(\Gamma(W_k))\subset d_\mathfrak{g}^{-1}\big(\Gamma(\Lambda^2F_s)\cap(\Gamma(\Lambda^2F_{k-1}))^\perp\big)\subset\Gamma(F_s)\cap(\Gamma(F_k))^\perp$;
     \item $(d-d_\mathfrak{g})d_\mathfrak{g}^{-1}(d-d_\mathfrak{g})(\Gamma(W_k)) \subset(d-d_\mathfrak{g})\big(\Gamma(F_s)\cap(\Gamma(F_k))^\perp\big)\subset\Gamma(\Lambda^2F_s)\cap(\Gamma(\Lambda^2F_{k}))^\perp$;
     \item $[d_\mathfrak{g}^{-1}(d-d_\mathfrak{g})]^2(\Gamma(W_k))\subset d_\mathfrak{g}^{-1}\big(\Gamma(\Lambda^2F_s)\cap(\Gamma(\Lambda^2F_{k}))^\perp\big)\subset\Gamma(F_s)\cap(\Gamma(F_{k+1}))^\perp$
\end{itemize}
and so on.

The same reasoning can be applied to arbitrary $k$-forms, so that the nilpotency of the differential operator $D$ defined in Corollary \ref{operator D} follows from the finiteness of the lower central series of $\mathfrak{g}$.

\begin{ese}\textbf{6-dimensional filiform group of the second kind}

Let us consider the 6-dimensional filiform nilpotent Lie group of the second kind whose Lie algebra is given by the following non-trivial Lie brackets:
\begin{align*}
    [X_1, X_2] = X_{3}\;,\;[X_1,X_3]=X_4\;,\;[X_1,X_4]=X_5\;,\;[X_2,X_5]=X_6\;,\;[X_3,X_4]=-X_6\,.
\end{align*}

The nilpotency step is $s=5$.

The action of $d_\mathfrak{g}$ on the left-invariant orthonormal basis $\lbrace \theta_i=X^\ast_i\rbrace_{1\le i\le 6}$ is the following:
\begin{itemize}
    \item $d_\mathfrak{g}\theta_1=d_\mathfrak{g}\theta_2=0$;
    \item $d_\mathfrak{g}\theta_3=-\theta_1\wedge\theta_2$;
    \item $d_\mathfrak{g}\theta_4=-\theta_1\wedge\theta_3$;
    \item $d_\mathfrak{g}\theta_5=-\theta_1\wedge\theta_4$;
    \item $d_\mathfrak{g}\theta_6=-\theta_2\wedge\theta_5+\theta_3\wedge\theta_4$.
\end{itemize}

Let us construct the filtration $\lbrace F_i\rbrace_{i=1}^5$ of 1-forms.
\begin{itemize}
    \item $F_0=0$;
    \item $F_1=span_{\mathbb{R}}\lbrace \theta_1,\theta_2 \rbrace$;
    \item  $F_2=span_{\mathbb{R}}\lbrace \theta_1,\theta_2,\theta_3\rbrace$;
    \item  $F_3=span_{\mathbb{R}}\lbrace \theta_1,\theta_2,\theta_3,\theta_4\rbrace$;
     \item  $F_4=span_{\mathbb{R}}\lbrace \theta_1,\theta_2,\theta_3,\theta_4,\theta_5\rbrace$;
     \item  $F_5=span_{\mathbb{R}}\lbrace \theta_1,\theta_2,\theta_3,\theta_4,\theta_5,\theta_6\rbrace$.
\end{itemize}

Then by definition we have
\begin{itemize}
    \item $W_1=F_1=span_{\mathbb{R}}\lbrace \theta_1,\theta_2 \rbrace$;
    \item $W_2=F_2\cap(F_1)^\perp=span_{\mathbb{R}}\lbrace \theta_3 \rbrace$;
    \item $W_3=F_3\cap(F_2)^\perp=span_{\mathbb{R}}\lbrace \theta_4 \rbrace$;
    \item $W_4=F_4\cap(F_3)^\perp=span_{\mathbb{R}}\lbrace \theta_5 \rbrace$;
    \item $W_5=F_5\cap(F_4)^\perp=span_{\mathbb{R}}\lbrace \theta_6 \rbrace$,
\end{itemize}
and
\begin{itemize}
    \item $\Lambda^2F_1=\Lambda^2W_1=span_{\mathbb{R}}\lbrace \theta_1\wedge \theta_2 \rbrace$;
    \item $\Lambda^2F_2=\Lambda^2W_1\oplus\Lambda^1W_1\otimes \Lambda^1W_2\oplus\Lambda^2W_2=span_{\mathbb{R}}\lbrace \theta_1\wedge \theta_2 ,\theta_1\wedge \theta_3, \theta_2\wedge \theta_3\rbrace$;
    \item $\Lambda^2F_3=\Lambda^2W_1\oplus\Lambda^1W_1\otimes \Lambda^1W_2\oplus\Lambda^2W_2\oplus\Lambda^1W_1\otimes \Lambda^1W_3\oplus\Lambda^1W_2\otimes \Lambda^1W_3\oplus\Lambda^2W_3=span_{\mathbb{R}}\lbrace \theta_1\wedge \theta_2 ,\theta_1\wedge \theta_3, \theta_2\wedge \theta_3, \theta_1\wedge \theta_4,\theta_2\wedge \theta_4, \theta_3\wedge \theta_4\rbrace$;
    \item  $\Lambda^2F_4=\Lambda^2W_1\oplus\Lambda^1W_1\otimes \Lambda^1W_2\oplus\Lambda^2W_2\oplus\Lambda^1W_1\otimes \Lambda^1W_3\oplus\Lambda^1W_2\otimes \Lambda^1W_3\oplus\Lambda^2W_3\oplus\Lambda^1W_1\otimes \Lambda^1W_4\oplus\Lambda^1W_2\otimes\Lambda^1W_4\oplus\Lambda^1W_3\otimes \Lambda^1W_4\oplus\Lambda^2W_4=span_{\mathbb{R}}\lbrace \theta_1\wedge \theta_2 ,\theta_1\wedge \theta_3, \theta_2\wedge \theta_3, \theta_1\wedge \theta_4,\theta_2\wedge \theta_4, \theta_3\wedge \theta_4,\theta_1\wedge \theta_5,\theta_2\wedge \theta_5,\theta_3\wedge \theta_5, \theta_4\wedge \theta_5\rbrace$;
    \item  $\Lambda^2F_5=\Lambda^2W_1\oplus\Lambda^1W_1\otimes \Lambda^1W_2\oplus\Lambda^2W_2\oplus\Lambda^1W_1\otimes \Lambda^1W_3\oplus\Lambda^1W_2\otimes \Lambda^1W_3\oplus\Lambda^2W_3\oplus\Lambda^1W_1\otimes \Lambda^1W_4\oplus\Lambda^1W_2\otimes\Lambda^1W_4\oplus\Lambda^1W_3\otimes \Lambda^1W_4\oplus\Lambda^2W_4\oplus\Lambda^1W_1\otimes \Lambda^1W_5\oplus\Lambda^1W_2\otimes\Lambda^1W_5\oplus\Lambda^1W_3\wedge \Lambda^1W_5\oplus\Lambda^1W_4\otimes\Lambda^1W_5\oplus\Lambda^2W_5=span_{\mathbb{R}}\lbrace \theta_1\wedge \theta_2 ,\theta_1\wedge \theta_3, \theta_2\wedge \theta_3, \theta_1\wedge \theta_4,\theta_2\wedge \theta_4, \theta_3\wedge \theta_4,\theta_1\wedge \theta_5,\theta_2\wedge \theta_5,\theta_3\wedge \theta_5, \theta_4\wedge \theta_5,\theta_1\wedge \theta_6,\theta_2\wedge \theta_6,\theta_3\wedge \theta_6,\theta_4\wedge \theta_6,\theta_5\wedge \theta_6\rbrace=\Lambda^2\mathfrak{g}^\ast$,
\end{itemize}
so that
\begin{itemize}
    \item $\Lambda^2F_1=\Lambda^2W_1=span_{\mathbb{R}}\lbrace \theta_1\wedge \theta_2 \rbrace$;
    \item $\Lambda^2F_2\cap(\Lambda^2F_1)^\perp=\Lambda^1W_1\otimes \Lambda^1W_2\oplus\Lambda^2W_2=span_{\mathbb{R}}\lbrace \theta_1\wedge \theta_3, \theta_2\wedge \theta_3\rbrace$ ($\Lambda^2W_2=0$);
    \item $\Lambda^2F_3\cap(\Lambda^2F_2)^\perp=\Lambda^1W_1\otimes \Lambda^1W_3\oplus\Lambda^1W_2\otimes \Lambda^1W_3\oplus\Lambda^2W_3=span_{\mathbb{R}}\lbrace  \theta_1\wedge \theta_4,\theta_2\wedge \theta_4\rbrace\oplus span_{\mathbb{R}}\lbrace \theta_3\wedge \theta_4\rbrace$ ($\Lambda^2W_3=0$);
    \item  $\Lambda^2F_4\cap(\Lambda^2F_3)^\perp=\Lambda^1W_1\otimes \Lambda^1W_4\oplus\Lambda^1W_2\otimes\Lambda^1W_4\oplus\Lambda^1W_3\otimes \Lambda^1W_4\oplus\Lambda^2W_4=span_{\mathbb{R}}\lbrace \theta_1\wedge \theta_5,\theta_2\wedge \theta_5\rbrace\oplus span_{\mathbb{R}}\lbrace \theta_3\wedge \theta_5\rbrace\oplus span_{\mathbb{R}}\lbrace \theta_4\wedge \theta_5\rbrace$ ($\Lambda^2W_4=0$);
    \item  $\Lambda^2F_5\cap(\Lambda^2F_4)^\perp=\Lambda^1W_1\otimes \Lambda^1W_5\oplus\Lambda^1W_2\otimes\Lambda^1W_5\oplus\Lambda^1W_3\otimes \Lambda^1W_5\oplus\Lambda^1W_4\otimes\Lambda^1W_5\oplus\Lambda^2W_5=span_{\mathbb{R}}\lbrace \theta_1\wedge \theta_6,\theta_2\wedge \theta_6\rbrace\oplus span_{\mathbb{R}}\lbrace \theta_3\wedge \theta_6 \rbrace\oplus span_{\mathbb{R}}\lbrace \theta_4\wedge \theta_6\rbrace\oplus span_{\mathbb{R}}\lbrace \theta_5\wedge \theta_6\rbrace$ ($\Lambda^2W_5=0$).
\end{itemize}

In particular, we have that:
\begin{itemize}
    \item $\theta_3\in W_2=F_2\cap(F_1)^\perp$ and $d_\mathfrak{g}\theta_3=-\theta_1\wedge\theta_2\in\Lambda^2W_1=\Lambda^2F_1$;
    \item $\theta_4\in W_3=F_3\cap(F_2)^\perp$ and $d_\mathfrak{g}\theta_4=-\theta_1\wedge\theta_3\in\Lambda^2 F_2\cap(\Lambda^2F_1)^\perp\subset\Lambda^2F_2$;
    \item $\theta_5\in W_4=F_4\cap(F_3)^\perp$ and $d_\mathfrak{g}\theta_5=-\theta_1\wedge\theta_4\in\Lambda^2F_3\cap(\Lambda^2F_2)^\perp\subset\Lambda^2F_3$;
    \item $\theta_6\in W_5=F_5\cap(F_4)^\perp$ and $d_\mathfrak{g}\theta_6=-\theta_2\wedge\theta_5+\theta_3\wedge\theta_4\subset\Lambda^2F_4\cap(\Lambda^2F_3)^\perp\oplus\Lambda^2F_3\cap(\Lambda^2F_2)^\perp$\,.
\end{itemize}

Given a Rumin 1-form $\alpha\in E_0^1=Ker\,d_\mathfrak{g}\cap(Im\,d_\mathfrak{g})^\perp$, $\alpha=f_1\theta_1+f_2\theta_2\in\Gamma(W_1)$, we have:
\begin{align*}
    d\alpha=&(d-d_\mathfrak{g})\alpha=-X_2f_1\theta_1\wedge\theta_2-X_3f_1\theta_1\wedge\theta_3-X_4f_1\theta_1\wedge\theta_4-X_5f_1\theta_1\wedge\theta_5-X_6f_1\theta_1\wedge\theta_6+\\&+X_1f_2\theta_1\wedge\theta_2-X_3f_2\theta_2\wedge\theta_3-X_4f_2\theta_2\wedge\theta_4-X_5f_2\theta_2\wedge\theta_5-X_6f_2\theta_2\wedge\theta_6\\=&\underbrace{(X_1f_2-X_2f_1)\theta_1\wedge\theta_2}_{\in\Gamma(\Lambda^2F_1)}-\underbrace{X_3f_1\theta_1\wedge\theta_3-X_3f_2\theta_2\wedge\theta_3}_{\in\Gamma(\Lambda^2F_2)\cap(\Gamma(\Lambda^2F_1))^\perp}-\underbrace{X_4f_1\theta_1\wedge\theta_4-X_4f_2\theta_2\wedge\theta_4}_{\in\Gamma(\Lambda^2F_3)\cap(\Gamma(\Lambda^2F_2))^\perp}+\\&-\underbrace{X_5f_1\theta_1\wedge\theta_5-X_5f_2\theta_2\wedge\theta_5}_{\in\Gamma(\Lambda^2F_4)\cap(\Gamma(\Lambda^2F_3))^\perp}-\underbrace{X_6f_1\theta_1\wedge\theta_6-X_6f_2\theta_2\wedge\theta_6}_{\in\Gamma(\Lambda^2F_5)\cap(\Gamma(\Lambda^2F_4))^\perp}\in\Gamma(\Lambda^2F_5)\cap(\Gamma(\Lambda^2F_0))^\perp\,,
\end{align*}
so that
\begin{align*}
    d_\mathfrak{g}^{-1}d\alpha=&-\underbrace{(X_1f_2-X_2f_1)\theta_3}_{\in\Gamma(F_2)\cap(\Gamma(F_1))^\perp}+\underbrace{X_3f_1\theta_4}_{\in\Gamma(F_3)\cap(\Gamma(F_2))^\perp}+\underbrace{X_4f_1\theta_5}_{\in\Gamma(F_4)\cap(\Gamma(F_3))^\perp}+\underbrace{\frac{X_5f_2}{2}\theta_6}_{\in\Gamma(F_5)\cap(\Gamma(F_4))^\perp}\in\Gamma(F_5)\cap(\Gamma(F_1))^\perp\,,
\end{align*}
Furthermore
\begin{align*}
    (d-&d_\mathfrak{g})d_\mathfrak{g}^{-1}d\alpha=-X_1(X_1f_2-X_2f_1)\theta_1\wedge\theta_3-X_2(X_1f_2-X_2f_1)\theta_2\wedge\theta_3+X_4(X_1f_2-X_2f_1)\theta_3\wedge\theta_4+\\&+X_5(X_1f_2-X_2f_1)\theta_3\wedge\theta_5+X_6(X_1f_2-X_2f_1)\theta_3\wedge\theta_6+X_1X_3f_1\theta_1\wedge\theta_4+X_2X_3f_1\theta_2\wedge\theta_4+\\&+X_3^2f_1\theta_3\wedge\theta_4-X_5X_3f_1\theta_4\wedge\theta_5-X_6X_3f_1\theta_4\wedge\theta_6+X_1X_4f_1\theta_1\wedge\theta_5+X_2X_4f_1\theta_2\wedge\theta_5+\\&+X_3X_4f_1\theta_3\wedge\theta_5+X_4^2f_1\theta_4\wedge\theta_5-X_6X_4f_1\theta_5\wedge\theta_6+\frac{X_1X_5f_2}{2}\theta_1\wedge\theta_6+\frac{X_2X_5f_2}{2}\theta_2\wedge\theta_6+\\&+\frac{X_3X_5f_2}{2}\theta_3\wedge\theta_6+\frac{X_4X_5f_2}{2}\theta_4\wedge\theta_6+\frac{X_5^2f_2}{2}\theta_5\wedge\theta_6\\=&\underbrace{-X_1(X_1f_2-X_2f_1)\theta_1\wedge\theta_3-X_2(X_1f_2-X_2f_1)\theta_2\wedge\theta_3}_{\in\Gamma(\Lambda^2F_2)\cap(\Gamma(\Lambda^2F_1))^\perp}+\\
    &+\underbrace{X_1X_3f_1\theta_1\wedge\theta_4+X_2X_3f_1\theta_2\wedge\theta_4+[X_4(X_1f_2-X_2f_1)+X_3^2f_1]\theta_3\wedge\theta_4}_{\in\Gamma(\Lambda^2F_3)\cap(\Gamma(\Lambda^2F_2))^\perp}+\\
    &\underbrace{+X_1X_4f_1\theta_1\wedge\theta_5+X_2X_4f_1\theta_2\wedge\theta_5+[X_5(X_1f_2-X_2f_1)+X_3X_4f_1]\theta_3\wedge\theta_5+(X_4^2f_1-X_5X_3f_1)\theta_4\wedge\theta_5}_{\in\Gamma(\Lambda^2F_4)\cap(\Gamma(\Lambda^2F_3))^\perp}+\\
    &+\frac{X_1X_5f_2}{2}\theta_1\wedge\theta_6+\frac{X_2X_5f_2}{2}\theta_2\wedge\theta_6+[\frac{X_3X_5f_2}{2}+X_6(X_1f_2-X_2f_1)]\theta_3\wedge\theta_6+\\&+[\frac{X_4X_5f_2}{2}-X_6X_3f_1]\theta_4\wedge\theta_6+[\frac{X_5^2f_2}{2}-X_6X_4f_1]\theta_5\wedge\theta_6\in\Gamma(\Lambda^2F_5)\cap(\Gamma(\Lambda^2F_1))^\perp
\end{align*}
and
\begin{align*}
    d_\mathfrak{g}^{-1}(d-d_\mathfrak{g})d_\mathfrak{g}^{-1}d\alpha=&\underbrace{X_1(X_1f_2-X_2f_1)\theta_4}_{\in\Gamma(F_3)\cap(\Gamma(F_2))^\perp}-\underbrace{X_1X_3f_1\theta_5}_{\in\Gamma(F_4)\cap(\Gamma(F_3))^\perp}+\underbrace{\frac{X_4(X_1f_2-X_2f_1)+X_3^2f_1}{2}\theta_6}_{\in\Gamma(F_5)\cap(\Gamma(F_4))^\perp}+\\&-\underbrace{\frac{X_2X_4f_1}{2}\theta_6}_{\in\Gamma(F_5)\cap(\Gamma(F_4))^\perp}\in\Gamma(F_5)\cap(\Gamma(F_2))^\perp
\end{align*}
Then
\begin{align*}
    [d_\mathfrak{g}^{-1}(d-d_\mathfrak{g})]^2d_\mathfrak{g}^{-1}d\alpha=&d_\mathfrak{g}^{-1}\big[X_1^2(X_1f_2-X_2f_1)\theta_1\wedge\theta_4+X_3X_1(X_1f_2-X_2f_1)\theta_3\wedge\theta_4+\\&-X_2X_1X_3f_1\theta_2\wedge\theta_5+(Im\,d_\mathfrak{g})^\perp\big]\\=-X_1^2(X_1&f_2-X_2f_1)\theta_5+\frac{X_3X_1(X_1f_2-X_2f_1)+X_2X_1X_3f_1}{2}\theta_6\in\Gamma(F_5)\cap(\Gamma(F_3))^\perp\,,\\
    [d_\mathfrak{g}^{-1}(d-d_\mathfrak{g})]^3d_\mathfrak{g}^{-1}d\alpha=&d_\mathfrak{g}^{-1}\big[-X_2X_1^2(X_1f_2-X_2f_1)\theta_2\wedge\theta_5+(Im\,d_\mathfrak{g})^\perp\big]\\=&\frac{X_2X_1^2(X_1f_2-X_2f_1)}{2}\theta_6\in\Gamma(F_5)\cap(\Gamma(F_4))^\perp\,,
\end{align*}
and finally
\begin{align*}
    [d_\mathfrak{g}^{-1}(d-d_\mathfrak{g})]^4d_\mathfrak{g}^{-1}d\alpha\in\Gamma(F_5)\cap(\Gamma(F_5))^\perp=0\;\Rightarrow \;[d_\mathfrak{g}^{-1}(d-d_\mathfrak{g})]^4d_\mathfrak{g}^{-1}d\alpha=0\,.
\end{align*}

\end{ese}

\section*{Acknowledgements}

The author is supported by the University of Bologna, funds for selected research topics, and by the European Union's Horizon 2020 research and innovation programme under the Marie Sk\l{}odowska-Curie grant agreement No 777822 GHAIA (`\emph{Geometric and Harmonic Analysis with Interdisciplinary Applications}').

The author would like to thank Prof. Pierre Pansu and Prof. Bruno Franchi for their support during the preparation
of the paper.

\bibliographystyle{plain}
\bibliography{biblio.bib}

\bigskip
\vspace{0.5cm}

\tiny{
\par\noindent
Francesca Tripaldi 
\par\noindent Department of Mathematics,
\par\noindent Bologna University,
\par\noindent 40126, Bologna, Italy.
\par\noindent email: francesca.tripaldi2@unibo.it
}

\end{document}